\documentclass[11pt,a4paper]{amsart}

\usepackage{amsmath,amssymb,amsthm,mathtools,mathrsfs}
\usepackage{scalerel,stackengine}
\stackMath
\newcommand{\reallywidehat}[1]{%
  \savestack{\tmpbox}{\stretchto{%
    \scaleto{%
      \scalerel*[\widthof{\ensuremath{#1}}]{\kern-.6pt\bigwedge\kern-.6pt}%
        {\rule[-\textheight/2]{1ex}{\textheight}}%
    }{\textheight}%
  }{0.5ex}}%
  \stackon[1pt]{#1}{\tmpbox}%
}
\usepackage[margin=1.in]{geometry}
\usepackage{graphicx}
\graphicspath{{figures/}}
\usepackage[colorlinks=true,linkcolor=blue,citecolor=blue,urlcolor=blue]{hyperref}
\usepackage{cleveref}

\newtheorem{theorem}{Theorem}[section]
\newtheorem{proposition}[theorem]{Proposition}
\newtheorem{lemma}[theorem]{Lemma}

\theoremstyle{definition}
\newtheorem{definition}[theorem]{Definition}
\newtheorem{remark}[theorem]{Remark}

\newcommand{\R}{\mathbb{R}}

\newcommand{\Es}{\mathcal{E}}
\newcommand{\Phis}{\Phi_{s}}
\newcommand{\Lop}{\mathcal{L}_{s}}
\newcommand{\cs}{c_{1,s}}

\newcommand{\fl}{(-\Delta)^{s}}
\newcommand{\Lodd}{L^{2}_{\mathrm{odd}}}
\DeclareMathOperator{\sgn}{sgn}

\title[Energy of fractional Allen--Cahn layers in dimension one]{The energy of fractional Allen--Cahn layers in dimension one}

\author{Alvaro Carballeira}
\address{Department of Mathematics, Brown University, 151 Thayer Street, Providence, RI 02912, USA}
\email{alvaro\_carballeira@brown.edu}

\author{Damien Galant}
\address{Department of Mathematics, Brown University, 151 Thayer Street, Providence, RI 02912, USA}
\email{damien\_galant@brown.edu}

\author{Javier G\'omez-Serrano}
\address{Department of Mathematics, Brown University, 151 Thayer Street, Providence, RI 02912, USA}
\email{javier\_gomez\_serrano@brown.edu}

\date{\today}

\begin{document}

\begin{abstract}
	We study the energy $\Es(s) := E_{s}[\Phi_s]$ of the one-dimensional fractional Allen--Cahn layer solution $\Phis$, defined as the unique odd, increasing solution of $(-\Delta)^{s} \Phis = \Phis - \Phis^{3}$ with $\Phis(\pm\infty)=\pm 1$ and $\Phis(0)=0$, for $s\in(1/2,1]$. Our main results are a sharp qualitative and quantitative description of the energy $\Es$ on this interval. We show that the energy is continuous and strictly decreasing, and we obtain explicit asymptotic expansions at both endpoints. At the upper endpoint we prove $\Es(s) = \frac{2\sqrt{2}}{3} + \kappa_1 (1-s) + o(1-s)$ with an explicit formula for $\kappa_1$. At the lower endpoint we prove that the energy goes to infinity as $\Es(s)=\frac{1}{\pi(s-1/2)}+O(1)$. The strict decrease is proved with computer assistance. On an interior subinterval it is reduced to finitely many inequalities verified by interval-arithmetic computations. The proofs combine the Cabré--Sire construction of the layer, the minimality theorem of Palatucci--Savin--Valdinoci, an identity for the $s$ derivative of the energy, and a computer-assisted coercivity estimate at explicit approximate layers.
\end{abstract}

\maketitle

\tableofcontents

\section{Introduction}\label{sec:intro}

The Allen--Cahn equation
\begin{equation*}
	-\Delta u = u - u^{3}\qquad\text{on }\R^{n},
\end{equation*}
introduced in~\cite{AllenCahn1979} in its time-dependent form to model the motion of antiphase boundaries in crystalline solids, is the Euler--Lagrange equation of the energy functional
\begin{equation*}
	\mathcal{F}[u] = \int_{\R^{n}}\left(\frac{1}{2}|\nabla u|^{2}+W(u)\right)\,dx
\end{equation*}
associated with the double-well potential $W(t)=(1/4)(1-t^{2})^{2}$. Its sharp-interface limit, identified rigorously through the celebrated work of Modica--Mortola and Modica~\cite{Modica1987,ModicaMortola1977} via $\Gamma$-convergence, reproduces the perimeter functional and provides the classical bridge between the diffuse and sharp formulations of phase transitions.

The fractional analog
\begin{equation}\label{eq:fAC}
	\fl u = u - u^{3}\qquad\text{on }\R^{n},\quad s\in(0,1),
\end{equation}
where $\fl$ denotes the fractional Laplacian, encodes long-range interactions between the phases and arises as the Euler--Lagrange equation of the energy
\begin{equation*}
	E_{s}[u]
	= \frac{\cs}{4}\iint_{\R^{n}\times\R^{n}}\frac{(u(x)-u(y))^{2}}{|x-y|^{n+2s}}\,dx\,dy
	\;+\;\int_{\R^{n}}W(u)\,dx.
\end{equation*}
The normalizing constant $\cs$ is fixed in~\eqref{eq:cs-def} below. The analytical study of~\eqref{eq:fAC} was initiated in the works of Caffarelli--Souganidis~\cite{CaffarelliSouganidis2010}, Sire--Valdinoci~\cite{SireValdinoci2009}, Cabré--Sire~\cite{CabreSire2014I,CabreSire2015II}, Cabré--Sol\`a-Morales~\cite{CabreSolaMorales2005}, Palatucci--Savin--Valdinoci~\cite{PalatucciSavinValdinoci2013}, and Savin--Valdinoci~\cite{SavinValdinoci2012,SavinValdinoci2014}, among others. Three distinct regimes emerge according to the value of~$s$: for $s>1/2$ the leading-order $\Gamma$-limit is again the classical perimeter (after the appropriate renormalization of the energy), for $s=1/2$ the limiting functional is a fractional perimeter, and for $s<1/2$ the relevant scaling is different and the limit is given by an $s$-perimeter. The threshold $s=1/2$ appears as a critical exponent throughout the analysis.

The object of this paper is the energy of a single interface in dimension one in \eqref{eq:fAC}. By the symmetry of the problem this is captured by the unique odd, monotone increasing solution $ \Phis:\R\to(-1,1), $ called the \emph{layer}, that satisfies
\begin{equation*}
	\fl \Phis = \Phis-\Phis^{3}\quad\text{on }\R,
	\qquad \Phis(\pm\infty)=\pm 1,\qquad \Phis(0)=0.
\end{equation*}
Its corresponding energy will be denoted by
\begin{equation*}
	\Es(s):=E_{s}[\Phis],\qquad s\in(1/2,1].
\end{equation*}
Existence, uniqueness, monotonicity, and a sharp algebraic decay rate for the layer were established by Cabré and Sire~\cite{CabreSire2015II}. The precise statements we use are recalled in Section~\ref{ss:prelim}. The case of the endpoint $s=1$ is exactly solvable. The layer is $ \Phi_{1}(x)=\tanh(x/\sqrt 2), $ and an elementary calculation yields $ \Es(1)=2\sqrt 2/3. $

\subsection{Main results}\label{ss:main}
For each $s\in(1/2,1]$, $\Es(s)$ is the energy of the unique odd, monotone increasing layer $\Phi_s$. The map $s\mapsto\Es(s)$ therefore records how the energy of this one-dimensional transition profile depends on the fractional exponent. It records the energy cost of a single phase transition as the order of nonlocality varies. Despite a substantial body of work on related qualitative questions, sharp quantitative information on~$\Es$ is, to our knowledge, absent from the literature beyond the endpoint value $\Es(1)$. Our two main results are the following.

\begin{theorem}[Asymptotic regimes]\label{thm:asymp}
	The function $\Es$ admits the following asymptotic expansions.

	\smallskip\noindent (i) Expansion at $s=1$: There exists a constant $\kappa_{1}\in\R$ such that
	\begin{equation}\label{eq:asymp-s1}
		\Es(s) \;=\; \frac{2\sqrt{2}}{3}\;+\;\kappa_{1}\,(1-s)\;+\;o(1-s)
		\qquad\text{as }s\to 1^{-}.
	\end{equation}
	The constant $\kappa_{1}$ is given by the explicit integral
	\begin{equation*}
		\kappa_{1}:=-2\pi\int_{0}^{\infty}
		\frac{\xi^{2}\log\xi}{\sinh^{2}(\pi\xi/\sqrt2)}\,d\xi
		\;\approx\;1.073.
	\end{equation*}

	\smallskip\noindent (ii) Pole at $s=1/2$: The function $\Es$ satisfies
	\begin{equation}\label{eq:asymp-s12-sharp}
		\Es(s) \;=\; \frac{1/\pi}{\,s-1/2\,}\;+\;O(1)\qquad\text{as }s\to (1/2)^{+}.
	\end{equation}
	In particular, no logarithmic intermediate term arises between the pole and the bounded remainder.
\end{theorem}

\begin{theorem}[Existence, continuity, and strict monotonicity]\label{thm:mono}
	The function $\Es:(1/2,1]\to\R$ is well-defined, continuous, and strictly decreasing. Combined with Theorem~\ref{thm:asymp}, it extends to a continuous strictly decreasing homeomorphism
	\begin{equation}\label{eq:homeo}
		\Es: [1/2,1]\;\longrightarrow\;\Bigl[\tfrac{2\sqrt{2}}{3},\;+\infty\Bigr].
	\end{equation}
\end{theorem}

\begin{figure}[b]
	\centering
	\includegraphics[width=0.78\textwidth]{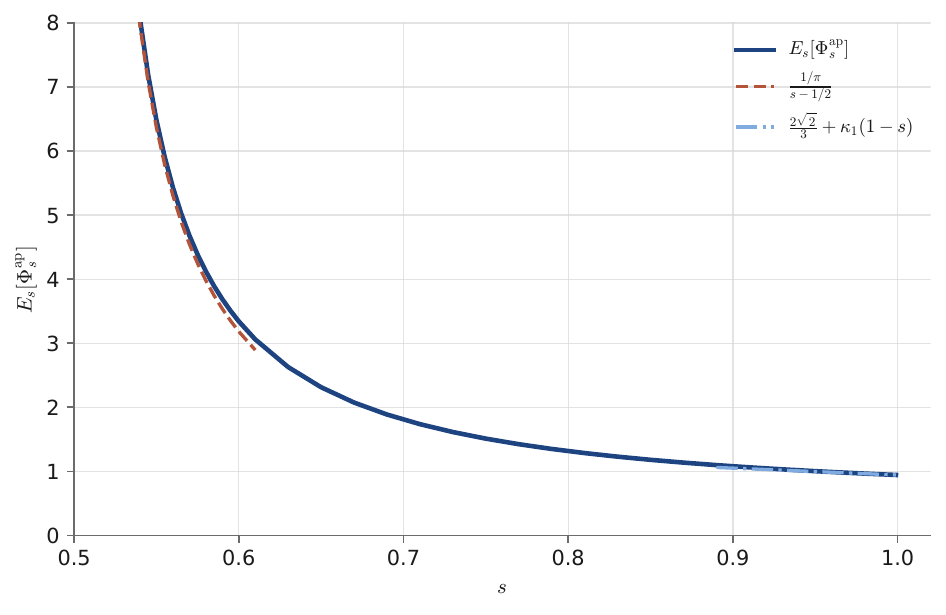}
	\caption{The energy $\Es(s)$ on $(1/2,1]$, together with the two endpoint
		regimes of Theorem~\ref{thm:asymp}. The solid
		curve is $E_{s}[\Phi^{\mathrm{ap}}_{s}]$, the energy of the explicit
		approximate layer of Appendix~\ref{app:layer}.
		It is displayed for illustration and is not part of any proof.}
	\label{fig:energy}
\end{figure}

The strict monotonicity in Theorem~\ref{thm:mono} is, in our view, the most delicate result of the paper. Continuity follows directly by comparing minimizing layers at two different exponents, whereas strict decrease rests on an identity expressing the derivative $\Es'(s)$ as a spectral integral, whose sign is controlled near the boundary by the endpoint expansions of Theorem~\ref{thm:asymp} and on an interior interval by computer-assisted techniques. The argument is outlined in Section~\ref{ss:mono-strategy}. Figure~\ref{fig:energy} shows the plot of the curve $\Es(s)$.

\subsection{Preliminaries}\label{ss:prelim}

Throughout the paper we work in dimension $n=1$ and with the parameter $s\in(1/2,1]$ unless otherwise stated. We write $C$ for a positive constant that may change from line to line but is independent of any quantity declared variable in its context.

We use the Fourier transform and inverse transform
\begin{equation*}
	\widehat u(\xi)=\int_{\R} u(x)\,e^{-i\xi x}\,dx,
	\qquad u(x)=\frac{1}{2\pi}\int_{\R}\widehat u(\xi)\,e^{i\xi x}\,d\xi.
\end{equation*}
The fractional Laplacian of order $s\in(0,1)$ is the Fourier multiplier
\begin{equation}\label{eq:fl-fourier}
	\widehat{\fl u}(\xi) = |\xi|^{2s}\,\widehat u(\xi),
\end{equation}
defined initially on the Schwartz class and extended by duality to tempered distributions for which the right-hand side makes sense. For $s\in(0,1)$ the operator~\eqref{eq:fl-fourier} admits the equivalent singular-integral representation
\begin{equation}\label{eq:fl-singular}
	\fl u(x) = \cs\,\mathrm{P.V.}\int_{\R}\frac{u(x)-u(y)}{|x-y|^{1+2s}}\,dy,
\end{equation}
where the constant $\cs$ is fixed by the requirement that the two definitions~\eqref{eq:fl-fourier} and~\eqref{eq:fl-singular} agree. In dimension one this gives
\begin{equation}\label{eq:cs-def}
	\cs \;=\; \frac{4^{s}\,s\,\Gamma(s+\tfrac{1}{2})}{\sqrt{\pi}\,\Gamma(1-s)},
\end{equation}
see, e.g.,~\cite{DiNezzaPalatucciValdinoci2012,Stein1970}. With~\eqref{eq:cs-def} in force, the Gagliardo seminorm satisfies the Plancherel identity
\begin{equation}\label{eq:plancherel}
	\frac{\cs}{2}\iint_{\R\times\R}\frac{(u(x)-u(y))^{2}}{|x-y|^{1+2s}}\,dx\,dy
	\;=\;\frac{1}{2\pi}\int_{\R}|\xi|^{2s}|\widehat u(\xi)|^{2}\,d\xi.
\end{equation}

\begin{lemma}\label{lem:sobolev-constants}
	For $\sigma>1/2$, the embedding $H^{\sigma}(\R)\hookrightarrow L^{\infty}(\R)$ satisfies
	\begin{equation*}
		\|u\|_{L^{\infty}(\R)}\le C^{\ast}(\sigma)\|u\|_{H^{\sigma}(\R)},
		\qquad
		C^{\ast}(\sigma)^{2}
		=\frac{\Gamma(\sigma-\tfrac12)}{2\sqrt{\pi}\,\Gamma(\sigma)}.
	\end{equation*}
\end{lemma}

\begin{proof}
	By Fourier inversion and Cauchy--Schwarz,
	\begin{equation*}
		|u(x)|\le \frac{1}{2\pi}\|\widehat u\|_{L^{1}}
		\le \frac{1}{2\pi}
		\biggl(\int_{\R}(1+|\xi|^{2})^{-\sigma}\,d\xi\biggr)^{1/2}
		\biggl(\int_{\R}(1+|\xi|^{2})^{\sigma}|\widehat u(\xi)|^{2}\,d\xi\biggr)^{1/2}.
	\end{equation*}
	By Plancherel, the last factor is $(2\pi)^{1/2}\|u\|_{H^{\sigma}}$. The beta integral gives
	\begin{equation*}
		\int_{\R}(1+|\xi|^{2})^{-\sigma}\,d\xi
		=\sqrt{\pi}\,\frac{\Gamma(\sigma-\tfrac12)}{\Gamma(\sigma)},
	\end{equation*}
	which yields the stated constant.
\end{proof}

We work with the double-well potential $W(t) = (1/4) (1-t^2)^{2}$, whose two non-degenerate minima are $t=\pm 1$. The fractional Allen--Cahn energy in dimension one is
\begin{equation}\label{eq:Es}
	E_{s}[u]
	:=\frac{\cs}{4}\iint_{\R\times\R}\frac{(u(x)-u(y))^{2}}{|x-y|^{1+2s}}\,dx\,dy
	\;+\;\int_{\R}W(u(x))\,dx,
\end{equation}
defined for measurable $u:\R\to\R$ for which both terms are finite. At the endpoint $s=1$ we use the classical energy
\begin{equation}\label{eq:E1}
	E_{1}[u]:=\frac12\int_{\R}|u'(x)|^{2}\,dx+\int_{\R}W(u(x))\,dx,
\end{equation}
which is the limit of~\eqref{eq:Es} on profiles with derivative in $L^{2}(\R)$.

\begin{definition}\label{def:admissible}
	For each $s\in(1/2,1]$ the admissible class $\mathcal A$ consists of all measurable functions $u:\R\to\R$ such that $|u|\leq 1$ almost everywhere, $E_{s}[u]<+\infty$, and the continuous representative of~$u$ satisfies
	\begin{equation*}
		\lim_{x\to-\infty}u(x)=-1,
		\qquad
		\lim_{x\to+\infty}u(x)=1.
	\end{equation*}
	The continuous representative exists because finite energy and $s>1/2$ imply $u\in H^{s}_{\mathrm{loc}}(\R)\hookrightarrow C^{0}(\R)$.
\end{definition}

The class $\mathcal A$ is independent of~$s$ at the level of the boundary conditions, and the dependence on~$s$ enters only through the integrability requirement. For $s\leq 1/2$ the constraint $u(\pm\infty)=\pm 1$ is incompatible with finiteness of the Gagliardo seminorm in~\eqref{eq:Es}, so $\mathcal A=\emptyset$ in that range. This is the analytic reason for restricting our discussion to $s\in(1/2,1]$.

We state now the existence, uniqueness, and decay properties of the layer.
\begin{theorem}[Cabré--Sire~\cite{CabreSire2015II}]\label{thm:CS}
	Let $s\in(0,1)$. There exists a unique function $\Phis\in C^{\infty}(\R)$ such that
	\begin{equation*}
		\fl \Phis = \Phis - \Phis^{3}\quad\text{on }\R,\qquad \Phis(\pm\infty)=\pm 1,\qquad \Phis(0)=0.
	\end{equation*}
	The layer~$\Phis$ is odd and strictly increasing on~$\R$ with $|\Phis|<1$. Moreover, there exist constants $0<c\leq C<\infty$, depending on~$s$, such that
	\begin{equation}\label{eq:CS-tail}
		c|x|^{-2s}\;\leq\;1-\Phis^{2}(x)\;\leq\;C|x|^{-2s},
		\qquad |x|>1.
	\end{equation}
\end{theorem}

The existence, uniqueness and monotonicity statements of Theorem~\ref{thm:CS} are Theorems~2.4 and~2.9 of~\cite{CabreSire2015II}. The decay estimate~\eqref{eq:CS-tail} follows from Theorem~2.7 in the same paper, since $1-\Phi_s^2=(1-\Phi_s)(1+\Phi_s)$ and each of these factors is bounded above and below.

\begin{theorem}[Palatucci--Savin--Valdinoci~\cite{PalatucciSavinValdinoci2013}]\label{thm:PSV}
	Let $s\in(1/2,1)$ and let $\Phis$ be the Cabré--Sire layer from Theorem~\ref{thm:CS}. Then for every $u\in\mathcal A$,
	\begin{equation}\label{eq:PSV}
		E_{s}[\Phis]\;\leq\;E_{s}[u].
	\end{equation}
\end{theorem}

The case $s=1$ of~\eqref{eq:PSV} is the classical minimality of the $\tanh$ profile and goes back to~\cite{Modica1987}. For $s\in(1/2,1)$, apply Theorem~2 and Remark~2 of~\cite{PalatucciSavinValdinoci2013} with the rescaled potential $ \widetilde W_{s}:=(2/\cs)W. $ The corresponding functional in that paper is
\begin{equation*}
	\frac12\iint_{\R\times\R}\frac{(u(x)-u(y))^{2}}{|x-y|^{1+2s}}\,dx\,dy
	+\int_{\R}\widetilde W_{s}(u)\,dx
	=\frac{2}{\cs}E_{s}[u].
\end{equation*}
Remark~2 shows that its infimum over functions with limits $-1$ and $+1$ at the two ends is attained by a nondecreasing function, and Theorem~2 gives uniqueness up to translation in the monotone class. After centering at the zero level, the minimizer is $\Phis$ by Theorem~\ref{thm:CS}, which yields~\eqref{eq:PSV} for the admissible class of Definition~\ref{def:admissible}.

Let us explain why minimality does not yield a soft proof of strict monotonicity. If $r<s$, using $\Phi_{r}$ and $\Phi_{s}$ as competitors in~\eqref{eq:PSV} gives inequalities~\eqref{eq:sandwich-lem}. For a fixed profile $u$, however, the Fourier representation~\eqref{eq:fixed-layer-fourier-energy} of the energy contains the multiplier $|\xi|^{2s-2}$, which decreases with $s$ on $\{|\xi|<1\}$ and increases with $s$ on $\{|\xi|>1\}$. Neither outer difference in~\eqref{eq:sandwich-lem} therefore has a prescribed sign. At the infinitesimal level, Lemma~\ref{lem:envelope-paper} makes the same obstruction explicit. We have $\Es'(s)=J(s,\Phis)$, and the multiplier $\log|\xi|\,|\xi|^{2s-2}$ in $J(s, \cdot)$ changes sign at $|\xi|=1$. Thus the standard variational and PDE estimates give continuity and differentiability, but do not decide the sign of the derivative. Near the endpoints we can carry out asymptotic expansions, whereas on the interior interval there is no small parameter and $\Phis$ is not known in closed form.

\subsection{Strategy of the proof}\label{ss:mono-strategy}

The proof of the strict monotone decrease asserted in Theorem~\ref{thm:mono} is organized around the cover
\begin{equation}\label{eq:three-interval-cover}
	(1/2,1]
	=\bigl(1/2,53/100\bigr]
	\cup[53/100,49/50]
	\cup[49/50,1].
\end{equation}
The strict decrease on the right endpoint interval is proved in Proposition~\ref{prop:mono-upper} in Section~\ref{sec:s1}, and that on the left endpoint interval in Proposition~\ref{prop:mono-lower} in Section~\ref{sec:s12}. Both arguments consist of asymptotic expansions. Section~\ref{sec:mono-interior} treats the middle interval $I_{0}=[0.530,0.980]$. The endpoint widths $1/50$ and $3/100$ are deliberately not optimized. They are chosen to meet exactly the endpoints of~$I_{0}$ but the endpoint estimates used are generally not sharp.

The derivative functional in these arguments is
\begin{equation}\label{eq:J-functional}
	J(s,u):=\frac{1}{2\pi}\int_{\R}\log|\xi|\,|\xi|^{2s-2}
	|\widehat{u'}(\xi)|^{2}\,d\xi,
\end{equation}
defined for $s\in(1/2,1]$ and odd $u$ for which the integral is finite. As we prove in Section~\ref{sec:cont} (Lemma~\ref{lem:envelope-paper}), the function $\Es$ belongs to $C^{1}((1/2,1))$ with
\begin{equation*}
	\Es'(s)\;=\;J(s,\Phis)\qquad(s\in(1/2,1)).
\end{equation*}
Strict decrease is therefore equivalent to the strict negativity of $J(s,\Phis)$, which we establish separately on the three intervals of the cover~\eqref{eq:three-interval-cover}.

Near $s=1$ we linearize the layer equation in the parameter $s$ around the profile $\Phi_{1}=\tanh(\cdot/\sqrt2)$. A single corrector, obtained by inverting the linearized operator on odd functions, produces a profile that solves the layer equation up to $O((1-s)^{2})$. A contraction argument places the true layer within $O((1-s)^{2})$ of it in $H^{2s}$, and the expansion~\eqref{eq:asymp-s1} is derived from $J$. The constant $\kappa_{1}$ emerges as an explicit one-dimensional integral of $\Phi_{1}'$ against a logarithmic weight. Finally, carrying the expansion to second order, with explicit remainders, yields the strict sign of $\Es'$ for $s \in [49/50,1]$. Throughout this analysis, we track explicit constants in the estimates to control the errors in the asymptotic expansion and thereby prove strict monotonicity on that interval.

Near $s=1/2$ the asymptotics rest on two ingredients. The one-sided minimality of Theorem~\ref{thm:PSV} gives a matching upper bound on $\Es(s)$ through a carefully chosen competitor, while a Fourier-side analysis of the Gagliardo seminorm of a competitor with controlled tails produces a residual term of order $(s-1/2)^{-1}$. The absence of a logarithmic correction in~\eqref{eq:asymp-s12-sharp} follows from a refined comparison between the competitor and the true layer using the sharp tail bound~\eqref{eq:CS-tail} of Cabré--Sire. Finally, the same expansion, differentiated, yields the strict sign of $\Es'$ in the interval $s \in (1/2, 53/100]$. Again, explicit constants need to be tracked to make sure the asymptotic expansions are valid throughout this interval.

Away from the endpoints $s=1/2$ and $s=1$, neither endpoint expansion is applicable, and the layer $\Phi_s$ is not available explicitly. We therefore construct a numerical approximate layer and complete the monotonicity proof using a computer-assisted argument. This is why the asymptotic analysis carried out at the endpoints must cover explicit intervals.

\subsection{Computer-assisted techniques}\label{ss:cap}

We first recall interval arithmetic, which provides rigorous upper and lower bounds. A real number $x$ is represented by an enclosure $[x]=[\underline x,\overline x]$, with $\underline x\le x\le\overline x$, stored internally as a ball $m\pm r$. Each arithmetic operation and elementary function used in the computation returns an enclosure guaranteed to contain the exact image of the input enclosures.

On the interior interval $I_{0}=[0.530,0.980]$ the layer is not available in closed form, so $J(s,\Phis)$ is never estimated directly. Instead we prove the strict negativity of $J(s,\Phis)$ at a finite grid of values of $s$ using rigorous numerics. This involves evaluating $J(s,\cdot)$ at explicit approximate layers and then controlling the distance between the true layer and the numerical approximation. We then transfer the value of $J$ from the approximate to the true profile and propagate the sign from the grid to all of $I_{0}$ using a bound on the second derivative of $\Es$. All numerical bounds entering this argument are certified using interval arithmetic in the \texttt{Arb} library~\cite{ArbJohansson} at $200$-bit working precision.

The computer-assisted argument completes the monotonicity proof on the interior interval. The broader strategy of combining quantitative asymptotic analysis with computer-assisted estimates has precedents in related PDE problems~\cite{BuckmasterCaoLaboraGomezSerrano,CaoLaboraColomboDolceVentura,CastroGomezSerranoPascualCaballo,ChenHouI,ChenHouII,DahneFKdV2024,DahneFigueras,DahneGomezSerrano2023,DahneGomezSerranoPechAlberich2026,DonningerSchorkhuber}. In the present paper, computer assistance enters at the following two points.
\begin{itemize}
	\item We certify, in interval arithmetic, the quantities entering the conditions~\textup{(C1)--(C4)} of Proposition~\ref{prop:cells-I0}: the value of $J(s_{i},\cdot)$ at the approximate layer, the radius $\tau(s_{i})$ produced by the fixed-point argument, the Lipschitz constant $L_{J}(s_{i})$ that transfers this value to the true layer, and the curvature constant $M^{\ast}_{\mathrm{sc}}([a_{i},b_{i}])$.
	\item We validate the Schur inequality~\eqref{eq:Gprime-cert-BS}, on which Proposition~\ref{prop:Gprime-coercivity-cert} rests, and which yields the inverse estimate used in the fixed-point argument of Proposition~\ref{prop:layer-radius}. The pointwise inequality is recast as a weighted integral bound, in the spirit of the operator-norm estimates of~\cite{CastroCordobaGomezSerranoSQG} for the SQG equation.
\end{itemize}
Related fixed-point validations of numerical approximations for nonlocal equations and semilinear PDEs on unbounded domains can be found in~\cite{CadiotWhitham,CadiotLessardNave,CadiotHaziot}. We refer to the book~\cite{Tucker2011} for an introduction to validated numerics, and to the survey~\cite{GomezSerranoSurvey2019} and the book~\cite{NakaoPlumWatanabe2019} for a more specific treatment of computer-assisted proofs in PDE.

\noindent\textbf{Organization of the paper.} Section~\ref{sec:cont} proves the local Lipschitz continuity of $\Es$ on $(1/2,1]$ and establishes the identity for $\Es'(s)$ and the Lipschitz bound for $J(s, \cdot)$ used throughout. Sections~\ref{sec:s1} and~\ref{sec:s12} establish the endpoint expansions of Theorem~\ref{thm:asymp} and the corresponding strict decrease near $s=1$ and $s=1/2$, respectively. Section~\ref{sec:mono-interior} treats the remaining interior interval and assembles the three regimes into the strict monotonicity of Theorem~\ref{thm:mono}. Appendix~\ref{app:implementation} explains the details of the computer-assisted argument used in the interior analysis. The code is attached as supplementary material and can also be found in the GitHub repository~\cite{GithubRepo}.

\noindent\textbf{Usage of Large Language Models.} The authors have used a combination of the LLMs GPT-5.5 and Claude Opus 4.7 during the development of this work. These models have been employed for mathematical exploration, including suggesting proof strategies and assisting with calculations. The implementation of the code in the supplementary material was also done with the assistance of the same models. All mathematical statements, proofs and code scripts have been checked by the authors, who take full responsibility. All proofs and code scripts influenced by LLMs have been substantially rewritten. In the last stage of the project, GPT-5.6 Sol and Claude Opus 5 were used for proofreading.

\section{Continuity and differentiability of the energy}\label{sec:cont}
The continuity of the energy $\Es$ admits a direct proof that uses only minimality and a uniform Fourier estimate for the layers. Compactness of the profiles enters later, when we differentiate the energy and must pass to the limit in the Fourier transform of~$\Phis'$. The solution of the classical, local Allen--Cahn equation is explicit; we write
\begin{equation*}
	\Phi_1=\tanh(x/\sqrt2).
\end{equation*}

\subsection{Lipschitz continuity}\label{ss:cont-lipschitz}

The first step is to prove a Fourier-side representation for the kinetic energy term in $E_s$. Since the layer solutions $\Phi_s$ do not decay at infinity, taking their Fourier transforms directly is not possible. The next lemma solves this problem using difference quotients.

\begin{lemma}\label{lem:fourier-energy-nondecaying}
	Let $u\in W^{1,1}_{\mathrm{loc}}(\R)$ satisfy $u'\in L^{1}(\R)$. Then, for every $s\in(0,1)$,
	\begin{equation*}
		\frac{\cs}{4}\iint_{\R\times\R}
		\frac{|u(x)-u(y)|^{2}}{|x-y|^{1+2s}}\,dx\,dy
		=
		\frac{1}{4\pi}\int_{\R}
		|\xi|^{2s-2}|\widehat{u'}(\xi)|^{2}\,d\xi .
	\end{equation*}
\end{lemma}

\begin{proof}
	For $h\ne0$, let $(\tau_hu)(x):=u(x-h)$ and set $I_h:=(0,h)$ if $h>0$ and $I_h:=(h,0)$ if $h<0$. The fundamental theorem of calculus gives
	\begin{equation*}
		u-\tau_hu=(\sgn h)\int_{I_h}u'(x-t)\,dt
		=(\sgn h)\,\mathbf 1_{I_h}*u'.
	\end{equation*}
	Since $\mathbf 1_{I_h}\in L^{2}(\R)$ and $u'\in L^{1}(\R)$, Young's inequality shows that $u-\tau_hu\in L^{2}(\R)$. Thus Plancherel's theorem and the convolution formula apply to this difference and give
	\begin{equation}\label{eq:translation-difference-plancherel}
		\int_{\R}|u(x)-u(x-h)|^{2}\,dx
		=
		\frac{1}{2\pi}\int_{\R}
		\frac{2(1-\cos(h\xi))}{\xi^{2}}
		|\widehat{u'}(\xi)|^{2}\,d\xi ,
	\end{equation}
	where we have used $|\widehat{\mathbf 1_{I_{h}}}(\xi)|^{2}=2\bigl(1-\cos h\xi\bigr)/\xi^{2}$.

	Changing variables $h=x-y$ in the Gagliardo seminorm, substituting \eqref{eq:translation-difference-plancherel}, and applying Tonelli's theorem, we obtain
	\begin{align*}
		\iint_{\R\times\R}
		\frac{|u(x)-u(y)|^{2}}{|x-y|^{1+2s}}\,dx\,dy
		=
		\int_{\R}\frac{1}{|h|^{1+2s}}
		\int_{\R}|u(x)-u(x-h)|^{2}\,dx\,dh
		=
		\frac{K_s}{\pi}\int_{\R}
		|\xi|^{2s-2}|\widehat{u'}(\xi)|^{2}\,d\xi,
	\end{align*}
	where we have also applied the change of variable $\eta=h\xi$ and the constant $K_s$ is
	\begin{equation*}
		K_s:=\int_{\R}\frac{1-\cos\eta}{|\eta|^{1+2s}}\,d\eta.
	\end{equation*}
	The constant $K_s$ is finite and positive for $s\in(0,1)$. Applying the identity just obtained to a Schwartz function and comparing it with the Plancherel identity~\eqref{eq:plancherel} yields $\cs K_s=1$.
\end{proof}

We now control the variation of the energy exponent while keeping the profile fixed. The point is that the estimate is uniform over all layers in a compact subinterval of~$(1/2,1]$.

\begin{lemma}\label{lem:layer-energy-variation}
	Let $a\in(1/2,1)$. For every $\sigma,s\in[a,1]$, the quantity $E_s[\Phi_\sigma]$ is finite. Moreover, the map $s\mapsto E_s[\Phi_\sigma]$ is continuously differentiable on $[a,1]$, with one-sided derivatives at the endpoints, and there is a constant $C_a<+\infty$ such that
	\begin{equation}\label{eq:fixed-layer-Lipschitz}
		|J(s,\Phi_\sigma)|\le C_a,
		\qquad
		|E_s[\Phi_\sigma]-E_r[\Phi_\sigma]|\le C_a|s-r|
	\end{equation}
	for all $\sigma,s,r\in[a,1]$.
\end{lemma}

\begin{proof}
	Theorem~\ref{thm:CS} gives monotonicity and $\Phi_\sigma(\pm \infty) = \pm 1$, so the measure $\Phi_\sigma'\,dx$ is nonnegative and has total mass~$2$. Hence $|\widehat{\Phi_\sigma'}(\xi)|\le2$. Differentiating the layer equation,
	\begin{equation*}
		(-\Delta)^{\sigma} \Phi_\sigma' = (1 - 3\Phi_\sigma^{2})\Phi_\sigma'.
	\end{equation*}
	Then, passing to Fourier variables gives
	\begin{equation*}
		|\xi|^{2\sigma}\widehat{\Phi_\sigma'}(\xi)
		=\reallywidehat{(1-3\Phi_\sigma^2)\Phi_\sigma'}(\xi).
	\end{equation*}
	Since $|1-3\Phi_\sigma^2|\le2$ and $\|\Phi_\sigma'\|_{L^1}=2$, we obtain, using $| \reallywidehat{(1-3\Phi_\sigma^2)\Phi_\sigma'}(\xi) | \leq \|1-3\Phi_\sigma^2\|_{L^{\infty}} \| \Phi_\sigma' \|_{L^{1}}$,
	\begin{equation}\label{eq:hf-fourier-bound-lem}
		|\widehat{\Phi_\sigma'}(\xi)|
		\le \min\{2,4|\xi|^{-2\sigma}\}
		\qquad(\sigma \le 1,\ \xi\ne0).
	\end{equation}

	For $\sigma,s\in[a,1]$, it follows from \eqref{eq:hf-fourier-bound-lem} that
	\begin{equation}\label{eq:dom-lem}
		|\log|\xi||\,|\xi|^{2s-2}|\widehat{\Phi_\sigma'}(\xi)|^2
		\le G_a(\xi),
	\end{equation}
	where
	\begin{equation*}
		G_a(\xi):=
		4|\log|\xi||\,|\xi|^{2a-2}\mathbf 1_{\{|\xi|<1\}}
		+16|\log|\xi||\,|\xi|^{-4a}\mathbf 1_{\{|\xi|\ge1\}}.
	\end{equation*}
	The function $G_a$ belongs to $L^1(\R)$. Omitting the logarithmic factor gives an integrable majorant for the spectral kinetic-energy integrand. Since $\Phi_\sigma'\in L^1(\R)$, Lemma~\ref{lem:fourier-energy-nondecaying}, applied to $\Phi_\sigma$, gives
	\begin{equation}\label{eq:fixed-layer-fourier-energy}
		E_s[\Phi_\sigma]
		=\frac{1}{4\pi}\int_\R |\xi|^{2s-2}
		|\widehat{\Phi_\sigma'}(\xi)|^2\,d\xi
		+\int_\R W(\Phi_\sigma)\,dx
	\end{equation}
	for every $\sigma\in[a,1]$ and $s\in[a,1)$. Here the case $s=1$ is precisely the classical energy~\eqref{eq:E1}. The potential term is finite by the Cabr\'e--Sire tail estimate~\eqref{eq:CS-tail} when $\sigma<1$, and by the explicit formula for~$\Phi_1$ when $\sigma=1$.

	Differentiating~\eqref{eq:fixed-layer-fourier-energy} under the integral sign, using~\eqref{eq:dom-lem}, yields
	\begin{equation}\label{eq:D1-lem}
		\frac{d}{ds}E_s[\Phi_\sigma]
		=\frac{1}{2\pi}\int_\R\log|\xi|\,|\xi|^{2s-2}
		|\widehat{\Phi_\sigma'}(\xi)|^2\,d\xi
		=J(s,\Phi_\sigma).
	\end{equation}
	The dominated differentiation theorem and the same majorant give continuity of the derivative up to the endpoints. Taking $C_a:=(2\pi)^{-1}\|G_a\|_{L^1}$ and integrating \eqref{eq:D1-lem} between $r$ and~$s$ proves \eqref{eq:fixed-layer-Lipschitz}.
\end{proof}

The uniformity in both the exponent and the layer allows us to compare two minimizers directly.

\begin{theorem}\label{thm:cont}
	For every $a\in(1/2,1)$ there exists $C_a<+\infty$ such that
	\begin{equation}\label{eq:energy-local-Lipschitz}
		|\Es(s)-\Es(r)|\le C_a|s-r|
		\qquad(r,s\in[a,1]).
	\end{equation}
	Consequently, $\Es:(1/2,1]\to\R$ is locally Lipschitz continuous.
\end{theorem}

\begin{proof}
	Let $a\le r<s\le1$. By Lemma~\ref{lem:layer-energy-variation}, each of $\Phi_r$ and $\Phi_s$ is admissible for both energies. Theorem~\ref{thm:PSV}, together with the classical minimality at $s=1$ recalled after that theorem, therefore gives $\Es(s)\le E_s[\Phi_r]$ and $\Es(r)\le E_r[\Phi_s]$. Subtracting these inequalities in the two possible directions yields
	\begin{equation}\label{eq:sandwich-lem}
		E_s[\Phi_s]-E_r[\Phi_s]
		\le \Es(s)-\Es(r)
		\le E_s[\Phi_r]-E_r[\Phi_r].
	\end{equation}
	Both outer terms are bounded in absolute value by $C_a(s-r)$, by \eqref{eq:fixed-layer-Lipschitz}.
\end{proof}

\begin{remark}\label{rmk:endpoint-limit}
	Taking $r=1$ in~\eqref{eq:energy-local-Lipschitz} gives, for every $a\in(1/2,1)$,
	\begin{equation*}
		|\Es(s)-\Es(1)|\le C_a(1-s)\quad(s\in[a,1]),
		\qquad
		\lim_{s\to1^-}\Es(s)=\Es(1)=\frac{2\sqrt2}{3}.
	\end{equation*}
	To see the last equality note that, since $\frac12|\Phi_1'|^2=W(\Phi_1)=\frac14\mathrm{sech}^4(x/\sqrt2)$, we can compute
	\begin{equation*}
		\Es(1)
		=\frac12\int_{\mathbb R}\mathrm{sech}^4(x/\sqrt2)\,dx
		=\frac{\sqrt2}{2}\int_{\mathbb R}\mathrm{sech}^4 (y)\,dy
		=\frac{\sqrt2}{2} \frac43
		=\frac{2\sqrt2}{3}.
	\end{equation*}
\end{remark}

\subsection{The derivative of the energy}\label{ss:cont-deriv}

The preceding argument does not require any continuity of the profiles with respect to~$s$. Such continuity is needed, however, to identify the derivative of the minimizing energy. We record it now in the precise form used below.

\begin{proposition}\label{prop:equicoerc}
	Let $s_n\to s_*\in(1/2,1)$. Then
	\begin{equation*}
		\Phi_{s_n}\longrightarrow\Phi_{s_*}
		\quad\text{pointwise on $\R$ and in $L^1_{\mathrm{loc}}(\R)$.}
	\end{equation*}
	Moreover,
	\begin{equation*}
		\widehat{\Phi_{s_n}'}(\xi)
		\longrightarrow\widehat{\Phi_{s_*}'}(\xi)
		\qquad(\xi\in\R).
	\end{equation*}
\end{proposition}

\begin{proof}
	Choose $[a,b]\subset(1/2,1)$ with $s_*\in(a,b)$ and $s_n\in[a,b]$ for all sufficiently large~$n$. Theorem~\ref{thm:cont} gives
	\begin{equation}\label{eq:uniform-energy}
		M_{a,b}:=\sup_{s\in[a,b]}\Es(s)<+\infty.
	\end{equation}
	Since every $\Phi_s$ is odd, increasing, and takes values in $[-1,1]$, Helly's selection theorem gives, along a subsequence, a nondecreasing function $\Phi_\infty:\R\to[-1,1]$ such that $\Phi_{s_n}\to\Phi_\infty$ at every continuity point of~$\Phi_\infty$. In particular, the convergence holds almost everywhere and in $L^1_{\mathrm{loc}}(\R)$, and the limit is odd with $\Phi_\infty(0)=0$, after changing it on a countable set if necessary.

	The energy bound prevents the transition from spreading out. Indeed, for $R\ge1$ set $\alpha_s(R):=2(1-\Phi_s(R))$. The monotonicity and oddness of $\Phi_s$, together with the fact that $W$ is even and decreasing on $[0,1]$, give
	\begin{equation*}
		\int_\R W(\Phi_s)\,dx
		\geq \int_{-R}^{R} W(\Phi_s)\,dx
		\ge 2R\,W(1-\alpha_s(R)/2)
		\ge \frac{R\alpha_s(R)^2}{16},
	\end{equation*}
	where in the last step we have also used $W(t) \ge (1/4)(1-t)^{2}$. It follows from~\eqref{eq:uniform-energy} that
	\begin{equation}\label{eq:uniform-layer-tail}
		2(1-\Phi_s(R))\le4\sqrt{M_{a,b}}\,R^{-1/2}
		\qquad(s\in[a,b],\ R\ge1).
	\end{equation}
	Passing to the limit at continuity points of $\Phi_\infty$, and then letting $R\to\infty$, shows that $\Phi_\infty(\pm\infty)=\pm1$.

	It remains to identify the limit. For $\varphi\in C_c^\infty(\R)$, the singular-integral representation~\eqref{eq:fl-singular} gives, uniformly for $s\in[a,b]$,
	\begin{equation*}
		|(-\Delta)^s\varphi(x)|
		\le C(\varphi)(1+|x|)^{-1-2a}.
	\end{equation*}
	Moreover, $(-\Delta)^{s_n}\varphi(x)\to(-\Delta)^{s_*}\varphi(x)$ pointwise. Testing the layer equation against~$\varphi$ and using dominated convergence therefore yields
	\begin{equation*}
		\int_\R\Phi_\infty(-\Delta)^{s_*}\varphi\,dx
		=\int_\R(\Phi_\infty-\Phi_\infty^3)\varphi\,dx.
	\end{equation*}
	Thus $\Phi_\infty$ is a bounded distributional solution of the layer equation at~$s_*$. Interior regularity~\cite{CabreSire2014I,CabreSire2015II} makes it a classical solution, and the uniqueness statement in Theorem~\ref{thm:CS}, together with its oddness, normalization, and limits at infinity, gives $\Phi_\infty=\Phi_{s_*}$. Every subsequence has a further subsequence with this same limit, hence the full sequence converges pointwise and in $L^1_{\mathrm{loc}}$.

	It remains to prove the convergence of the Fourier transforms. Fix $\xi\in\R$. For every $R\ge1$, integration by parts gives
	\begin{equation*}
		\int_{-R}^{R}e^{-i\xi x}\Phi_{s_n}'(x)\,dx
		=e^{-i\xi R}\Phi_{s_n}(R)-e^{i\xi R}\Phi_{s_n}(-R)
		+i\xi\int_{-R}^{R}e^{-i\xi x}\Phi_{s_n}(x)\,dx.
	\end{equation*}
	The pointwise convergence at $\pm R$ and the convergence in $L^1(-R,R)$ show that the truncated integral converges to the corresponding integral for~$\Phi_{s_*}'$. On the other hand, since $\Phi_s'\ge0$,
	\begin{equation*}
		\left|\widehat{\Phi_s'}(\xi)
		-\int_{-R}^{R}e^{-i\xi x}\Phi_s'(x)\,dx\right|
		\le\int_{\{|x|>R\}}\Phi_s'(x)\,dx
		=2(1-\Phi_s(R))
		\le4\sqrt{M_{a,b}}\,R^{-1/2}
	\end{equation*}
	for every $s\in[a,b]$, by~\eqref{eq:uniform-layer-tail}. Consequently,
	\begin{equation*}
		\limsup_{n\to\infty}
		\left|\widehat{\Phi_{s_n}'}(\xi)
		-\widehat{\Phi_{s_*}'}(\xi)\right|
		\le8\sqrt{M_{a,b}}\,R^{-1/2}.
	\end{equation*}
	Letting $R\to\infty$ proves the convergence.
\end{proof}

We now combine the fixed-layer differentiation formula with the compactness just proved. Recall the functional~$J$ from~\eqref{eq:J-functional}.

\begin{lemma}\label{lem:envelope-paper}
	The function $\Es$ belongs to $C^1((1/2,1))$, and
	\begin{equation}\label{eq:envelope-identity}
		\Es'(s)=J(s,\Phis)
		=\frac{1}{2\pi}\int_\R\log|\xi|\,|\xi|^{2s-2}
		|\widehat{\Phis'}(\xi)|^2\,d\xi
		\qquad(s\in(1/2,1)).
	\end{equation}
\end{lemma}

\begin{proof}
	Fix $s_*\in(1/2,1)$, and choose $a\in(1/2,s_*)$. We first note the joint continuity needed below. If $\sigma_k\to s_*$ and $\tau_k\to s_*$, then Proposition~\ref{prop:equicoerc} gives
	\begin{equation*}
		\widehat{\Phi_{\tau_k}'}(\xi)
		\longrightarrow\widehat{\Phi_{s_*}'}(\xi)
		\qquad(\xi\in\R).
	\end{equation*}
	For all sufficiently large~$k$, the integrands below are dominated by the function $G_a$ from~\eqref{eq:dom-lem}. Hence dominated convergence gives
	\begin{equation}\label{eq:D2-lem}
		\frac{1}{2\pi}\int_\R\log|\xi|\,|\xi|^{2\sigma_k-2}
		|\widehat{\Phi_{\tau_k}'}(\xi)|^2\,d\xi
		\longrightarrow J(s_*,\Phi_{s_*}).
	\end{equation}

	For $h>0$ sufficiently small, divide~\eqref{eq:sandwich-lem}, with $r=s_*$ and $s=s_*+h$, by~$h$. This gives
	\begin{equation*}
		\frac{E_{s_*+h}[\Phi_{s_*+h}]-E_{s_*}[\Phi_{s_*+h}]}{h}
		\le\frac{\Es(s_*+h)-\Es(s_*)}{h}
		\le\frac{E_{s_*+h}[\Phi_{s_*}]-E_{s_*}[\Phi_{s_*}]}{h}.
	\end{equation*}
	By~\eqref{eq:D1-lem}, the right-hand quotient tends to $J(s_*,\Phi_{s_*})$. The mean value theorem writes the left-hand quotient as $J(\widetilde s_h,\Phi_{s_*+h})$ for some $\widetilde s_h\in(s_*,s_*+h)$, and~\eqref{eq:D2-lem} gives the same limit. Thus the right derivative of~$\Es$ at~$s_*$ equals $J(s_*,\Phi_{s_*})$. Applying the same argument to $h<0$ gives the left derivative and proves~\eqref{eq:envelope-identity}. Finally, \eqref{eq:D2-lem} with $\sigma_k=\tau_k$ shows that $s\mapsto J(s,\Phis)$ is continuous, and therefore $\Es\in C^1((1/2,1))$.
\end{proof}

The monotonicity arguments of Sections~\ref{sec:s1} and~\ref{sec:mono-interior} never estimate $J(s,\Phis)$ directly. The layer $\Phis$ is not available in an explicit expression, so instead we evaluate $J(s,\cdot)$ at an explicit profile close to $\Phis$ in $H^{2s}(\R)$ and then transfer the result to $\Phis$. The error is measured in the weighted seminorm
\begin{equation}\label{eq:starnorm}
	\|u\|^{2}_{\star}(s)\;:=\;\frac{1}{2\pi}\int_{\R}|{\log|\xi|}|\;|\xi|^{2s-2}\,|\widehat{u'}(\xi)|^{2}\,d\xi.
\end{equation}
This is the natural seminorm for $J$, obtained by replacing $\log|\xi|$ in~\eqref{eq:J-functional} by its absolute value. The associated Sobolev constant of this seminorm is
\begin{equation}\label{eq:Cstar-LJ-def}
	C_{\star}(s)\;:=\;\Bigl(\sup_{t>0}\;\frac{|\log t|\,t^{s}}{2(1+t)^{2s}}\Bigr)^{1/2}.
\end{equation}
We record the error inequality once here, in the generality in which Sections~\ref{sec:s1} and~\ref{sec:mono-interior} use it.

\begin{lemma}\label{lem:LJ-polarization}
	Let $s\in(1/2,1]$. Then $C_{\star}(s)<\infty$, and the following hold.

	\smallskip\noindent (i) For every $w\in H^{2s}(\R)$,
	\begin{equation}\label{eq:star-sobolev}
		\|w\|_{\star}(s)\;\le\;C_{\star}(s)\,\|w\|_{H^{2s}(\R)}.
	\end{equation}

	\smallskip\noindent (ii) If $u_{1}$ and $u_{2}$ have finite seminorms
	$\|u_{1}\|_{\star}(s)$ and $\|u_{2}\|_{\star}(s)$, then
	\begin{equation}\label{eq:LJ-bound}
		\bigl|J(s,u_{1})-J(s,u_{2})\bigr|\;\le\;\bigl(\|u_{1}\|_{\star}+\|u_{2}\|_{\star}\bigr)\,\|u_{1}-u_{2}\|_{\star}.
	\end{equation}
\end{lemma}

\begin{proof}
	The supremum in~\eqref{eq:Cstar-LJ-def} is finite because $t\mapsto|\log t|\,t^{s}/(1+t)^{2s}$ is continuous on $(0,\infty)$ and tends to $0$ as $t\to0^{+}$ and as $t\to\infty$.

	\smallskip
	\noindent\emph{(i).}  Let $w\in H^{2s}(\R)$.  Then $\widehat{w'}=i\xi\widehat w$
	in the sense of distributions, so $|\xi|^{2s-2}|\widehat{w'}(\xi)|^{2}
		=|\xi|^{2s}|\widehat w(\xi)|^{2}$ and
	\begin{equation*}
		\|w\|_{\star}(s)^{2}
		=\frac{1}{2\pi}\int_{\R}\bigl|\log|\xi|\bigr|\,|\xi|^{2s}\,|\widehat w(\xi)|^{2}\,d\xi .
	\end{equation*}
	For $\xi\ne0$, applying the definition~\eqref{eq:Cstar-LJ-def} with $t=|\xi|^{2}$ gives the pointwise multiplier bound
	\begin{equation*}
		\bigl|\log|\xi|\bigr|\,|\xi|^{2s}
		=\frac{\bigl|\log(|\xi|^{2})\bigr|\,\bigl(|\xi|^{2}\bigr)^{s}}{2}
		\le C_{\star}(s)^{2}\,(1+|\xi|^{2})^{2s},
	\end{equation*}
	and integrating against $|\widehat w|^{2}/(2\pi)$ yields~\eqref{eq:star-sobolev}.

	\smallskip
	\noindent\emph{(ii).}  The functional $J(s,\cdot)$ is the quadratic form
	associated with the real weight $\log|\xi|\,|\xi|^{2s-2}$, so
	\begin{equation*}
		J(s,u_{1})-J(s,u_{2})=\frac{1}{2\pi}\operatorname{Re}\int_{\R}\log|\xi|\,|\xi|^{2s-2}\,\widehat{(u_{1}+u_{2})'}\,\overline{\widehat{(u_{1}-u_{2})'}}\,d\xi .
	\end{equation*}
	Now, applying Cauchy--Schwarz,
	\begin{equation*}
		\frac{1}{2\pi}\int_{\R}\bigl|\log|\xi|\bigr|\,|\xi|^{2s-2}\,
		\bigl|\widehat{(u_{1}+u_{2})'}\bigr|\,
		\bigl|\widehat{(u_{1}-u_{2})'}\bigr|\,d\xi
		\le \|u_{1}+u_{2}\|_{\star}\,\|u_{1}-u_{2}\|_{\star},
	\end{equation*}
	which is finite by the triangle inequality for $\|\cdot\|_{\star}$. Taking real parts and using $\|u_{1}+u_{2}\|_{\star}\le\|u_{1}\|_{\star}+\|u_{2}\|_{\star}$ gives~\eqref{eq:LJ-bound}.
\end{proof}

\section{The endpoint \texorpdfstring{$s=1$}{s=1}: expansion and strict decrease}\label{sec:s1}
\providecommand{\Lone}{\mathcal{L}_{1}}
\providecommand{\Ltwo}{\mathcal{L}_{2}}
\providecommand{\Jop}{\mathcal{J}}
\providecommand{\F}{\mathcal{F}}
\providecommand{\Gs}{\mathcal{G}_{s}}
\providecommand{\sech}{\operatorname{sech}}

For the case $s=1$ the layer is the profile
\begin{equation*}
	\Phi_1(x)=\tanh(x/\sqrt{2}),
\end{equation*}
and Remark~\ref{rmk:endpoint-limit} shows $\Es(1) = 2\sqrt{2}/3$. This section expands $\Es$ around that value. A single expansion, carried to second order with explicit remainders on $[49/50,1]$, yields both results we need at this endpoint. First we obtain the first-order expansion, which is the $s\to1^{-}$ statement of Theorem~\ref{thm:asymp}, and then we obtain the strict decrease on the right endpoint interval of the cover~\eqref{eq:three-interval-cover}. For the asymptotic statement alone, qualitative remainders would suffice. We retain explicit constants because the endpoint argument must give a quantitative sign on the whole interval $[0.980,1]$, so that it meets the computer-assisted analysis on $I_{0}=[0.530,0.980]$ without leaving a gap.

\begin{theorem}\label{thm:kappa1}
	Let $\kappa_{1}$ and $\kappa_{2}$ be the constants defined later in~\eqref{eq:kappa-def}, and write $\varepsilon:=1-s$. Then $\kappa_{2}\ge0$, and the following hold for $s\in[49/50,1]$.

	\smallskip\noindent (i) Expansion:
	\begin{equation}\label{eq:URem2-integrated}
		\Es(s) \;=\; \frac{2\sqrt{2}}{3} \;+\; \kappa_{1}\varepsilon
		\;+\;\kappa_{2}\varepsilon^{2}\;+\;R_{3}(s),
		\qquad |R_{3}(s)|\;\le\;80\,\varepsilon^{3}.
	\end{equation}

	\smallskip\noindent (ii) Derivative: $\Es\in C^{1}([49/50,1))$ and
	\begin{equation}\label{eq:URem2-paper}
		\Es'(s) \;=\; -\kappa_{1}-2\kappa_{2}\varepsilon+R_{2}(s),
		\qquad |R_{2}(s)|\;\le\;240\,\varepsilon^{2}
		\qquad (s\in[49/50,1)).
	\end{equation}
\end{theorem}

Part~(i) contains in particular the statement $\Es(s)=\tfrac{2\sqrt2}{3}+\kappa_{1}(1-s)+o(1-s)$ asserted in Theorem~\ref{thm:asymp}(i), with the remainder improved from $o(1-s)$ to $O((1-s)^{2})$. Part~(ii) gives the monotonicity near the endpoint. The following proposition is exactly the statement we need, which will be proved at the end of this section.

\begin{proposition}\label{prop:mono-upper}
	Let $\gamma_{\!*}:=\tfrac{1}{50}=0.020$. Then $\Es'(s)\le-0.03$ for every $s\in[1-\gamma_{\!*},1)$, and $\Es$ is strictly decreasing on $[0.980,1]$.
\end{proposition}

Throughout this section we work with the parameter $ \varepsilon := 1-s\in[0,1/50], $. We restrict to $\varepsilon\le1/50$ since the conclusions of this section concern the limit $s\to1^{-}$ and the interval $[0.980,1]$, whose left endpoint is the right endpoint of the interior interval $I_{0}$ of~\eqref{eq:three-interval-cover}. All constants below are explicit and are deliberately rounded upward; none is optimized.

We record once, in the form in which they are used repeatedly below, the inequalities and integrals underlying the explicit constants of this section. From the explicit expression of $\Phi_1$ we have
\begin{equation}\label{eq:Phi1-hat}
	\Phi_1'(x)=\frac{1}{\sqrt2}\sech^{2}\Bigl(\frac{x}{\sqrt2}\Bigr),
	\qquad
	\widehat{\Phi_1'}(\xi)=\frac{\sqrt2\,\pi\,\xi}{\sinh(\pi\xi/\sqrt2)},
	\qquad
	\widehat{\Phi_1}(\xi)=\frac{\widehat{\Phi_1'}(\xi)}{i\xi}
	=\frac{-i\sqrt2\,\pi}{\sinh(\pi\xi/\sqrt2)}
\end{equation}
for $\xi\ne0$. The last two identities follow from the standard Fourier transform of $\sech^{2}$. Since $\Phi_1$ is bounded but not integrable, $\widehat{\Phi_1}$ is understood as a tempered distribution. When the case of $\xi=0$ occurs below, it is always through the combination $|\xi|^{2\sigma}|\widehat{\Phi_1}(\xi)|^{2}=|\xi|^{2\sigma-2}|\widehat{\Phi_1'}(\xi)|^{2}$.

\subsection{The first-order corrector}\label{ss:s1-trial}

The expansion is driven by a single linear correction to the classical layer~$\Phi_1$. We recall the convergence of Fourier multipliers
\begin{equation}\label{eq:Lone-def}
	\frac{\fl f - (-\Delta)f}{1-s} \;\longrightarrow\; \Lone f
	\quad\text{in }\mathcal{S}'(\R),
	\qquad
	\widehat{\Lone f}(\xi)\;=\;-2|\xi|^{2}\log|\xi|\;\widehat f(\xi),
\end{equation}
valid for every $f\in\mathcal{S}(\R)$. The proof is the Taylor expansion
\begin{equation}\label{eq:symbol-taylor}
	|\xi|^{2s}=|\xi|^{2}-2\varepsilon|\xi|^{2}\log|\xi|
	+2\varepsilon^{2}|\xi|^{2}(\log|\xi|)^{2}+\cdots,
\end{equation}
combined with dominated convergence against the Fourier transform of a Schwartz function. We write $\Ltwo$ for the Fourier multiplier with symbol $2|\xi|^{2}(\log|\xi|)^{2}$, the next term of~\eqref{eq:symbol-taylor}.

Linearizing the classical Allen--Cahn equation $-\Phi_1'' = \Phi_1-\Phi_1^{3}$ about $\Phi_1$ yields the operator
\begin{equation}\label{eq:Jacobi}
	\Jop \;:=\; -\partial_{x}^{2} + W''(\Phi_1)
	\;=\; -\partial_{x}^{2} + (3\Phi_1^{2}-1),
\end{equation}
self-adjoint on $L^{2}(\R)$ with domain $H^{2}(\R)$, whose kernel is spanned by the even function $\Phi_1'$. From~\cite[Example~2, pp.~73--75]{NikiforovUvarov1988} we know that, restricted to the odd subspace, $\Jop$ has a strictly positive spectral gap:
\begin{equation}\label{eq:J-odd-spectrum}
	\sigma\bigl(\Jop|_{\Lodd(\R)}\bigr)=\{3/2\}\cup[2,+\infty),
\end{equation}
its essential spectrum being $[2,+\infty)$ and its lowest odd eigenvalue $3/2$, with the explicit eigenfunction $\sech(x/\sqrt{2})\tanh(x/\sqrt{2})$. In particular $\Jop\ge3/2$ on $\Lodd(\R)$, so $\Jop$ is boundedly invertible there with
\begin{equation}\label{eq:J-inverse}
	\bigl\|\Jop^{-1}\bigr\|_{\Lodd\to\Lodd}\le\frac23 .
\end{equation}

The corrector equation is dictated by cancellation of the first-order residual. Indeed, formally substituting $\Phis=\Phi_1+\varepsilon\psi+O(\varepsilon^2)$ into the fractional Allen--Cahn equation and using~\eqref{eq:Lone-def} and~\eqref{eq:Jacobi} gives
\begin{equation*}
	\fl\Phis+W'(\Phis)
	=\varepsilon\bigl(\Jop\psi+\Lone\Phi_1\bigr)+O(\varepsilon^2),
\end{equation*}
since $(-\Delta)\Phi_1+W'(\Phi_1)=0$. Thus the term of order $\varepsilon$ vanishes precisely when $\Jop\psi=-\Lone\Phi_1$, which is the equation solved in the next lemma.

\begin{lemma}\label{lem:corrector}
	Let $g := \Lone \Phi_1$, interpreted via~\eqref{eq:Lone-def} as the tempered distribution with Fourier transform
	\begin{equation}\label{eq:ghat}
		\widehat g(\xi)
		\;=\; -2|\xi|^{2}\log|\xi|\,\widehat{\Phi_1}(\xi)
		\;=\; \frac{2i\sqrt{2}\,\pi\,\xi^{2}\log|\xi|}
		{\sinh(\pi\xi/\sqrt2)}
		\qquad(\xi\ne 0).
	\end{equation}
	Then $g$ is a real-valued odd function in $H^{1}(\R)\cap C_{0}(\R)$, and there is a unique odd $\psi\in H^{3}(\R)\cap C_{0}(\R)$ solving
	\begin{equation}\label{eq:corrector-eq}
		\Jop\psi \;=\; -\Lone \Phi_1\;=\;-g
		\qquad\text{on }\R
	\end{equation}
	in the classical sense. Moreover $\Lone\psi$ and $\fl\psi$ (for $s\in(1/2,1)$) define elements of $C_{0,\mathrm{odd}}(\R)$.
\end{lemma}

\begin{proof}
	The multiplier $-2|\xi|^{2}\log|\xi|$ is $O(|\xi|^{2}|\log|\xi||)$ at the origin and, weighted against $\widehat{\Phi_1}(\xi)=\widehat{\Phi_1'}(\xi)/(i\xi)$, decays exponentially, because $\widehat{\Phi_1'}$ does by~\eqref{eq:Phi1-hat}. Hence $\widehat g$ and $\xi\widehat g$ belong to $L^{1}\cap L^{2}$, so $g \in H^{1}(\mathbb{R})$ and, by the Sobolev embedding, it also belongs to $C_0(\mathbb{R})$. Oddness of $g$ follows from oddness of $\widehat{\Phi_1}$ and evenness of the symbol.

	Since $g\in\Lodd(\R)$, the bounded inverse~\eqref{eq:J-inverse} produces the unique odd solution $\psi:=-\Jop^{-1}g\in\Lodd(\R)$ of~\eqref{eq:corrector-eq} in $L^2(\mathbb{R})$. For the regularity, equation~\eqref{eq:corrector-eq} reads
	\begin{equation}\label{eq:psi-ode}
		\psi''=(3\Phi_1^{2}-1)\psi+g ,
	\end{equation}
	so $\psi\in L^{2}$ and $g\in H^{1}$ give $\psi\in H^{2}$. Differentiating~\eqref{eq:psi-ode} once and using $g'\in L^{2}$ gives $\psi\in H^{3}$. In particular $\psi\in C_{0}$, and $\psi$ solves \eqref{eq:corrector-eq} classically. Consequently the symbols $-2|\xi|^{2}\log|\xi|\,\widehat\psi(\xi)$ and (for $s\in(1/2,1)$) $|\xi|^{2s}\widehat\psi(\xi)$ both belong to $L^{1}(\R)$, so their inverse Fourier transforms $\Lone\psi$ and $\fl\psi$ lie in $C_{0,\mathrm{odd}}(\R)$.
\end{proof}

The following explicit bounds are used by the fixed-point argument of Section~\ref{ss:s1-fixed-point}. They are kept explicit because they propagate through the residual and fixed-point estimates to the quantitative derivative remainder needed up to $s=0.980$.

\begin{lemma}\label{lem:corrector-bounds}
With $g$ and $\psi$ as in Lemma~\ref{lem:corrector},
\begin{equation}\label{eq:g-bounds}
	\|g\|_{L^{2}}\le\tfrac56,
	\qquad
	\|g'\|_{L^{2}}\le\tfrac{13}{10},
\end{equation}
and
\begin{equation*}
\begin{gathered}
\|\psi\|_{L^{2}}\le\tfrac59,\qquad
\|\psi'\|_{L^{2}}\le\tfrac9{10},\qquad
\|\psi''\|_{L^{2}}\le2,\qquad
\|\psi'''\|_{L^{2}}\le\tfrac{41}{10},\\begin{equation*}2pt]
	\|\psi\|_{H^{1}}\le\tfrac{11}{10},\qquad
	\|\psi\|_{H^{3}}\le6,\qquad
	\|\psi\|_{L^{\infty}}\le\tfrac45 .
	\end{gathered}
\end{equation*}
\end{lemma}

\begin{proof}
	\emph{Step 1: the bounds on $g$.}  By~\eqref{eq:ghat}, Plancherel's
	identity, and the evenness of the integrands,
	\begin{equation*}
		\|g\|_{L^{2}}^{2}
		=8\pi\int_{0}^{\infty}
		\frac{\xi^{4}(\log\xi)^{2}}{\sinh^{2}(\pi\xi/\sqrt2)}\,d\xi,\qquad
		\|g'\|_{L^{2}}^{2}
		=8\pi\int_{0}^{\infty}
		\frac{\xi^{6}(\log\xi)^{2}}{\sinh^{2}(\pi\xi/\sqrt2)}\,d\xi.
	\end{equation*}
	We estimate both integrals by splitting them at $\xi=1$. On $(0,1)$, the inequality $\sinh(\pi\xi/\sqrt2)\ge\pi\xi/\sqrt2$ and the substitution $\xi=e^{-t}$ give
  \begin{equation*}
	\begin{aligned}
		8\pi\int_{0}^{1}
		\frac{\xi^{4}(\log\xi)^{2}}{\sinh^{2}(\pi\xi/\sqrt2)}\,d\xi
		 & \le\frac{16}{\pi}\int_{0}^{1}\xi^{2}(\log\xi)^{2}\,d\xi
		=\frac{32}{27\pi}<0.378,                                   \\
		8\pi\int_{0}^{1}
		\frac{\xi^{6}(\log\xi)^{2}}{\sinh^{2}(\pi\xi/\sqrt2)}\,d\xi
		 & \le\frac{16}{\pi}\int_{0}^{1}\xi^{4}(\log\xi)^{2}\,d\xi
		=\frac{32}{125\pi}<0.082.
	\end{aligned}
	\end{equation*}
	For $\xi\ge1$, we have
	\begin{equation*}
		\sinh\left(\frac{\pi\xi}{\sqrt2}\right)
		=\frac12e^{\pi\xi/\sqrt2}\left(1-e^{-\sqrt2\pi\xi}\right)
		\ge\frac{49}{100}e^{\pi\xi/\sqrt2},
	\end{equation*}
	because $\sqrt2\pi\xi>4$ and $e^{-4}<1/50$. Together with $\log\xi\le\xi-1$, this yields
	\begin{equation*}
		\begin{aligned}
			8\pi\int_{1}^{\infty}
			\frac{\xi^{4}(\log\xi)^{2}}{\sinh^{2}(\pi\xi/\sqrt2)}\,d\xi
			 & \le8\pi\left(\frac{100}{49}\right)^{2}
			\int_{1}^{\infty}\xi^{4}(\xi-1)^{2}e^{-\sqrt2\pi\xi}\,d\xi<0.310, \\
			8\pi\int_{1}^{\infty}
			\frac{\xi^{6}(\log\xi)^{2}}{\sinh^{2}(\pi\xi/\sqrt2)}\,d\xi
			 & \le8\pi\left(\frac{100}{49}\right)^{2}
			\int_{1}^{\infty}\xi^{6}(\xi-1)^{2}e^{-\sqrt2\pi\xi}\,d\xi<1.491.
		\end{aligned}
	\end{equation*}
	Indeed, after substituting $\xi=1+t$, expanding the powers, and integrating term by term, the two expressions on the right are respectively
	\begin{equation*}
		\begin{aligned}
			 & 8\pi\left(\frac{100}{49}\right)^{2}e^{-\sqrt2\pi}
			\sum_{k=0}^{4}\binom{4}{k}\frac{(k+2)!}{(\sqrt2\pi)^{k+3}}
			=0.3091\ldots<0.310,                                 \\
			 & 8\pi\left(\frac{100}{49}\right)^{2}e^{-\sqrt2\pi}
			\sum_{k=0}^{6}\binom{6}{k}\frac{(k+2)!}{(\sqrt2\pi)^{k+3}}
			=1.4900\ldots<1.491.
		\end{aligned}
	\end{equation*}
	Combining the estimates on the two regions, we obtain
	\begin{equation*}
		\|g\|_{L^{2}}^{2}<0.688<\left(\frac56\right)^{2},
		\qquad
		\|g'\|_{L^{2}}^{2}<1.573<\left(\frac{13}{10}\right)^{2},
	\end{equation*}
	which proves~\eqref{eq:g-bounds}.

	\smallskip
	\noindent\emph{Step 2: the $L^{2}$-bounds for $\psi$ and its derivatives.}
	Since $\psi=-\Jop^{-1}g$, estimates~\eqref{eq:J-inverse}
	and~\eqref{eq:g-bounds} give
	\begin{equation*}
		\|\psi\|_{L^{2}}
		\le\frac23\|g\|_{L^{2}}
		\le\frac59.
	\end{equation*}
	Testing~\eqref{eq:corrector-eq} against $\psi$ and integrating by parts, we obtain
	\begin{equation*}
		\|\psi'\|_{L^{2}}^{2}+\int_{\R}(3\Phi_1^{2}-1)\psi^{2}\,dx
		=\langle\Jop\psi,\psi\rangle_{L^{2}}=-\langle g,\psi\rangle_{L^{2}}
		\le\|g\|_{L^{2}}\|\psi\|_{L^{2}} .
	\end{equation*}
	Because $3\Phi_1^{2}-1\ge-1$, it follows that
	\begin{equation*}
		\|\psi'\|_{L^{2}}^{2}
		\le\|g\|_{L^{2}}\|\psi\|_{L^{2}}+\|\psi\|_{L^{2}}^{2}
		\le\frac56\cdot\frac59+\left(\frac59\right)^{2}
		=\frac{125}{162}<\left(\frac9{10}\right)^{2}.
	\end{equation*}

	Since $\|3\Phi_1^{2}-1\|_{L^{\infty}}\le2$, equation~\eqref{eq:psi-ode} implies
	\begin{equation*}
		\|\psi''\|_{L^{2}}
		\le2\|\psi\|_{L^{2}}+\|g\|_{L^{2}}
		\le2\cdot\frac59+\frac56=\frac{35}{18}<2.
	\end{equation*}
	Moreover, the explicit formula for $\Phi_1$ shows that
	\begin{equation*}
		\|(3\Phi_1^{2}-1)'\|_{L^{\infty}}
		=6\|\Phi_1\Phi_1'\|_{L^{\infty}}
		=2\sqrt{\frac23}<\frac53.
	\end{equation*}
	Here the maximum is attained where $\Phi_1^{2}=1/3$. Differentiating~\eqref{eq:psi-ode} and using the preceding estimates therefore gives
	\begin{equation*}
		\begin{aligned}
			\|\psi'''\|_{L^{2}}
			 & \le\|(3\Phi_1^{2}-1)'\|_{L^{\infty}}\|\psi\|_{L^{2}}
			+\|3\Phi_1^{2}-1\|_{L^{\infty}}\|\psi'\|_{L^{2}}
			+\|g'\|_{L^{2}}                                          \\
			 & \le\frac53\cdot\frac59+2\cdot\frac9{10}+\frac{13}{10}
			=\frac{25}{27}+\frac{31}{10}<\frac{41}{10}.
		\end{aligned}
	\end{equation*}

	\smallskip
	\noindent\emph{Step 3: the Sobolev bounds.}  By Plancherel's identity and the
	estimates just proved,
	\begin{equation*}
		\|\psi\|_{H^{1}}^{2}
		=\|\psi\|_{L^{2}}^{2}+\|\psi'\|_{L^{2}}^{2}
		\le\left(\frac59\right)^{2}+\left(\frac9{10}\right)^{2}
		<\left(\frac{11}{10}\right)^{2}.
	\end{equation*}
	Similarly, using $(1+\xi^{2})^{3}=1+3\xi^{2}+3\xi^{4}+\xi^{6}$, we have
	\begin{equation*}
		\begin{aligned}
			\|\psi\|_{H^{3}}^{2}
			 & =\|\psi\|_{L^{2}}^{2}+3\|\psi'\|_{L^{2}}^{2}
			+3\|\psi''\|_{L^{2}}^{2}+\|\psi'''\|_{L^{2}}^{2} \\
			 & \le\left(\frac59\right)^{2}
			+3\left(\frac9{10}\right)^{2}
			+3\cdot2^{2}
			+\left(\frac{41}{10}\right)^{2}
			<32<6^{2}.
		\end{aligned}
	\end{equation*}
	Finally, Lemma~\ref{lem:sobolev-constants} with $\sigma=1$ gives $C^{\ast}(1)=1/\sqrt2$, and hence
	\begin{equation*}
		\|\psi\|_{L^{\infty}}\le C^{\ast}(1)\|\psi\|_{H^{1}}
		\le\frac{11}{10\sqrt2}<\frac45 . \qedhere
	\end{equation*}
\end{proof}

For each $s\in[49/50,1)$ we now define the approximate profile
\begin{equation*}
	u_{s} \;:=\; \Phi_1 + (1-s)\,\psi \;=\; \Phi_1 + \varepsilon\,\psi,
\end{equation*}
an odd smooth function on $\R$ with $u_{s}\to\pm1$ at $\pm\infty$. By Lemma~\ref{lem:corrector-bounds} and $\varepsilon\le1/50$,
\begin{equation}\label{eq:us-bounds}
	\|u_{s}\|_{L^{\infty}}\le1+\varepsilon\|\psi\|_{L^{\infty}}\le1+\tfrac1{50}\cdot\tfrac45
	\le\tfrac{51}{50},
	\qquad
	\|W''(u_{s})\|_{L^{\infty}}\le3\bigl(\tfrac{51}{50}\bigr)^{2}-1\le\tfrac{107}{50},
\end{equation}
the latter because $|W''(u)|=|3u^{2}-1|\le\max(1,3\|u\|_{L^{\infty}}^{2}-1)$.

We can now define the two constants of Theorem~\ref{thm:kappa1}. Both are integrals of the explicit classical layer against a logarithmic weight:
\begin{equation}\label{eq:kappa-def}
\begin{gathered}
  \kappa_{1}:=-\frac{1}{2\pi}\int_{\R}\log|\xi|\;|\widehat{\Phi_1'}(\xi)|^{2}\,d\xi
  =-2\pi\int_{0}^{\infty}
  \frac{\xi^{2}\log\xi}{\sinh^{2}(\pi\xi/\sqrt2)}\,d\xi,\\[2pt]
	\kappa_{2}:=\mu_{\Phi_1}-\tfrac12\langle\Jop\psi,\psi\rangle_{L^{2}},
	\quad
	\mu_{\Phi_1}:=\frac{1}{2\pi}\int_{\R}(\log|\xi|)^{2}\,|\widehat{\Phi_1'}(\xi)|^{2}\,d\xi
	=2\pi\int_{0}^{\infty}
	\frac{\xi^{2}(\log\xi)^{2}}{\sinh^{2}(\pi\xi/\sqrt2)}\,d\xi,
\end{gathered}
\end{equation}
	the second expression in each line following from $|\widehat{\Phi_1'}(\xi)|^{2} =2\pi^{2}\xi^{2}/\sinh^{2}(\pi|\xi|/\sqrt2)$ and the evenness of the integrands. Both integrals converge absolutely. Their role is suggested by a formal expansion of the derivative functional~\eqref{eq:J-functional}. Substituting $u_s=\Phi_1+\varepsilon\psi$, expanding $|\xi|^{-2\varepsilon}=1-2\varepsilon\log|\xi|+O(\varepsilon^2)$, and using the corrector equation and Plancherel for the cross term gives
\begin{equation*}
	\begin{aligned}
		J(s,u_s)
		 & =\frac{1}{2\pi}\int_{\R}\log|\xi|
		\Bigl[|\widehat{\Phi_1'}|^{2}
		+2\varepsilon\operatorname{Re}\bigl(\widehat{\Phi_1'}\,
		\overline{\widehat{\psi'}}\bigr)
		-2\varepsilon\log|\xi|\,|\widehat{\Phi_1'}|^{2}\Bigr]\,d\xi
		+O(\varepsilon^{2})                                    \\
		 & =-\kappa_1-\varepsilon\langle g,\psi\rangle_{L^{2}}
		-2\varepsilon\mu_{\Phi_1}+O(\varepsilon^{2}).
	\end{aligned}
\end{equation*}

\begin{lemma}\label{lem:kappa-constants}
	The constants in~\eqref{eq:kappa-def} satisfy $\kappa_{1}>119/900$ and $\kappa_{2}>0$.
\end{lemma}

\begin{proof}
	By~\eqref{eq:kappa-def},
	\begin{equation*}
		\kappa_{1}=2\pi(A_{0}-A_{\infty}),
	\end{equation*}
	where
	\begin{equation*}
		\begin{aligned}
			A_{0}      & :=\int_{0}^{1}
			\frac{\xi^{2}|\log\xi|}{\sinh^{2}(\pi\xi/\sqrt2)}\,d\xi, \\
			A_{\infty} & :=\int_{1}^{\infty}
			\frac{\xi^{2}\log\xi}{\sinh^{2}(\pi\xi/\sqrt2)}\,d\xi.
		\end{aligned}
	\end{equation*}
	Convexity of $\sinh$, together with $\sinh 0=0$, gives $\sinh(\pi\xi/\sqrt2)\le\xi\sinh(\pi/\sqrt2)$ for $0\le\xi\le1$. Moreover, $\pi/\sqrt2<9/4$ and $e^{9/4}<10$ imply $\sinh(\pi/\sqrt2)<5$. Therefore
	\begin{equation*}
		\begin{aligned}
			A_{0}
			 & \ge\frac{1}{\sinh^{2}(\pi/\sqrt2)}
			\int_{0}^{1}|\log\xi|\,d\xi
			>\frac1{25}.
		\end{aligned}
	\end{equation*}

	For $\xi\ge1$, we use
	\begin{equation*}
		\sinh\left(\frac{\pi\xi}{\sqrt2}\right)
		=\frac12e^{\pi\xi/\sqrt2}\bigl(1-e^{-\sqrt2\pi\xi}\bigr)
		\ge\frac{49}{100}e^{\pi\xi/\sqrt2},
	\end{equation*}
	because $\sqrt2\pi\xi>4$ and $e^{-4}<1/50$. Since $\sqrt2\pi>21/5$ (because $\sqrt2>7/5$ and $\pi>3$) and $\log\xi\le\xi-1$, the substitution $t=\xi-1$ yields
	\begin{equation*}
		\begin{aligned}
			2\pi A_{\infty}
			 & \le2\pi\left(\frac{100}{49}\right)^{2}
			\int_{1}^{\infty}\xi^{2}(\xi-1)e^{-21\xi/5}\,d\xi \\
			 & =2\pi\left(\frac{100}{49}\right)^{2}e^{-21/5}
			\left[
				\left(\frac5{21}\right)^{2}
				+4\left(\frac5{21}\right)^{3}
				+6\left(\frac5{21}\right)^{4}
			\right]                                           \\
			 & =0.0509\ldots<0.051.
		\end{aligned}
	\end{equation*}
	It follows, using $\pi>3$, that
	\begin{equation*}
		\kappa_{1}>\frac{2\pi}{25}-\frac{51}{1000}
		>\frac{189}{1000}>\frac{119}{900}.
	\end{equation*}

	For the second term of~\eqref{eq:kappa-def}, $\Jop\psi=-g$ gives $\psi=-\Jop^{-1}g$ and hence $\langle\Jop\psi,\psi\rangle=\langle g,\Jop^{-1}g\rangle$; since $\Jop\ge3/2$ on $\Lodd(\R)$, the spectral theorem and~\eqref{eq:g-bounds} give
	\begin{equation*}
		0\le\frac12\langle\Jop\psi,\psi\rangle
		\le\frac13\|g\|_{L^{2}}^{2}
		\le \frac13\left(\frac56\right)^{2}
		=\frac{25}{108}<\frac12 .
	\end{equation*}
	For $\mu_{\Phi_1}$, convexity of $\sinh$ together with $\sinh0=0$ gives $\sinh(\pi\xi/\sqrt2)\le\xi\sinh(\pi/\sqrt2)$ for $0\le\xi\le1$, so that,
	\begin{equation*}
		\mu_{\Phi_1}\ge2\pi\int_{0}^{1}
		\frac{\xi^{2}(\log\xi)^{2}}{\xi^{2}\sinh^{2}(\pi/\sqrt2)}\,d\xi
		= \frac{2\pi}{\sinh^{2}(\pi/\sqrt2)} \int_{0}^{1}(\log\xi)^{2}\,d\xi
		=\frac{4\pi}{\sinh^{2}(\pi/\sqrt2)} .
	\end{equation*}
	Now $\pi/\sqrt2<9/4$ and $e^{9/4}<10$ give $\sinh(\pi/\sqrt2)<\tfrac12 e^{\pi/\sqrt2}<5$, hence $\sinh^{2}(\pi/\sqrt2)<25$ and $\mu_{\Phi_1}>4\pi/25>1/2$, the last step because $\pi>25/8$. Combining these estimates, we conclude $\kappa_{2}=\mu_{\Phi_1}-\tfrac12\langle\Jop\psi,\psi\rangle>0$.
\end{proof}

\subsection{The layer near the approximate profile}\label{ss:s1-fixed-point}

Let $\F_{s}(u):=\fl u+W'(u)$ be the Allen--Cahn residual, so that $\Phis$ is characterized by $\F_{s}(\Phis)=0$. We show in this subsection that $\Phis$ lies within $O(\varepsilon^{2})$ of $u_{s}$ in $H^{2s}$. We will do this by a fixed-point argument. Throughout we use the comparison
\begin{equation}\label{eq:mult-comp}
	\|w\|_{H^{2s}}^{2}\;\le\;2\bigl(\|w\|_{L^{2}}^{2}+\|\fl w\|_{L^{2}}^{2}\bigr)
	\qquad(s\in[49/50,1]),
\end{equation}
which follows from Plancherel and the bound $(1+|\xi|^{2})^{2s}\le2^{2s-1}(1+|\xi|^{4s})\le2(1+|\xi|^{4s})$.

The fixed point argument is organized in the next three lemmas as follows: first, we bound the error of $u_s$ as an approximate solution to the Allen--Cahn equation, second, we invert the corresponding linearized operator and finally we go through the nonlinear estimates needed to apply Banach's fixed-point theorem.

We use the expansion of the fractional Laplacian
\begin{equation}\label{eq:frac-laplacian-exp}
	\fl \Phi_1=(-\Delta)\Phi_1+\varepsilon\Lone \Phi_1+\rho_{s}^{(\Phi_1)}
	\quad\text{and}\quad
	\fl\psi=(-\Delta)\psi+\varepsilon\Lone\psi+\rho_{s}^{(\psi)},
\end{equation}
where $\rho_{s}^{(\Phi_1)}$ and $\rho_{s}^{(\psi)}$ are the second-order Taylor remainders of the symbol $|\xi|^{2s}$ acting on $\Phi_1$ and $\psi$, respectively.

\begin{lemma}\label{lem:residual}
	For $s\in[49/50,1)$ set
	\begin{equation}\label{eq:Rs-def}
		R_{s} \;:=\; -\bigl(\rho_{s}^{(\Phi_1)} + \varepsilon\,\rho_{s}^{(\psi)}
		+ \varepsilon^{2}\Lone\psi + 3\Phi_1\,\varepsilon^{2}\psi^{2}
		+ \varepsilon^{3}\psi^{3}\bigr).
	\end{equation}
	Then $\F_{s}(u_{s})=-R_{s}$, and $R_{s}\in C_{0,\mathrm{odd}}(\R)\cap L^{2}(\R)$ with
	\begin{equation}\label{eq:Rs-L2}
		\|R_{s}\|_{L^{2}}\;\le\;\tfrac{15}{2}\,\varepsilon^{2}
		\qquad(0\le\varepsilon\le\tfrac1{50}).
	\end{equation}
\end{lemma}

\begin{proof}
	\emph{Step 1: the residual identity.}  Insert
	$u_{s}=\Phi_1+\varepsilon\psi$ into $\F_{s}$, using the symbol expansions in
	\eqref{eq:frac-laplacian-exp}
	together with the exact Taylor expansion of the cubic $W'$,
	\begin{equation*}
		W'(\Phi_1+\varepsilon\psi)=W'(\Phi_1)+\varepsilon W''(\Phi_1)\psi
		+3\Phi_1\,\varepsilon^{2}\psi^{2}+\varepsilon^{3}\psi^{3}.
	\end{equation*}
	Grouping by powers of $\varepsilon$,
	\begin{equation*}
		\F_{s}(u_{s})
		=\underbrace{\bigl[(-\Delta)\Phi_1+W'(\Phi_1)\bigr]}_{=\,0}
		+\varepsilon\,\underbrace{\bigl[\Lone \Phi_1+\Jop\psi\bigr]}_{=\,0}
		+\rho_{s}^{(\Phi_1)}+\varepsilon\rho_{s}^{(\psi)}+\varepsilon^{2}\Lone\psi
		+3\Phi_1\,\varepsilon^{2}\psi^{2}+\varepsilon^{3}\psi^{3},
	\end{equation*}
	the first bracket vanishing because $\Phi_1$ solves the classical equation and the second by the corrector equation~\eqref{eq:corrector-eq}.

	\smallskip
	\noindent\emph{Step 2: the multiplier remainder.} First we control the error in the Taylor
	expansion of $(-\Delta)^{s}$. Since
	$|\xi|^{2s}=|\xi|^{2}e^{-2\varepsilon\log|\xi|}$, the Taylor bound
	$|e^{z}-1-z|\le\tfrac12 z^{2}e^{|z|}$ applied with
	$z=-2\varepsilon\log|\xi|$, gives the estimate
	\begin{equation}\label{eq:symbol-rem2}
		\bigl|\,|\xi|^{2s}-|\xi|^{2}+2\varepsilon|\xi|^{2}\log|\xi|\,\bigr|
		\;\le\; 2\varepsilon^{2}|\xi|^{2}\bigl|\log|\xi|\bigr|^{2}
		e^{2\varepsilon|\log|\xi||}.
	\end{equation}
	By Plancherel and \eqref{eq:symbol-rem2},
	\begin{equation*}
		\|\rho_{s}^{(\Phi_1)}\|_{L^{2}}^{2}
		\le\frac{4\varepsilon^{4}}{2\pi}\int_{\R}
		\xi^{4}\bigl|\log|\xi|\bigr|^{4}e^{4\varepsilon|\log|\xi||}
		|\widehat{\Phi_1}(\xi)|^{2}\,d\xi
		= 8\pi \varepsilon^{4}\int_{0}^{\infty}
		\frac{\xi^{4}(\log\xi)^{4}e^{4\varepsilon|\log\xi|}}
		{\sinh^{2}(\pi\xi/\sqrt2)}\,d\xi,
	\end{equation*}
	using $|\widehat{\Phi_1}(\xi)|^{2} =2\pi^{2}/\sinh^{2}(\pi|\xi|/\sqrt2)$ and evenness. To estimate the last integral, we split it at $\xi=1$. On $(0,1)$, the Taylor series of $\sinh$ gives $\sinh(\pi\xi/\sqrt2)\ge\pi\xi/\sqrt2$, and $e^{4\varepsilon|\log\xi|}=\xi^{-4\varepsilon}$. On $(1,\infty)$, we use
	\begin{equation*}
		\sinh\left(\frac{\pi\xi}{\sqrt{2}}\right)
		= \frac12 e^{\pi\xi/\sqrt{2}} \left(1-e^{-\sqrt{2}\pi\xi}\right)
		\ge\tfrac{49}{100}e^{\pi\xi/\sqrt2}
	\end{equation*}
	which holds since $\sqrt{2}\pi\xi>4$ and $e^{-4}<1/50$. Moreover, $e^{4\varepsilon|\log\xi|}=\xi^{4\varepsilon}\le\xi$, and $\log\xi\le\xi/e$ for $\xi\ge1$. Thus,
	\begin{equation*}
		\begin{aligned}
			 & 8\pi\int_{0}^{\infty}
			\frac{\xi^{4}(\log\xi)^{4}e^{4\varepsilon|\log\xi|}}
			{\sinh^{2}(\pi\xi/\sqrt2)}\,d\xi                         \\
			 & \quad\le \frac{16}{\pi}\int_{0}^{1}
			\xi^{2-4\varepsilon}|\log\xi|^{4}\,d\xi
			+\frac{8\pi}{e^{4}}\Bigl(\frac{100}{49}\Bigr)^{2}
			\int_{1}^{\infty}\xi^{9}e^{-\sqrt2\pi\xi}\,d\xi          \\
			 & \quad\le\frac{16}{\pi}\frac{4!}{(3-4\varepsilon)^{5}}
			+\frac{8\pi}{e^{4}}\Bigl(\frac{100}{49}\Bigr)^{2}
			\frac{9!}{(\sqrt2\pi)^{10}}                              \\
			 & \quad<0.576+0.233<1.
		\end{aligned}
	\end{equation*}
	Here we used the substitution $\xi=e^{-t}$ in the first integral. In the second one, we enlarged the domain of integration to $(0,\infty)$ and used $\int_{0}^{\infty}t^{n}e^{-ct}\,dt=n!/c^{n+1}$. The bounds in the last line also use $\varepsilon\le1/50$. It follows that
	\begin{equation*}
		\|\rho_{s}^{(\Phi_1)}\|_{L^{2}}
		\le\varepsilon^{2}.
	\end{equation*}

	Again, by~\eqref{eq:symbol-rem2},
	\begin{equation*}
		\begin{aligned}
			\|\rho_{s}^{(\psi)}\|_{L^{2}}
			 & \le2\varepsilon^{2}\Biggl(\frac{1}{2\pi}\int_{\R}
			\xi^{4}\bigl|\log|\xi|\bigr|^{4}e^{4\varepsilon|\log|\xi||}
			|\widehat\psi(\xi)|^{2}\,d\xi\Biggr)^{1/2}           \\
			 & \le2\varepsilon^{2}\,\Biggl(\sup_{\xi\ne0}
			\frac{\xi^{4}\bigl|\log|\xi|\bigr|^{4}e^{4\varepsilon|\log|\xi||}}
			{(1+\xi^{2})^{3}}\Biggr)^{1/2}\|\psi\|_{H^{3}} .
		\end{aligned}
	\end{equation*}
	We now use the elementary identities
	\begin{equation}\label{eq:sup-texp}
		\sup_{t\ge0}t^{4}e^{-at}=\left(\frac{4}{ae}\right)^{4}
		\quad\text{and}\quad
		\sup_{t\ge0}te^{-at}=\frac{1}{ae}
		\qquad(a>0).
	\end{equation}
	If $|\xi|\ge1$, then $(1+\xi^{2})^{3}\ge\xi^{6}$ hence,
	\begin{equation*}
		\frac{\xi^{4}\bigl|\log|\xi|\bigr|^{4}e^{4\varepsilon|\log|\xi||}}{(1+\xi^{2})^{3}}
		\le (\log|\xi|)^{4}e^{-2(1-2\varepsilon)\log|\xi|}
		\le\left(\frac{2}{e(1-2\varepsilon)}\right)^{4}<1.
	\end{equation*}
	If $0<|\xi|\le1$, then $(1+\xi^{2})^{3}\ge1$, so, using the first identity with $t=-\log|\xi|$,
	\begin{equation*}
		\frac{\xi^{4}\bigl|\log|\xi|\bigr|^{4}e^{4\varepsilon|\log|\xi||}}{(1+\xi^{2})^{3}}
		\le (\log|\xi|)^{4}e^{4(1-\varepsilon)\log|\xi|}
		\le\left(\frac{1}{e(1-\varepsilon)}\right)^{4}<1.
	\end{equation*}
	Thus the supremum is at most $1$, and $\|\psi\|_{H^{3}}\le6$ yields
	\begin{equation*}
		\|\rho_{s}^{(\psi)}\|_{L^{2}}
		\le2\varepsilon^{2}\|\psi\|_{H^{3}}\le12\varepsilon^{2}.
	\end{equation*}

	\smallskip
	\noindent\emph{Step 3: the remaining three terms.}  For $\Lone\psi$, whose
	symbol is $-2|\xi|^{2}\log|\xi|$, we bound that symbol against
	$(1+\xi^{2})^{3/2}$.  For $|\xi|\ge1$, using
	$(1+\xi^{2})^{3/2}\ge|\xi|^{3}$ and the second identity in \eqref{eq:sup-texp} with
	$t=\log|\xi|$ and $a=1$,
	\begin{equation*}
		\frac{2\xi^{2}\bigl|\log|\xi|\bigr|}{(1+\xi^{2})^{3/2}}
		\le\frac{2\log|\xi|}{|\xi|}\le\frac2e ;
	\end{equation*}
	for $0<|\xi|\le1$, using $(1+\xi^{2})^{3/2}\ge1$ and the same identity with $t=-\log|\xi|$ and $a=2$,
	\begin{equation*}
		\frac{2\xi^{2}\bigl|\log|\xi|\bigr|}{(1+\xi^{2})^{3/2}}
		\le2\xi^{2}\bigl|\log|\xi|\bigr|\le\frac{2}{2e}=\frac1e .
	\end{equation*}
	Hence
	\begin{equation*}
		\|\Lone\psi\|_{L^{2}}\le\frac{2}{e}\,\|\psi\|_{H^{3}}\le\frac{12}{e}\le\frac92 .
	\end{equation*}
	Next, $\|\Phi_1\|_{L^{\infty}}\le1$ and Lemma~\ref{lem:corrector-bounds} give
	\begin{equation*}
		\|3\Phi_1\psi^{2}\|_{L^{2}}\le3\|\psi\|_{L^{\infty}}\|\psi\|_{L^{2}}
		\le3\cdot\frac45\cdot\frac59=\frac43,
		\qquad
		\|\psi^{3}\|_{L^{2}}\le\|\psi\|_{L^{\infty}}^{2}\|\psi\|_{L^{2}}
		\le\left(\frac45\right)^{2}\frac59=\frac{16}{45}\le\frac25.
	\end{equation*}

	\smallskip
	\noindent\emph{Step 4: conclusion.}  By the triangle inequality
	in~\eqref{eq:Rs-def}, and using $\varepsilon\le1/50$ on the two terms
	carrying a spare power of $\varepsilon$,
	\begin{equation*}
		\|R_{s}\|_{L^{2}}
		\le\Bigl(1+\frac1{50}\cdot12+\frac92+\frac43
		+\frac1{50}\cdot\frac25\Bigr)\varepsilon^{2}
		=\frac{5311}{750}\,\varepsilon^{2}<\frac{15}{2}\varepsilon^{2},
	\end{equation*}
	which is~\eqref{eq:Rs-L2}. Finally, each of the five terms in~\eqref{eq:Rs-def} lies in $C_{0,\mathrm{odd}}(\R)$. Indeed, the two Taylor remainders lie in this space because their symbols are $L^{1}$ against $\widehat{\Phi_1}$ and $\widehat\psi$ by the bounds just proved; $\Lone\psi$ does so by Lemma~\ref{lem:corrector}; and the last two do so because $\Phi_1$ is bounded and $\psi\in C_{0}$.
\end{proof}

The next step in the fixed point argument is to invert the operator linearized around the approximate profile $u_s$. This is done in the following lemma.

\begin{lemma}\label{lem:uniform-invert}
	For $s\in[49/50,1)$, the operator $\Lop:=\fl+W''(u_{s})\colon H^{2s}_{\mathrm{odd}}(\R)\to\Lodd(\R)$ is a bounded isomorphism. Its inverse $\Gs:=\Lop^{-1}$ satisfies
	\begin{equation}\label{eq:Gs-bound}
		\|\Gs\|_{L^{2}\to L^{2}}\le1,
		\qquad
		\|\Gs\|_{L^{2}\to H^{2s}}\le M:=5 .
	\end{equation}
\end{lemma}

\begin{proof}
	\emph{Step 1: self-adjointness and the essential spectrum.}  Since
	$W''(u_{s})$ is real-valued and bounded by~\eqref{eq:us-bounds},
	$\Lop$ is self-adjoint on $L^{2}(\R)$
	with domain $H^{2s}(\R)$.  Moreover, $u_{s}$ is odd and
	$W''(u_{s})=3u_{s}^{2}-1$ is even.  Thus $\Lop$ preserves the odd
	subspace, and its restriction to $\Lodd(\R)$ is self-adjoint with domain
	$H^{2s}_{\mathrm{odd}}(\R)$.  The estimate~\eqref{eq:us-bounds} also gives,
	for every $v\in H^{2s}_{\mathrm{odd}}(\R)$,
	\begin{equation*}
		\|\Lop v\|_{L^{2}}
		\le\|\fl v\|_{L^{2}}+\frac{107}{50}\|v\|_{L^{2}}
		\le\frac{157}{50}\|v\|_{H^{2s}},
	\end{equation*}
	so $\Lop\colon H^{2s}_{\mathrm{odd}}(\R)\to\Lodd(\R)$ is bounded.

	Since $u_{s}(x)\to\pm1$ as $x\to\pm\infty$,
	\begin{equation*}
		W''(u_{s})-2=3(u_{s}^{2}-1)\in C_{0}(\R).
	\end{equation*}
	Multiplication by $W''(u_{s})-2$, composed with $(\fl+1)^{-1}\colon L^{2}(\R)\to H^{2s}(\R)$, is compact. Indeed, after restricting the multiplier to a bounded interval, this follows from the compact embedding of $H^{2s}$ into $L^{2}$ on bounded intervals, while the operator norm of the remaining tail tends to zero because
	\begin{equation*}
		\sup_{|x|\ge R}|W''(u_{s}(x))-2|\longrightarrow0
		\qquad\text{as }R\to\infty.
	\end{equation*}
	The same argument applies on the odd subspace. The odd restriction of $\fl+2$ has spectrum $[2,+\infty)$, so Weyl's theorem on the stability of the essential spectrum under relatively compact perturbations~\cite[Ch.~IV, Thm.~5.35]{Kato1976} implies that the essential spectrum of the odd restriction of $\Lop$ is also $[2,+\infty)$. Consequently, any spectral point of this restriction in $[-1,1]$ must be an eigenvalue with an eigenfunction in $H^{2s}_{\mathrm{odd}}(\R)$.

	\smallskip
	\noindent\emph{Step 2: exclusion of $[-1,1]$.}  Suppose
	$\Lop w=\lambda w$ for some $w\in H^{2s}_{\mathrm{odd}}(\R)$ satisfying
	$\|w\|_{L^{2}}=1$ and some $\lambda\in[-1,1]$.  Since $2s>1$, the
	function $w$ belongs to $H^{1}(\R)$, so the quadratic form
	$\langle\Jop w,w\rangle$ is well defined, and the spectral
	gap~\eqref{eq:J-odd-spectrum} gives
	$\langle\Jop w,w\rangle\ge\tfrac32$.  On the other hand, adding and
	subtracting the quadratic forms of $\Lop$ and $\Jop$, we obtain
	\begin{equation}\label{eq:J-vs-L}
		\langle\Jop w,w\rangle
		=\lambda
		+\frac{1}{2\pi}\int_{\R}\bigl(|\xi|^{2}-|\xi|^{2s}\bigr)|\widehat w(\xi)|^{2}\,d\xi
		+\int_{\R}\bigl(W''(\Phi_1)-W''(u_{s})\bigr)|w|^{2}\,dx .
	\end{equation}
	We estimate the two error terms in~\eqref{eq:J-vs-L}. The multiplier $|\xi|^{2}-|\xi|^{2s}$ is non-positive when $|\xi|<1$. When $|\xi|\ge1$, we have
	\begin{equation*}
		\frac{|\xi|^{2}-|\xi|^{2s}}{|\xi|^{4s}}
		=|\xi|^{-2+4\varepsilon}-|\xi|^{-2+2\varepsilon}.
	\end{equation*}
	The expression on the right attains its maximum when $ |\xi|^{2\varepsilon}=(1-\varepsilon)/(1-2\varepsilon) $, and hence
	\begin{equation*}
		|\xi|^{2}-|\xi|^{2s}
		\le
		\frac{\varepsilon}{1-\varepsilon}
		\Bigl(\frac{1-2\varepsilon}{1-\varepsilon}\Bigr)^{(1-2\varepsilon)/\varepsilon}
		|\xi|^{4s}.
	\end{equation*}
	Moreover,
	\begin{equation*}
		\log\Bigl(\frac{1-2\varepsilon}{1-\varepsilon}\Bigr)
		=\log\Bigl(1-\frac{\varepsilon}{1-\varepsilon}\Bigr)
		\le-\frac{\varepsilon}{1-\varepsilon}.
	\end{equation*}
	Since $\varepsilon\le1/50$, this implies
	\begin{equation*}
		\frac{\varepsilon}{1-\varepsilon}
		\Bigl(\frac{1-2\varepsilon}{1-\varepsilon}\Bigr)^{(1-2\varepsilon)/\varepsilon}
		\le\frac1{49}e^{-48/49}<\frac1{100}.
	\end{equation*}
	Since $\fl w=\lambda w-W''(u_{s})w$, it follows from~\eqref{eq:us-bounds} that
	\begin{equation*}
		\|\fl w\|_{L^{2}}
		\le|\lambda|\|w\|_{L^{2}}
		+\|W''(u_{s})\|_{L^{\infty}}\|w\|_{L^{2}}
		\le1+\frac{107}{50}=\frac{157}{50}.
	\end{equation*}
	Plancherel's identity now yields
	\begin{equation*}
		\frac{1}{2\pi}\int_{\R}\bigl(|\xi|^{2}-|\xi|^{2s}\bigr)|\widehat w|^{2}\,d\xi
		\le\frac1{100}\,\|\fl w\|_{L^{2}}^{2}
		\le\frac1{100}\Bigl(\frac{157}{50}\Bigr)^{2}\le\frac{1}{10}.
	\end{equation*}

	For the potential term, the identity
	\begin{equation*}
		W''(u_{s})-W''(\Phi_1)
		=3(u_{s}^{2}-\Phi_1^{2})
		=3\varepsilon(u_{s}+\Phi_1)\psi
	\end{equation*}
	and the bounds~\eqref{eq:us-bounds} and Lemma~\ref{lem:corrector-bounds} give
	\begin{equation*}
		\|W''(u_{s})-W''(\Phi_1)\|_{L^{\infty}}
		\le3\varepsilon\bigl(\|u_{s}\|_{L^{\infty}}+1\bigr)\|\psi\|_{L^{\infty}}
		\le3\cdot\frac1{50}\cdot\frac{101}{50}\cdot\frac45\le\frac1{10}.
	\end{equation*}
	Since $\|w\|_{L^{2}}=1$, it follows that
	\begin{equation*}
		\int_{\R}\bigl(W''(\Phi_1)-W''(u_{s})\bigr)|w|^{2}\,dx
		\le\frac1{10}.
	\end{equation*}
	Substituting these estimates into~\eqref{eq:J-vs-L}, and using $\lambda\le1$, we find
	\begin{equation*}
		\langle\Jop w,w\rangle\le 1+\frac1{10}+\frac1{10}=\frac65<\frac32,
	\end{equation*}
	which contradicts the spectral gap of $\Jop$. Hence the spectrum of the odd restriction of $\Lop$ does not intersect $[-1,1]$.

	\smallskip
	\noindent\emph{Step 3: the inverse bounds.}  The spectral exclusion proved
	above gives
	\begin{equation*}
		\operatorname{dist}\bigl(0,\sigma(\Lop|_{\Lodd(\R)})\bigr)\ge1.
	\end{equation*}
	The spectral theorem therefore shows that $0$ belongs to the resolvent set of the odd restriction of $\Lop$ and that
	\begin{equation*}
		\|\Gs\|_{L^{2}\to L^{2}}\le1.
	\end{equation*}
	In particular, $\Lop$ maps $H^{2s}_{\mathrm{odd}}(\R)$ bijectively onto $\Lodd(\R)$. For $f\in\Lodd(\R)$, the identity $\Lop\Gs f=f$ gives
	\begin{equation*}
		\|\fl\Gs f\|_{L^{2}}
		\le\|f\|_{L^{2}}+\frac{107}{50}\|\Gs f\|_{L^{2}}
		\le\frac{157}{50}\|f\|_{L^{2}}.
	\end{equation*}
	Using~\eqref{eq:mult-comp}, we conclude that
	\begin{equation*}
		\|\Gs f\|_{H^{2s}}
		\le\sqrt2\,\Bigl(\|\Gs f\|_{L^{2}}^{2}
		+\|\fl\Gs f\|_{L^{2}}^{2}\Bigr)^{1/2}
		\le\sqrt2\,\Bigl(1+\bigl(\tfrac{157}{50}\bigr)^{2}\Bigr)^{1/2}
		\|f\|_{L^{2}}
		<5\|f\|_{L^{2}},
	\end{equation*}
	which proves~\eqref{eq:Gs-bound} and completes the proof.
\end{proof}

The next lemma shows that $\Phi_s$ is a perturbation of $u_s$ and gives a quantitative bound for the error using a fixed point argument.

\begin{lemma}\label{lem:fixed-point}
	For every $s\in[49/50,1)$,
	\begin{equation*}
		\Phis \;=\; u_{s} + r_{s},\qquad r_{s}\in H^{2s}_{\mathrm{odd}}(\R),
		\qquad
		\|r_{s}\|_{H^{2s}(\R)}\;\le\; 75\,\varepsilon^{2}.
	\end{equation*}
\end{lemma}

\begin{proof}
	We first reformulate the equation for $r_s$ as a fixed point problem. Since $W'$ is a cubic, the exact Taylor expansion $W'(u_{s}+r)=W'(u_{s})+W''(u_{s})r+3u_{s}r^{2}+r^{3}$ yields, for $r\in H^{2s}_{\mathrm{odd}}(\R)$,
	\begin{equation*}
		\F_{s}(u_{s}+r)=\F_{s}(u_{s})+\Lop r-N_{s}(r),
		\qquad
		N_{s}(r):=-3u_{s}r^{2}-r^{3}.
	\end{equation*}
	Since $\F_{s}(u_{s})=-R_{s}$, solving $\F_{s}(u_{s}+r)=0$ is equivalent to $\Lop r=R_{s}+N_{s}(r)$. Applying $\Gs$ leads to the fixed-point equation
	\begin{equation*}
		r=T_{s}(r),
		\qquad
		T_{s}(r):=\Gs\bigl(R_{s}+N_{s}(r)\bigr).
	\end{equation*}
	These two formulations are equivalent because $\Lop$ is an isomorphism $H^{2s}_{\mathrm{odd}}\to\Lodd$ and its inverse is $\Gs$.

	We now prove that the operator $T_s$ is a contraction on a suitable ball. By Lemma~\ref{lem:sobolev-constants}, $\|w\|_{L^{\infty}}\le C^{\ast}(2s)\|w\|_{H^{2s}}$. Since $\sigma\mapsto C^{\ast}(\sigma)$ is decreasing and $2s\ge49/25$,
	\begin{equation}\label{eq:Cstar-2s-value}
		C^{\ast}(2s)\le C^{\ast}\left(\frac{49}{25}\right)
		=\Bigl(\frac{\Gamma(\tfrac{73}{50})}{2\sqrt\pi\,\Gamma(\tfrac{49}{25})}\Bigr)^{1/2}
		=0.5039\ldots\le\frac{51}{100}
		\qquad(s\in\left[\frac{49}{50},1\right]).
	\end{equation}
	Set $\delta:=75\varepsilon^{2}$. On the closed ball $B_{\delta}:=\{r\in H^{2s}_{\mathrm{odd}}:\|r\|_{H^{2s}}\le\delta\}$, the algebraic identity
	\begin{equation*}
		N_{s}(r_{1})-N_{s}(r_{2})
		=-3u_{s}(r_{1}+r_{2})(r_{1}-r_{2})-(r_{1}^{2}+r_{1}r_{2}+r_{2}^{2})(r_{1}-r_{2}),
	\end{equation*}
	together with $\|r_{1}^{2}+r_{1}r_{2}+r_{2}^{2}\|_{L^{\infty}} \le(\|r_{1}\|_{L^{\infty}}+\|r_{2}\|_{L^{\infty}})^{2}$ and $\|r_{1}\|_{H^{2s}}+\|r_{2}\|_{H^{2s}}\le2\delta$, gives
	\begin{equation*}
		\|N_{s}(r_{1})-N_{s}(r_{2})\|_{L^{2}}
		\le C_{N}\bigl(\|r_{1}\|_{H^{2s}}+\|r_{2}\|_{H^{2s}}\bigr)\|r_{1}-r_{2}\|_{H^{2s}},
	\end{equation*}
	with, by~\eqref{eq:us-bounds} and~\eqref{eq:Cstar-2s-value},
	\begin{equation*}
		C_{N}:=3\|u_{s}\|_{L^{\infty}}C^{\ast}(2s)+2\bigl(C^{\ast}(2s)\bigr)^{2}\delta
		\le3\cdot\frac{51}{50}\cdot\frac{51}{100}+2\left(\frac{51}{100}\right)^{2}\cdot\frac{3}{100}
		\le\frac85 .
	\end{equation*}
	Composing with $\Gs$ and using~\eqref{eq:Gs-bound} and~\eqref{eq:Rs-L2},
	\begin{equation*}
		\|\Gs R_{s}\|_{H^{2s}}\le M\|R_{s}\|_{L^{2}}\le5\cdot\frac{15}{2}\varepsilon^{2}
		=\frac{\delta}{2},
		\qquad
		\|T_{s}(r_{1})-T_{s}(r_{2})\|_{H^{2s}}\le 2MC_{N}\delta\,\|r_{1}-r_{2}\|_{H^{2s}} .
	\end{equation*}
	Since $\varepsilon\le1/50$,
	\begin{equation*}
		2MC_{N}\delta\le2\cdot5\cdot\frac85\cdot75\,\varepsilon^{2}
		=1200\,\varepsilon^{2}\le\frac{1200}{2500}<\frac12,
	\end{equation*}
	the map $T_{s}$ is a contraction of $B_{\delta}$. By the Banach fixed-point theorem there is a unique $r_{s}\in B_{\delta}$ with $r_{s}=T_{s}(r_{s})$, so $\Psi_{s}:=u_{s}+r_{s}$ solves $\fl\Psi_{s}+W'(\Psi_{s})=0$ in $L^{2}(\R)$ and $\|r_{s}\|_{H^{2s}}\le\delta=75\varepsilon^{2}$.

	Finally we identify $\Psi_s$ with $\Phi_s$. Since $2s>1$, the embedding $H^{2s}(\R)\hookrightarrow C_{0}(\R)$ gives $r_{s}\in C_{0}$; as $u_{s}=\Phi_1+\varepsilon\psi$ with $\Phi_1(\pm\infty)=\pm1$ and $\psi\in C_{0}$, we get $\Psi_{s}(\pm\infty)=\pm1$, while oddness forces $\Psi_{s}(0)=0$. The bounded function $\Psi_{s}$ solves $\fl\Psi_{s}=\Psi_{s}-\Psi_{s}^{3}$ distributionally with right-hand side in $L^{\infty}$, so the interior regularity theory for the fractional Laplacian~\cite{CabreSire2014I,CabreSire2015II} upgrades it to a classical solution in $C^{\infty}(\R)$. Thus $\Psi_{s}$ is an odd, smooth solution of the layer equation with the correct limits and normalization, and by uniqueness in Theorem~\ref{thm:CS} we have $\Psi_{s}=\Phis$.
\end{proof}

\subsection{The derivative expansion}\label{ss:s1-expansion}

We now expand the derivative functional $J$ of~\eqref{eq:J-functional} at the approximate profile, transfer the expansion to the layer, and record the endpoint limit.
\begin{lemma}\label{lem:J-trial}
	For $0\le\varepsilon\le1/50$,
	\begin{equation*}
		\bigl|J(s,u_{s})+\kappa_{1}+2\kappa_{2}\varepsilon\bigr|
		\;\le\; 140\,\varepsilon^{2}.
	\end{equation*}
\end{lemma}

\begin{proof}
	\emph{Step 1: identification of the constant and linear terms.}
	Since $s=1-\varepsilon$, the definition of $J$ gives
	\begin{equation*}
		J(s,u_s)=\frac{1}{2\pi}\int_{\R}\log|\xi|\,|\xi|^{-2\varepsilon}
		\bigl|\widehat{\Phi_1'}+\varepsilon\widehat{\psi'}\bigr|^2\,d\xi .
	\end{equation*}
	The square in the integrand satisfies
	\begin{equation*}
		\bigl|\widehat{\Phi_1'}(\xi)+\varepsilon\widehat{\psi'}(\xi)\bigr|^{2}
		=|\widehat{\Phi_1'}(\xi)|^{2}
		+2\varepsilon\operatorname{Re}\bigl(\widehat{\Phi_1'}(\xi)\overline{\widehat{\psi'}(\xi)}\bigr)
		+\varepsilon^{2}|\widehat{\psi'}(\xi)|^{2},
	\end{equation*}
	while the multiplier has the pointwise expansion $|\xi|^{-2\varepsilon}=1-2\varepsilon\log|\xi|+O(\varepsilon^2)$. Thus the constant term is $-\kappa_1$, and the linear contribution is
	\begin{equation}\label{eq:order-eps-term}
		\frac{\varepsilon}{2\pi}\int_{\R}\log|\xi|
		\Bigl[2\operatorname{Re}\bigl(\widehat{\Phi_1'}\overline{\widehat{\psi'}}\bigr)
		-2\log|\xi|\,|\widehat{\Phi_1'}|^{2}\Bigr]\,d\xi .
	\end{equation}
	To identify the cross term, we use $\operatorname{Re}(\widehat{\Phi_1'}\overline{\widehat{\psi'}}) =\xi^{2}\operatorname{Re}(\widehat{\Phi_1}\overline{\widehat\psi})$. Hence~\eqref{eq:ghat}, Plancherel, and the corrector equation~\eqref{eq:corrector-eq} yield
	\begin{equation*}
		\frac{1}{\pi}\int_{\R}\log|\xi|\;\xi^{2}\operatorname{Re}\bigl(\widehat{\Phi_1}\overline{\widehat\psi}\bigr)\,d\xi
		=-\langle g,\psi\rangle_{L^{2}}
		=\langle\Jop\psi,\psi\rangle_{L^{2}},
	\end{equation*}
	whereas the second term in~\eqref{eq:order-eps-term} is $-2\mu_{\Phi_1}\varepsilon$. By~\eqref{eq:kappa-def}, the expression in~\eqref{eq:order-eps-term} therefore equals
	\begin{equation*}
		\varepsilon\bigl(\langle\Jop\psi,\psi\rangle_{L^{2}}
		-2\mu_{\Phi_1}\bigr)=-2\kappa_2\varepsilon .
	\end{equation*}

	\smallskip
	\noindent\emph{Step 2: the quadratic remainder.}
	Subtracting the constant and linear terms gives the exact identity
	\begin{equation*}
		\begin{aligned}
			J(s,u_{s})+\kappa_{1}+2\kappa_{2}\varepsilon
			 & =\frac1{2\pi}\int_{\R}\log|\xi|\,\Bigl[
			\bigl(|\xi|^{-2\varepsilon}-1+2\varepsilon\log|\xi|\bigr)
			|\widehat{\Phi_1'}|^{2}                                                  \\
			 & \qquad+2\varepsilon\bigl(|\xi|^{-2\varepsilon}-1\bigr)
			\operatorname{Re}\bigl(\widehat{\Phi_1'}\overline{\widehat{\psi'}}\bigr) \\
			 & \qquad+\varepsilon^{2}|\xi|^{-2\varepsilon}|\widehat{\psi'}|^{2}
			\Bigr]\,d\xi .
		\end{aligned}
	\end{equation*}
	The Taylor bounds $|e^{z}-1-z|\le\tfrac12 z^{2}e^{|z|}$ and $|e^{z}-1|\le|z|e^{|z|}$ with $z=-2\varepsilon\log|\xi|$ give
	\begin{equation*}
		\begin{aligned}
			\bigl||\xi|^{-2\varepsilon}-1+2\varepsilon\log|\xi|\bigr|
			 & \le2\varepsilon^{2}\bigl|\log|\xi|\bigr|^{2}
			e^{2\varepsilon|\log|\xi||},                    \\
			\bigl||\xi|^{-2\varepsilon}-1\bigr|
			 & \le2\varepsilon\bigl|\log|\xi|\bigr|
			e^{2\varepsilon|\log|\xi||},                    \\
			|\xi|^{-2\varepsilon}
			 & \le e^{2\varepsilon|\log|\xi||}.
		\end{aligned}
	\end{equation*}
	Moreover, since $\varepsilon\le1/50$,
	\begin{equation}\label{eq:exp-bounds}
		e^{2\varepsilon|\log|\xi||}=|\xi|^{-2\varepsilon}
		\quad (0<|\xi|\le1),
		\qquad
		e^{2\varepsilon|\log|\xi||}=|\xi|^{2\varepsilon}\le|\xi|
		\quad (|\xi|\ge1).
	\end{equation}
	Consequently,
	\begin{equation}\label{eq:J-rem-majorant}
		\bigl|J(s,u_{s})+\kappa_{1}+2\kappa_{2}\varepsilon\bigr|
		\le\frac{\varepsilon^{2}}{2\pi}\bigl[2A_{1}+4A_{2}+A_{3}\bigr],
	\end{equation}
	where
	\begin{equation*}
		\begin{aligned}
			A_{1} & :=\int_{\R}\bigl|\log|\xi|\bigr|^{3}
			e^{2\varepsilon|\log|\xi||}\,|\widehat{\Phi_1'}|^{2}\,d\xi, \\
			A_{2} & :=\int_{\R}\bigl|\log|\xi|\bigr|^{2}
			e^{2\varepsilon|\log|\xi||}\,
			|\widehat{\Phi_1'}|\,|\widehat{\psi'}|\,d\xi,               \\
			A_{3} & :=\int_{\R}\bigl|\log|\xi|\bigr|
			e^{2\varepsilon|\log|\xi||}\,|\widehat{\psi'}|^{2}\,d\xi .
		\end{aligned}
	\end{equation*}

	\smallskip
	\noindent\emph{Step 3: estimates of the three integrals.}
	For $A_1$, split the integral at $|\xi|=1$.  The explicit formula
	in~\eqref{eq:Phi1-hat} and the inequality $\sinh t\ge t$ give
	$|\widehat{\Phi_1'}(\xi)|\le2$.  Thus, on $|\xi|\le1$,
	\eqref{eq:exp-bounds} and the substitution $t=-\log\xi$ yield
	\begin{equation*}
		\begin{aligned}
			\int_{|\xi|\le1}\bigl|\log|\xi|\bigr|^{3}
			e^{2\varepsilon|\log|\xi||}|\widehat{\Phi_1'}|^{2}\,d\xi
			 & \le 8\int_{0}^{1}|\log\xi|^{3}\xi^{-2\varepsilon}\,d\xi \\
			 & =8\int_{0}^{\infty}t^{3}e^{-(1-2\varepsilon)t}\,dt
			=\frac{48}{(1-2\varepsilon)^{4}}
			\le\frac{48}{(24/25)^{4}}<57,
		\end{aligned}
	\end{equation*}
	where the second equality follows by repeated integration by parts.

	For $\xi\ge1$, we have
	\begin{equation*}
		\sinh\left(\frac{\pi\xi}{\sqrt2}\right)
		=\frac12e^{\pi\xi/\sqrt2}\bigl(1-e^{-\sqrt2\pi\xi}\bigr)
		\ge\frac{49}{100}e^{\pi\xi/\sqrt2}.
	\end{equation*}
	Indeed, $\sqrt2\pi\xi>4$ and $e^{-4}<1/50$. Hence the same explicit formula for $\widehat{\Phi_1'}$, together with $\pi^2<10$ and $\sqrt2\pi>4$, gives
	\begin{equation*}
		|\widehat{\Phi_1'}(\xi)|^2
		\le 2\pi^2\left(\frac{100}{49}\right)^2
		\xi^2e^{-\sqrt2\pi\xi}
		\le84\xi^2e^{-4\xi}.
	\end{equation*}
	Using also $\log\xi\le\xi$ and $e^{2\varepsilon\log\xi}\le\xi$, we obtain
	\begin{equation*}
		\int_{|\xi|\ge1}\bigl|\log|\xi|\bigr|^{3}
		e^{2\varepsilon|\log|\xi||}|\widehat{\Phi_1'}|^{2}\,d\xi
		\le168\int_{1}^{\infty}\xi^6e^{-4\xi}\,d\xi
		\le168\left(\frac{2}{e}\right)^6\int_{1}^{\infty}e^{-\xi}\,d\xi
		=10752e^{-7}<10.
	\end{equation*}
	Here we used $\sup_{\xi\ge1}\xi^6e^{-3\xi}=(2/e)^6$, which follows by differentiation. Thus $A_{1}<67$.

	Next, on $|\xi|\le1$, $\bigl|\log|\xi|\bigr|e^{2\varepsilon|\log|\xi||}\xi^{2} =\bigl|\log|\xi|\bigr|\,|\xi|^{2-2\varepsilon} \le1/(e(2-2\varepsilon))\le1/e$. Indeed, differentiation shows that $t\mapsto-t^{2-2\varepsilon}\log t$ attains its maximum on $(0,1]$ at $t=e^{-1/(2-2\varepsilon)}$. On the other hand, for $|\xi|\ge1$, $\log|\xi|\le|\xi|$ and $e^{2\varepsilon|\log|\xi||}\le|\xi|$ give $\bigl|\log|\xi|\bigr|e^{2\varepsilon|\log|\xi||}\xi^{2} \le|\xi|\cdot|\xi|\cdot\xi^{2}=\xi^{4}\le(1+\xi^{2})^{3}$. Therefore
	\begin{equation*}
		\begin{aligned}
			A_{3} & \le\frac1e\int_{|\xi|\le1}|\widehat\psi|^{2}\,d\xi
			+\int_{|\xi|\ge1}(1+\xi^{2})^{3}|\widehat\psi|^{2}\,d\xi                    \\
			      & \le2\pi\Bigl[\frac1e\|\psi\|_{L^{2}}^{2}+\|\psi\|_{H^{3}}^{2}\Bigr] \\
			      & \le2\pi\Bigl[\frac1e\Bigl(\frac59\Bigr)^{2}+36\Bigr]
			\le230 .
		\end{aligned}
	\end{equation*}

	Finally, the factorization
	\begin{equation*}
		\bigl|\log|\xi|\bigr|^{2}e^{2\varepsilon|\log|\xi||}
		=\Bigl(\bigl|\log|\xi|\bigr|^{3}e^{2\varepsilon|\log|\xi||}\Bigr)^{1/2}
		\Bigl(\bigl|\log|\xi|\bigr|e^{2\varepsilon|\log|\xi||}\Bigr)^{1/2}
	\end{equation*}
	and Cauchy--Schwarz show that
	\begin{equation*}
		A_{2}\le A_{1}^{1/2}A_{3}^{1/2}
		\le\sqrt{67\cdot230}\le125.
	\end{equation*}
	Inserting the three bounds into~\eqref{eq:J-rem-majorant} yields
	\begin{equation*}
		\bigl|J(s,u_{s})+\kappa_{1}+2\kappa_{2}\varepsilon\bigr|
		\le\frac{\varepsilon^{2}}{2\pi}\bigl[2\cdot67+4\cdot125+230\bigr]
		=\frac{864}{2\pi}\varepsilon^{2}
		\le140\,\varepsilon^{2}.
	\end{equation*}
\end{proof}

The next lemma will allow us to transfer the expansion we have for $J(s,u_{s})$ to an expansion at the true layer.

\begin{lemma}\label{lem:J-transfer}
	For $0\le\varepsilon\le1/50$,
	\begin{equation*}
		|J(s,\Phis)-J(s,u_{s})|
		\;\le\; 100\,\varepsilon^{2}.
	\end{equation*}
\end{lemma}

\begin{proof}
	The assertion is immediate when $\varepsilon=0$, because then $u_1=\Phi_1$. We may therefore assume that $0<\varepsilon\le1/50$. By Lemma~\ref{lem:fixed-point}, $\Phis-u_s=r_s\in H^{2s}_{\mathrm{odd}}(\R)$, so it remains to estimate the seminorms in the polarization bound of Lemma~\ref{lem:LJ-polarization}.

	\smallskip
	\noindent\emph{Step 1: the Sobolev constant.}
	The function $t\mapsto|\log t|\,t^{s}/(1+t)^{2s}$ is invariant under
	$t\mapsto1/t$.  Hence the supremum in~\eqref{eq:Cstar-LJ-def} may be
	restricted to $t\ge1$.  On this interval,
	$(1+t)^{2s}\ge t^{2s}$, and differentiation shows that
	$(\log t)t^{-s}$ attains its maximum $1/(es)$ at $t=e^{1/s}$.
	Consequently,
	\begin{equation*}
		C_{\star}(s)
		\le\left(\frac{1}{2es}\right)^{1/2}
		\le\left(\frac{25}{49e}\right)^{1/2}
		=0.4332\ldots\le\frac12 ,
	\end{equation*}
	where we used $s\ge49/50$.

	\smallskip
	\noindent\emph{Step 2: the seminorms of $\Phi_1$ and $\psi$.}
	The explicit formula~\eqref{eq:Phi1-hat} gives
	\begin{equation*}
		\begin{aligned}
			\|\Phi_1\|_{\star}^{2}(s)
			=\frac1{2\pi}\int_{\R}\bigl|\log|\xi|\bigr|\,|\xi|^{-2\varepsilon}|\widehat{\Phi_1'}(\xi)|^{2}\,d\xi
			=2\pi\int_0^\infty
			\frac{\xi^{2-2\varepsilon}|\log\xi|}
			{\sinh^{2}(\pi\xi/\sqrt2)}\,d\xi .
		\end{aligned}
	\end{equation*}
	For $0<\xi\le1$, the inequality $\sinh t\ge t$ and the substitution $t=-\log\xi$ yield
	\begin{equation*}
		2\pi\int_0^1
		\frac{\xi^{2-2\varepsilon}|\log\xi|}
		{\sinh^{2}(\pi\xi/\sqrt2)}\,d\xi
		\le\frac4\pi\int_0^1\xi^{-2\varepsilon}|\log\xi|\,d\xi
		=\frac{4}{\pi(1-2\varepsilon)^2}
		\le\frac4\pi\left(\frac{25}{24}\right)^2
		<1.382.
	\end{equation*}
	For $\xi\ge1$, the identity
	\begin{equation*}
		\sinh\left(\frac{\pi\xi}{\sqrt2}\right)
		=\frac12e^{\pi\xi/\sqrt2}\bigl(1-e^{-\sqrt2\pi\xi}\bigr)
		\ge\frac{49}{100}e^{\pi\xi/\sqrt2}
	\end{equation*}
	follows from $\sqrt2\pi\xi>4$ and $e^{-4}<1/50$. Since $\sqrt2\pi>21/5$ (because $\sqrt2>7/5$ and $\pi>3$), $\xi^{-2\varepsilon}\le1$, and $\log\xi\le\xi-1$, the change of variables $t=\xi-1$ gives
	\begin{equation*}
		\begin{aligned}
			2\pi\int_1^\infty
			\frac{\xi^{2-2\varepsilon}\log\xi}
			{\sinh^{2}(\pi\xi/\sqrt2)}\,d\xi
			 & \le2\pi\left(\frac{100}{49}\right)^2
			\int_1^\infty\xi^2(\xi-1)e^{-21\xi/5}\,d\xi    \\
			 & =2\pi\left(\frac{100}{49}\right)^2e^{-21/5}
			\left[
				\left(\frac5{21}\right)^2
				+4\left(\frac5{21}\right)^3
				+6\left(\frac5{21}\right)^4
			\right]                                        \\
			 & =0.0509\ldots<0.051.
		\end{aligned}
	\end{equation*}
	Combining the two regions,
	\begin{equation*}
		\|\Phi_1\|_{\star}^{2}(s)<1.382+0.051
		=1.433<\left(\frac65\right)^2,
		\qquad\text{hence}\qquad
		\|\Phi_1\|_{\star}(s)\le\frac65.
	\end{equation*}
	Moreover, since $2s\le2<3$, \eqref{eq:star-sobolev}, Step~1, and Lemma~\ref{lem:corrector-bounds} imply
	\begin{equation*}
		\|\psi\|_{\star}(s)
		\le C_{\star}(s)\|\psi\|_{H^{2s}}
		\le\frac12\|\psi\|_{H^3}\le3.
	\end{equation*}

	\smallskip
	\noindent\emph{Step 3: conclusion.}
	By~\eqref{eq:star-sobolev}, Step~1, and Lemma~\ref{lem:fixed-point},
	\begin{equation*}
		\|r_s\|_{\star}(s)
		\le C_{\star}(s)\|r_s\|_{H^{2s}}
		\le\frac{75}{2}\varepsilon^2.
	\end{equation*}
	The triangle inequality and Step~2 now give
	\begin{equation*}
		\|u_s\|_{\star}
		\le\|\Phi_1\|_{\star}+\varepsilon\|\psi\|_{\star}
		\le\frac65+\frac3{50}=\frac{63}{50}\le\frac{13}{10},
		\qquad
		\|\Phis\|_{\star}
		\le\|u_s\|_{\star}+\|r_s\|_{\star}
		\le\frac{63}{50}+\frac{75}{2}\cdot\frac1{2500}
		\le\frac{13}{10}.
	\end{equation*}
	Applying Lemma~\ref{lem:LJ-polarization}(ii) to $\Phis$ and $u_s$, we conclude that
	\begin{equation*}
		|J(s,\Phis)-J(s,u_{s})|
		\le\bigl(\|\Phis\|_{\star}+\|u_{s}\|_{\star}\bigr)\|r_{s}\|_{\star}
		\le\left(\frac{13}{10}+\frac{13}{10}\right)
		\frac{75}{2}\varepsilon^{2}
		=97.5\,\varepsilon^{2}\le100\,\varepsilon^{2}. \qedhere
	\end{equation*}
\end{proof}

At this point, we are ready to integrate $J(s, \Phi_s)$ to prove the expansions and the monotonicity of the energy. The endpoint value needed in the integration is provided by Remark~\ref{rmk:endpoint-limit} in Section~\ref{sec:cont}.

\begin{proof}[Proof of Theorem~\ref{thm:kappa1}]
	The inequality $\kappa_{2}\ge0$ is Lemma~\ref{lem:kappa-constants}.

	By Lemma~\ref{lem:envelope-paper}, $\Es$ is $C^{1}$ on $(1/2,1)$ with $\Es'(s)=J(s,\Phis)$; in particular $\Es\in C^{1}([49/50,1))$. For $s\in[49/50,1)$ the triangle inequality, Lemma~\ref{lem:J-trial} and Lemma~\ref{lem:J-transfer} give
	\begin{equation*}
		\bigl|\Es'(s)+\kappa_{1}+2\kappa_{2}\varepsilon\bigr|
		\le\bigl|J(s,\Phis)-J(s,u_{s})\bigr|
		+\bigl|J(s,u_{s})+\kappa_{1}+2\kappa_{2}\varepsilon\bigr|
		\le(100+140)\,\varepsilon^{2},
	\end{equation*}
	which is~\eqref{eq:URem2-paper} with $240=140+100$.

	Now fix $s\in[49/50,1)$ and let $\sigma\in(s,1)$. Integrating~\eqref{eq:URem2-paper} over $[s,\sigma]$,
	\begin{equation*}
		\Es(\sigma)-\Es(s)
		=\int_{s}^{\sigma}\bigl(-\kappa_{1}-2\kappa_{2}(1-\tau)+R_{2}(\tau)\bigr)\,d\tau
		=-\kappa_{1}(\sigma-s)-\kappa_{2}\bigl[(1-s)^{2}-(1-\sigma)^{2}\bigr]
		+\int_{s}^{\sigma}R_{2}(\tau)\,d\tau .
	\end{equation*}
	Letting $\sigma\to1^{-}$ and using Remark~\ref{rmk:endpoint-limit} on the left gives
	\begin{equation*}
		\frac{2\sqrt2}{3}-\Es(s)
		=-\kappa_{1}\varepsilon-\kappa_{2}\varepsilon^{2}+\int_{s}^{1}R_{2}(\tau)\,d\tau ,
	\end{equation*}
	the integral converging absolutely because $|R_{2}(\tau)|\le240(1-\tau)^{2}$. Hence $\Es(s)=\tfrac{2\sqrt2}{3}+\kappa_{1}\varepsilon+\kappa_{2}\varepsilon^{2}+R_{3}(s)$ with $R_{3}(s):=-\int_{s}^{1}R_{2}(\tau)\,d\tau$, and
	\begin{equation*}
		|R_{3}(s)|\le\int_{s}^{1}240(1-\tau)^{2}\,d\tau
		=\frac{240}{3}\varepsilon^{3}=80\,\varepsilon^{3},
	\end{equation*}
	which is~\eqref{eq:URem2-integrated}. At $s=1$ both sides of~\eqref{eq:URem2-integrated} equal $2\sqrt2/3$, so the statement holds on the closed interval $[49/50,1]$.
\end{proof}

\begin{proof}[Proof of Proposition~\ref{prop:mono-upper}]
	Let $s\in[49/50,1)$ and $\varepsilon=1-s\in(0,1/50]$. Since $\kappa_{2}\ge0$ and $\varepsilon>0$, Theorem~\ref{thm:kappa1}(ii) gives
	\begin{equation*}
		\Es'(s)\le-\kappa_{1}+240\,\varepsilon^{2}
		<-\frac{119}{900}+\frac{12}{125}
		=-\frac{163}{4500}<-\frac{3}{100},
	\end{equation*}
	using $\kappa_{1}>119/900$ from Lemma~\ref{lem:kappa-constants} and $240\,\varepsilon^{2}\le240/2500=12/125$. For $49/50\le s_{1}<s_{2}<1$, integrating this bound over $[s_{1},s_{2}]$ gives $\Es(s_{2})-\Es(s_{1})\le-0.03\,(s_{2}-s_{1})<0$. Letting $s_{2}\to1^{-}$ and applying Remark~\ref{rmk:endpoint-limit},
	\begin{equation*}
		\Es(1)-\Es(s_{1})\le-0.03\,(1-s_{1})<0
		\qquad(49/50\le s_{1}<1),
	\end{equation*}
	so the strict inequality also holds when $s_{2}=1$. Hence $\Es$ is strictly decreasing on $[49/50,1]=[0.980,1]$.
\end{proof}

\begin{remark}\label{rem:slope-margin-upper}
	The purpose of the explicit bounds above is to retain a strictly negative margin all the way down to $s=0.980$, where this analytic endpoint argument joins the computer-assisted interior argument. The endpoint calculation uses only the sign $\kappa_{2}\ge0$ and the lower bound $\kappa_{1}>119/900$. The true coefficient is far larger, $\kappa_{1}\approx1.073$, so the genuine slope near $s=1$ is likely close to $-1.07$. We chose not to prove such a sharp bound to keep the computations in the argument as simple as possible.
\end{remark}

\section{The endpoint \texorpdfstring{$s=1/2$}{s=1/2}: pole and strict decrease}\label{sec:s12}
\providecommand{\eps}{\varepsilon}
\providecommand{\Kin}{\mathcal{K}}
\providecommand{\Pot}{\mathcal{P}}

In this section we prove part~(ii) of Theorem~\ref{thm:asymp}; that is, the function $\Es$ has a pole at $s=1/2$ with explicit residue $1/\pi$ and a bounded remainder. The same analysis controls the derivative $\Es'$ and yields strict decrease on $(1/2,0.530]$.

\begin{proposition}\label{prop:mono-lower}
	Let $\delta:=3/100$. Then $\Es\in C^{1}((1/2,1/2+\delta])$ and $\Es'(s)<0$ for every $s\in(1/2,1/2+\delta]$. Equivalently, $\Es$ is strictly decreasing on $(1/2,0.530]$.
\end{proposition}

The pole expansion has three ingredients. The first is a Pohozaev identity which decouples the kinetic and potential parts of the energy and shows that the pole comes entirely from the kinetic part. The second is a Fourier-side computation of the residue, using only the identity $\int_{\R}\Phis'\,dx=2$ and the behavior of the spectral weight $|\xi|^{2s}$ near $\xi=0$. The third is a control, uniform in $s$, of the remainders in Fourier variables, which excludes a logarithmic intermediate term.

For brevity we adopt the abbreviations
\begin{equation}\label{eq:s6-shorthand}
	\eps := s-\tfrac12,
	\qquad
	\Kin_{s}[u]
	:= \frac{\cs}{4}\iint_{\R\times\R}\frac{(u(x)-u(y))^{2}}{|x-y|^{1+2s}}\,dx\,dy,
	\qquad
	\Pot[u] := \int_{\R}W(u)\,dx,
\end{equation}
so that $E_{s}[u]=\Kin_{s}[u]+\Pot[u]$ for any $u\in\mathcal A$, and we write $\Kin_{s}:=\Kin_{s}[\Phis]$, $\Pot_{s}:=\Pot[\Phis]$ and $\Es(s)=\Kin_{s}+\Pot_{s}$. The scaling identity and the Fourier representation below hold for all $s\in(1/2,1)$. The quantitative estimates are stated on the subinterval $(1/2,53/100]$ (that is, $0<\eps\le3/100$), which suffices for the limit $s\to(1/2)^{+}$ and matches the endpoint needed for the computer assisted part of the proof. Throughout this section we use that $\Phis'$ is positive and integrable with total mass
\begin{equation}\label{eq:total-mass}
	\int_{\R}\Phis'(x)\,dx = \Phis(+\infty)-\Phis(-\infty) = 2,
\end{equation}
by Theorem~\ref{thm:CS}; in particular $\widehat{\Phis'}(0)=2$.

The first observation reduces the analysis of $\Es$ near $s=1/2$ to the analysis of the kinetic part $\Kin_{s}$.

\begin{lemma}\label{lem:pohozaev}
	For every $s\in(1/2,1)$,
	\begin{equation}\label{eq:pohozaev}
		\Pot_{s} \;=\; (2s-1)\,\Kin_{s},
	\end{equation}
	and consequently
	\begin{equation}\label{eq:E-K-V}
		\Es(s) \;=\; 2s\,\Kin_{s},
		\qquad
		\Pot_{s} \;=\; \frac{2s-1}{2s}\,\Es(s)
		\;=\; \frac{2\eps}{1+2\eps}\,\Es(s).
	\end{equation}
\end{lemma}

\begin{proof}
	For $\lambda>0$ set $\Phi_{s,\lambda}(x):=\Phis(x/\lambda)$. A change of variables in \eqref{eq:s6-shorthand} gives $ \Kin_{s}[\Phi_{s,\lambda}]=\lambda^{1-2s}\Kin_{s} $ and $ \Pot[\Phi_{s,\lambda}]=\lambda\,\Pot_{s}, $ both finite for $s\in(1/2,1)$. To see that $\Kin_{s}$ and $\Pot_{s}$ are finite, note that the arctan profile $g(x)=(2/\pi)\arctan x$ belongs to the class $\mathcal{A}$ and has finite energy $E_{s}[g]<\infty$. The explicit value is computed in Remark~\ref{rem:arctan}. Hence, by the minimality of Theorem~\ref{thm:PSV},
	\begin{equation*}
		\Kin_{s}+\Pot_{s}=E_{s}[\Phis]\;\leq\;E_{s}[g]\;<\;\infty,
	\end{equation*}
	and since $\Kin_{s}\geq 0$ and $\Pot_{s}\geq 0$, both terms are finite. Consequently each $\Phi_{s,\lambda}$ has finite energy $\lambda^{1-2s}\Kin_{s}+\lambda\,\Pot_{s}<\infty$ and hence is in $\mathcal{A}$, so by Theorem~\ref{thm:PSV} the function
	\begin{equation*}
		g(\lambda) := E_{s}[\Phi_{s,\lambda}]
		= \lambda^{1-2s}\Kin_{s} + \lambda\,\Pot_{s},
		\qquad \lambda\in(0,\infty),
	\end{equation*}
	attains its infimum at $\lambda=1$. Since $g$ is smooth on $(0,\infty)$, $g'(1)=0$, i.e.\ $(1-2s)\Kin_{s}+\Pot_{s}=0$, which is \eqref{eq:pohozaev}. The identities \eqref{eq:E-K-V} follow by adding $\Kin_{s}$ to both sides of \eqref{eq:pohozaev} and rearranging.
\end{proof}

Identity~\eqref{eq:pohozaev} is the fractional analog of the classical Pohozaev identity for $-\Delta u=f(u)$. The derivation above is a purely variational consequence of the minimality. The right-hand identity of \eqref{eq:E-K-V} immediately exhibits two consequences for $s\to(1/2)^{+}$:
\begin{itemize}
	\item Any upper bound $\Es(s)\leq M(s-1/2)^{-1}$ implies $\Pot_{s}\leq 2M/(1+2\eps)$, so the potential part of the energy stays bounded as $s\to(1/2)^{+}$. Consequently any singularity of $\Es$ is carried entirely by~$\Kin_{s}$.
	\item The limit $ \lim_{s\to(1/2)^{+}}(s-\tfrac{1}{2})\,\Es(s) $ coincides, modulo a factor $2s\to 1$, with the limit $ \lim_{s\to(1/2)^{+}}\eps\,\Kin_{s}, $ so either residue may be recovered from the other.
\end{itemize}

\subsection{The differentiated pole and its remainder}
\label{ss:s12-remainder}

In this section we extract the pole from the kinetic energy term $\mathcal{K}_s$. This analysis will need the Fourier representation of this term, which was computed in Lemma~\ref{lem:fourier-energy-nondecaying}, and it is given by
\begin{equation*}
	\Kin_{s}
	\;=\;
	\frac{1}{4\pi}\int_{\R}|\xi|^{2s-2}\,|\widehat{\Phis'}(\xi)|^{2}\,d\xi,
	\qquad s\in(1/2,1).
\end{equation*}

Before computing the residue we state two facts about the layer. First we use the arctan as a competitor to obtain both the matching upper bound for the pole in the kinetic energy and a uniform bound on the potential. Then we estimate the polynomial tail of the layer using this boundedness.

\begin{remark}\label{rem:arctan}
	The function $g(x)=(2/\pi)\arctan x$ is admissible in Definition~\ref{def:admissible} for every $s\in(1/2,1)$. It also satisfies $1-|g(x)|\sim 2/(\pi|x|)$, which matches the $x^{-2s}$ tail of fractional layers at $s=1/2$. Since $\widehat{g'}(\xi)=2e^{-|\xi|}$, Lemma~\ref{lem:fourier-energy-nondecaying} gives
	\begin{equation*}
		\Kin_{s}[g]
		= \frac{1}{4\pi}\int_{\R}|\xi|^{-1+2\eps}\cdot 4e^{-2|\xi|}\,d\xi
		= \frac{2}{\pi}\,2^{1-2s}\,\Gamma(2s-1)
		= \frac{2^{-2\eps}\,\Gamma(1+2\eps)}{\pi\,\eps},
	\end{equation*}
	the last equality by $\Gamma(2s-1)=\Gamma(2s)/(2s-1)$ and $\Gamma(2s)=\Gamma(1+2\eps)$. This equation reads $\Kin_{s}[g]=(1/\pi)/\eps+O(1)$ with a simple pole of residue $1/\pi$ and no logarithmic term. Two quantitative consequences of this are used below. First, on $0<\eps\le3/100$, since $0\le2\eps\le0.06$, $2^{-2\eps}\le1$ and $\Gamma(1+2\eps)\le1$ (log-convexity of $\Gamma$ with $\Gamma(1)=\Gamma(2)=1$), whence
	\begin{equation*}
		\Kin_{s}[g]\;\le\;\frac{1}{\pi\eps}.
	\end{equation*}
	Second, the potential of the competitor is bounded,
	\begin{equation*}
		\Pot[g]\;\le\;\frac43.
	\end{equation*}
	To see this, on $|x|\le1$ one uses $0\le W\le1/4$, a contribution at most $1/2$; on $|x|\ge1$, $\arctan x\ge\pi/2-1/x$ gives $1-g(x)^{2}\le4/(\pi|x|)$, hence $W(g(x))\le4/(\pi^{2}x^{2})$. In total, using minimality we obtain
	\begin{equation}\label{eq:arctan-upper}
		\Es(s)\;\le\;E_{s}[g]\;=\;\frac{1/\pi}{\eps}+O(1).
	\end{equation}
\end{remark}

Now we need an explicit estimate for the tail of the layer. The bound in the next lemma is uniform in $s$ and not sharp. The true decay is algebraic of order $|x|^{-2s}$ by~\eqref{eq:CS-tail}, as illustrated in Figure~\ref{fig:layer-tails}.

\begin{lemma}\label{lem:tail}
	For every $s\in(1/2,53/100]$,
	\begin{equation}\label{eq:tail-poly}
		1-\Phis(R)
		\;\leq\; \tfrac32\,R^{-1/2},
		\qquad R\geq 1.
	\end{equation}
\end{lemma}

\begin{proof}
	Fix $R\geq 1$ and put $a_{R}:=1-\Phis(R)$. Since $\Phis$ is odd and increasing, $|\Phis(x)|\leq 1-a_{R}$ on $[-R,R]$, so
	\begin{equation*}
		W(\Phis(x))=\frac{(1-\Phis(x)^{2})^{2}}{4}\ge\frac{a_{R}^{2}}{4},
		\qquad |x|\le R,
	\end{equation*}
	whence $\Pot_{s}\ge\int_{-R}^{R}W(\Phis)\ge R\,a_{R}^{2}/2$. The potential is bounded uniformly on the present interval by the scaling identity~\eqref{eq:E-K-V} and the arctan upper bound~\eqref{eq:arctan-upper},
	\begin{equation*}
		\Pot_{s}=\frac{2\eps}{1+2\eps}\,\Es(s)\le 2\eps\,E_{s}[g]
		=2\eps\bigl(\Kin_{s}[g]+\Pot[g]\bigr)
		\le \frac{2}{\pi}+\frac{8}{3}\eps
		\le \frac{2}{\pi}+\frac{2}{25}<\frac34.
	\end{equation*}
	Therefore
	\begin{equation*}
		a_{R}\le \sqrt{\frac{2\Pot_{s}}{R}}
		\le\sqrt{\tfrac32}\,R^{-1/2}\le \tfrac32\,R^{-1/2}. \qedhere
	\end{equation*}
\end{proof}

\begin{figure}[t]
	\centering
	\includegraphics[width=0.78\textwidth]{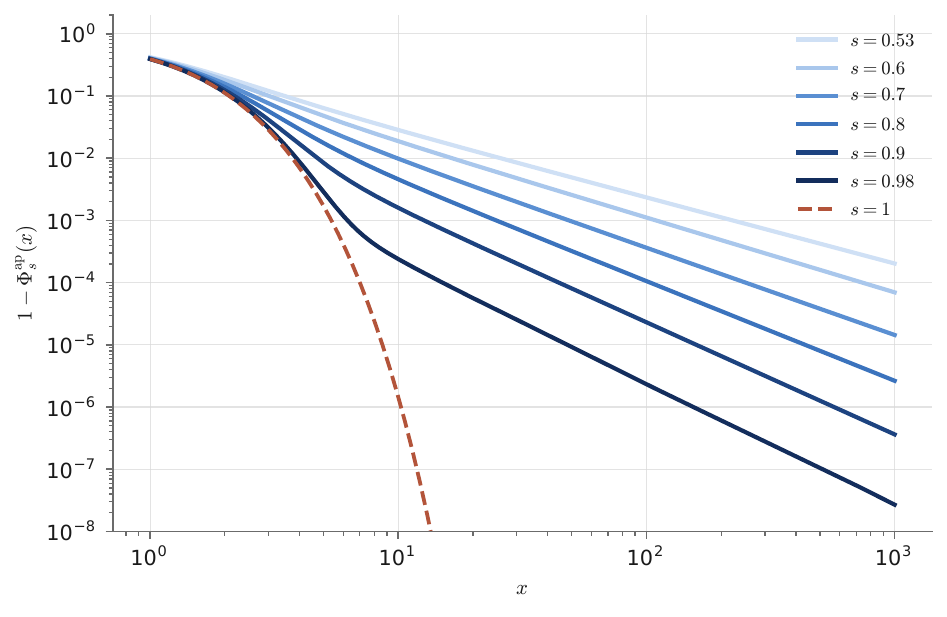}
	\caption{The layer tail $1-\Phis(x)$ on a log--log scale. For $s<1$ the
		decay is algebraic with slope $-2s$, in agreement with the Cabr\'e--Sire
		bounds~\eqref{eq:CS-tail}, and it degenerates to the exponential decay of
		$\tanh(x/\sqrt2)$ only at $s=1$. The curves are computed from the approximate layer
		$\Phi^{\mathrm{ap}}_{s}$ of Appendix~\ref{app:layer}.}
	\label{fig:layer-tails}
\end{figure}

We now carry out the analysis for $\Es'$ and show that the pole may be subtracted. Writing $C(s):=\Es(s)-(1/\pi)/\eps$ for the remainder in
\begin{equation*}
	\Es(s) \;=\; \frac{1/\pi}{s-1/2}+C(s),
	\qquad s\in(1/2,1],
\end{equation*}
the task is to bound $C'(s)=\Es'(s)+1/(\pi\eps^{2})$.

Recall the identity of Lemma~\ref{lem:envelope-paper} $\Es\in C^{1}((1/2,1))$ with
\begin{equation*}
	\Es'(s)\;=\;J(s,\Phis)\;=\;\frac{1}{2\pi}\int_{\R}\log|\xi|\,|\xi|^{2s-2}\,|\widehat{\Phis'}(\xi)|^{2}\,d\xi
	\qquad(s\in(1/2,1)).
\end{equation*}
The identity $\widehat{\Phis'}(0)=2$ isolates the low-frequency singularity,
\begin{equation}\label{eq:Es-prime-split}
	\Es'(s)\;=\;-\frac{1}{\pi\varepsilon^{2}}+K(s),
	\qquad K(s):=K_{<}(s)+K_{>}(s),
\end{equation}
since
\begin{equation*}
	-\frac{1}{\pi\varepsilon^{2}}
	\;=\;\frac{1}{2\pi}\int_{|\xi|<1}\log|\xi|\,|\xi|^{2s-2}\cdot 4\,d\xi.
\end{equation*}
Now the low- and high-frequency remainders are
\begin{equation}\label{eq:K-low}
	K_{<}(s):=\frac{1}{2\pi}\int_{|\xi|<1}\log|\xi|\,|\xi|^{2s-2}
	\bigl(|\widehat{\Phis'}(\xi)|^{2}-4\bigr)\,d\xi,
\end{equation}
\begin{equation*}
	K_{>}(s):=\frac{1}{2\pi}\int_{|\xi|\ge1}\log|\xi|\,|\xi|^{2s-2}
	|\widehat{\Phis'}(\xi)|^{2}\,d\xi .
\end{equation*}
In the next lemma we control these remainders. We need explicit constants in the argument because we need to show that our estimates hold through the interval $(1/2, 53/100]$.

\begin{lemma}\label{lem:deriv-remainder}
	For every $s\in(1/2,53/100]$,
	\begin{equation}\label{eq:L-Lip-paper}
		\left|\Es'(s)+\frac{1}{\pi(s-1/2)^{2}}\right|
		\;\le\;\frac{532}{\pi}.
	\end{equation}
	Equivalently, the remainder $C(s)=\Es(s)-(1/\pi)(s-1/2)^{-1}$ is Lipschitz on $(1/2,53/100]$ with constant $532/\pi$. In particular, $C$ is bounded on this interval. Consequently, there is a constant $C_{0}$ such that
	\begin{equation}\label{eq:K-residue}
		\Bigl|\,\Kin_{s} - \frac{1}{\pi\,(s-1/2)}\,\Bigr|\;\leq\; C_{0}.
	\end{equation}
\end{lemma}

\begin{proof}
	We bound $K(s)=K_{<}(s)+K_{>}(s)$ in~\eqref{eq:Es-prime-split}.

	\smallskip
	\noindent\emph{High frequencies.}  By the pointwise
	bound~\eqref{eq:hf-fourier-bound-lem}, namely
	$|\widehat{\Phis'}(\xi)|\le4|\xi|^{-2s}$ for
	$|\xi|\ge1$, we have
	$|\widehat{\Phis'}(\xi)|^{2}\le16|\xi|^{-4s}$. Since
	$|\xi|^{2s-2}=|\xi|^{-1+2\eps}$ and $-1+2\eps-4s=-2-2s$,
	\begin{equation}\label{eq:Khigh-explicit}
		|K_{>}(s)|
		\le \frac{16}{\pi}\int_{1}^{\infty}\log\xi\,\xi^{-2-2s}\,d\xi
		=\frac{16}{\pi(1+2s)^{2}}
		\le \frac{4}{\pi}.
	\end{equation}

	\smallskip
	\noindent\emph{Low frequencies.}
	By~\eqref{eq:total-mass}, the Fourier-transform representation gives
	\begin{equation*}
		4-|\widehat{\Phis'}(\xi)|^2
		=4 - \iint_{\mathbb R^2}\cos(\xi(x-y))\,\Phis'(x)\Phis'(y)\,dx\,dy
		=\iint_{\R^2}
		\bigl(1-\cos(\xi(x-y))\bigr)\,\Phis'(x)\Phis'(y)\,dx\,dy.
	\end{equation*}
	Define
	\begin{equation*}
		J^{\log}_{\eps}(z)
		:=\int_{-1}^{1}(-\log|\xi|)|\xi|^{-1+2\eps}
		\bigl(1-\cos(\xi z)\bigr)\,d\xi.
	\end{equation*}
	Since $|\widehat{\Phis'}(\xi)|\leq2$ by~\eqref{eq:total-mass}, \eqref{eq:K-low} shows that $K_{<}(s)\geq0$. All the integrands below are nonnegative, so Tonelli's theorem yields
	\begin{equation*}
		|K_{<}(s)|=K_{<}(s)
		=\frac{1}{2\pi}\iint_{\R\times\R}
		J^{\log}_{\eps}(x-y)\,\Phis'(x)\Phis'(y)\,dx\,dy.
	\end{equation*}
	We first show that
	\begin{equation}\label{eq:Jlog-bound}
		J^{\log}_{\eps}(z)\leq 3\log^{2}(e+|z|)
		\qquad (z\in\R).
	\end{equation}
	By evenness, it suffices to consider $z\geq0$. If $0\leq z\leq1$, the inequality $1-\cos t\leq t^2/2$ yields
	\begin{equation*}
		J^{\log}_{\eps}(z)
		\leq z^2\int_0^1(-\log\xi)\xi^{1+2\eps}\,d\xi
		=\frac{z^2}{(2+2\eps)^2}
		\leq\frac14.
	\end{equation*}
	Suppose now that $z>1$, and split the integral at $\xi=1/z$. On $(0,1/z)$, the same inequality and $\eps>0$ give
	\begin{equation*}
		2\int_0^{1/z}(-\log\xi)\xi^{-1+2\eps}
		\bigl(1-\cos(\xi z)\bigr)\,d\xi
		\leq z^2\int_0^{1/z}(-\log\xi)\xi\,d\xi
		=\frac12\log z+\frac14.
	\end{equation*}
	On $(1/z,1)$, using $1-\cos t\leq2$ and then setting $t=-\log\xi$, we obtain
	\begin{equation*}
		2\int_{1/z}^1(-\log\xi)\xi^{-1+2\eps}
		\bigl(1-\cos(\xi z)\bigr)\,d\xi
		\leq4\int_{1/z}^1(-\log\xi)\xi^{-1}\,d\xi
		=4\int_0^{\log z}t\,dt
		=2\log^{2}z.
	\end{equation*}
	Since $\log(e+z)\geq1$ and $\log z\leq\log(e+z)$, these estimates prove~\eqref{eq:Jlog-bound}.

	It remains to control the corresponding logarithmic moment of $\Phis'$. Set $t_0:=\log(e+1)<1.314$ and $m(t):=\int_{\{\log(e+|x|)>t\}}\Phis'(x)\,dx$, so that $m(t)\le2$ by~\eqref{eq:total-mass}. For $t\geq t_0$, the inequality $\log(e+|x|)>t$ is equivalent to $|x|>e^t-e$, where $e^t-e\geq1$; hence, by oddness of $\Phis$ and~\eqref{eq:tail-poly},
	\begin{equation*}
		m(t)=2\bigl(1-\Phis(e^t-e)\bigr)\leq 3\,(e^t-e)^{-1/2},
		\qquad t\geq t_0.
	\end{equation*}
	The layer-cake formula therefore implies
	\begin{equation*}
		\int_{\R}\log^{2}(e+|x|)\,\Phis'(x)\,dx
		=\int_0^\infty
		2t\,m(t)\,dt
		\leq 2t_0^2
		+6\int_{t_0}^\infty t\,(e^t-e)^{-1/2}\,dt.
	\end{equation*}
	For $t\geq t_0$, we have $e^t-e\geq e^t/(e+1)$, and therefore
	\begin{equation*}
		\int_{t_0}^\infty t\,(e^t-e)^{-1/2}\,dt
		\leq\sqrt{e+1}\int_{t_0}^\infty t e^{-t/2}\,dt
		=2(t_0+2)<6.628.
	\end{equation*}
	Consequently,
	\begin{equation}\label{eq:log-moment}
		\int_{\R}\log^{2}(e+|x|)\,\Phis'(x)\,dx<44.
	\end{equation}
	Moreover, the triangle inequality gives $e+|x-y|\leq(e+|x|)(e+|y|)$, and hence
	\begin{equation*}
		\log^{2}(e+|x-y|)
		\leq2\log^{2}(e+|x|)+2\log^{2}(e+|y|).
	\end{equation*}
	Combining this inequality with~\eqref{eq:total-mass} and~\eqref{eq:log-moment}, we find
	\begin{equation*}
		\iint_{\R\times\R}\log^{2}(e+|x-y|)
		\,\Phis'(x)\Phis'(y)\,dx\,dy<352.
	\end{equation*}
	Thus~\eqref{eq:Jlog-bound} gives
	\begin{equation}\label{eq:Klow-explicit}
		|K_{<}(s)|
		\leq \frac{3}{2\pi}\iint_{\R\times\R}\log^{2}(e+|x-y|)
		\,\Phis'(x)\Phis'(y)\,dx\,dy
		\leq \frac{3\cdot352}{2\pi}
		=\frac{528}{\pi}.
	\end{equation}

	Combining~\eqref{eq:Khigh-explicit} and~\eqref{eq:Klow-explicit}, we conclude that
	\begin{equation*}
		|K(s)|\leq\frac{528}{\pi}+\frac4\pi=\frac{532}{\pi},
	\end{equation*}
	which proves~\eqref{eq:L-Lip-paper}. By Lemma~\ref{lem:envelope-paper}, the remainder $C$ belongs to $C^1((1/2,53/100])$, and~\eqref{eq:Es-prime-split} gives $C'(s)=K(s)$. The bound above and the fundamental theorem of calculus therefore show that $C$ is Lipschitz with the stated constant.

	To prove boundedness up to the left endpoint, fix $s_0=53/100$. For every $s\in(1/2,s_0]$,
	\begin{equation*}
		|C(s)|
		\leq |C(s_0)|+\int_s^{s_0}|C'(r)|\,dr
		\leq |C(s_0)|+\frac{532}{\pi}(s_0-s),
	\end{equation*}
	so $C$ is bounded on $(1/2,s_0]$. Finally, \eqref{eq:E-K-V} yields
	\begin{equation*}
		\Kin_{s}-\frac{1}{\pi\eps}
		=\frac{C(s)}{1+2\eps}-\frac{2}{\pi(1+2\eps)}.
	\end{equation*}
	The right-hand side is uniformly bounded, which proves~\eqref{eq:K-residue} and completes the proof.
\end{proof}

\begin{proof}[Proof of Theorem~\ref{thm:asymp}(ii)]
	The boundedness of $C$ in Lemma~\ref{lem:deriv-remainder} yields the sharp statement
	\begin{equation}\label{eq:final-pole}
		\Es(s) \;=\; \frac{1/\pi}{s-1/2} \;+\; O(1)
		\qquad\text{as }s\to (1/2)^{+},
	\end{equation}
	which is \eqref{eq:asymp-s12-sharp}.
\end{proof}

\begin{remark}
	Identity~\eqref{eq:E-K-V} together with \eqref{eq:final-pole} gives $ \Pot_{s} = (2\eps/(1+2\eps))\,\Es(s) = 2/\pi + O(\eps), $ so the potential part of the layer's energy converges to $2/\pi$ as $s\to(1/2)^{+}$. The Fourier-side identification of the residue can thus be read as the statement that the kinetic part of the energy concentrates at zero frequency while the potential part contributes only to the $O(1)$ remainder.
\end{remark}

\subsection{Strict decrease}
\label{ss:s12-conclusion}

Now that we have all the estimates with explicit constants, we are able to show strict decrease of the energy near the endpoint $s=1/2$.

\begin{proof}[Proof of Proposition~\ref{prop:mono-lower}]
	Let $0<\varepsilon=s-1/2\le\delta=3/100$. Lemma~\ref{lem:deriv-remainder} and~\eqref{eq:Es-prime-split} give
	\begin{equation*}
		\Es'(s)\le -\frac{1}{\pi\varepsilon^{2}}+\frac{532}{\pi}.
	\end{equation*}
	Because $\varepsilon\le3/100$,
	\begin{equation*}
		\frac{532}{\pi}<\frac{5000}{9\pi}=\frac{1}{2\pi(3/100)^{2}}
		\le \frac{1}{2\pi\varepsilon^{2}}.
	\end{equation*}
	Therefore
	\begin{equation}\label{eq:Es-prime-neg-lower}
		\Es'(s)
		\le -\frac{1}{\pi\varepsilon^{2}}+\frac{1}{2\pi\varepsilon^{2}}
		=-\frac{1}{2\pi\varepsilon^{2}}<0,
		\qquad 0<\varepsilon\le\frac{3}{100}.
	\end{equation}
	Thus $\Es$ is strictly decreasing on $(1/2,53/100]$.
\end{proof}

\section{Strict monotone decrease on the interior interval}\label{sec:mono-interior}
In this section we close the proof of Theorem~\ref{thm:mono} by establishing strict decrease of $\Es$ on the closed interval
\begin{equation}\label{eq:interior-interval}
	I_{0} \;:=\; [\underline{s},\overline{s}],\qquad \underline{s}=0.530,\quad \overline{s}=0.980,
\end{equation}
which meets the endpoint intervals of Propositions~\ref{prop:mono-lower} and~\ref{prop:mono-upper}. Throughout the section, all displayed numerical constants are obtained by interval arithmetic at 200-bit precision.

The argument proceeds as follows. In Section~\ref{ss:semiconc} we establish a profile-free bound on the second derivative in $s$ of the energy of a frozen layer, and deduce that $\Es$ is semiconcave on each cell of a grid $s_{0}<\underline{s}<s_{1}<\cdots<s_{N-1}<\overline{s}<s_{N}$. Section~\ref{ss:inverse-ap} constructs an approximation to $\Phi_s$ and uses a fixed-point argument to pass from this approximation to the true Cabr\'e--Sire layer. In Section~\ref{ss:realization} a finite system of inequalities, verified at the grid points by interval arithmetic, is shown in Proposition~\ref{prop:cells-I0} to imply strict decrease on all of $I_{0}$. The proof of Theorem~\ref{thm:mono} concludes by combining this result with Propositions~\ref{prop:mono-lower} and~\ref{prop:mono-upper} on the other two intervals of~\eqref{eq:three-interval-cover}.

\subsection{Semiconcavity of the energy}\label{ss:semiconc}

We use the derivative functional $J$ defined in~\eqref{eq:J-functional}. The interior argument rests on an upper bound, uniform over a cell, for the second derivative in $s$ of the energy of a frozen layer. Combined with the minimality of Theorem~\ref{thm:PSV}, such a bound makes $\Es$ semiconcave on the cell, and semiconcavity propagates a strict sign of $\Es'=J(\cdot,\Phi_{\cdot})$ from a single grid point to the whole cell, at the cost of a condition on the cell width. The next lemma provides the bound, with an explicit constant depending only on the cell endpoints and not on the profile, together with the resulting semiconcavity inequality.

\begin{lemma}\label{lem:semiconcave}
	Let $[a,b]\subset(1/2,1)$. Set
	\begin{equation}\label{eq:Mstar-def}
		M^{\ast}_{\mathrm{sc}}([a,b])\;:=\;\frac{16}{\pi\,(2a-1)^{3}}\;+\;\frac{256}{\pi\,(4a+1-2b)^{3}}.
	\end{equation}
	Then
	\begin{enumerate}
		\item[(i)] For every $\sigma\in[a,b]$ the map $s\mapsto E_{s}[\Phi_{\sigma}]$ belongs to $C^{2}([a,b])$, and
		      \begin{equation}\label{eq:phisigma-2}
			      \partial_{s}^{2} E_{s}[\Phi_{\sigma}] \;=\; \frac{1}{\pi}\int_{\R}(\log|\xi|)^{2}\,|\xi|^{2s-2}\,|\widehat{\Phi_{\sigma}'}(\xi)|^{2}\,d\xi\;\ge\;0.
		      \end{equation}
		\item[(ii)] The right-hand side of~\eqref{eq:phisigma-2} satisfies, uniformly for $\sigma,s\in[a,b]$,
		      \begin{equation*}
			      \partial_{s}^{2} E_{s}[\Phi_{\sigma}]\;\le\;M^{\ast}_{\mathrm{sc}}([a,b]).
		      \end{equation*}
		\item[(iii)] The energy $\Es$ is semiconcave on $[a,b]$, in the sense that
		      \begin{equation}\label{eq:semiconc-pt}
			      \Es(s)\;\le\;\Es(\sigma)+J(\sigma,\Phi_{\sigma})\,(s-\sigma)+\tfrac{1}{2}M^{\ast}_{\mathrm{sc}}([a,b])\,(s-\sigma)^{2}
			      \qquad(\sigma,s\in[a,b]).
		      \end{equation}
	\end{enumerate}
\end{lemma}

\begin{proof}
	\emph{Step 1: Differentiation with respect to $s$.}
	Fix $\sigma\in[a,b]$.  By Lemma~\ref{lem:fourier-energy-nondecaying}, the kinetic part of $E_{s}[\Phi_{\sigma}]$ is
	$
		(4\pi)^{-1}\!\int_{\R}|\xi|^{2s-2}|\widehat{\Phi_{\sigma}'}(\xi)|^{2}\,d\xi,
	$
	whereas the potential part is independent of $s$.  The estimates established in the proof of Lemma~\ref{lem:layer-energy-variation} give, uniformly for $\tau\in[a,b]$,
	\begin{equation}\label{eq:hf-bound}
		|\widehat{\Phi_{\tau}'}(\xi)|\le2,\qquad
		|\widehat{\Phi_{\tau}'}(\xi)|\le8\,|\xi|^{-2a}\quad(|\xi|\ge1),
	\end{equation}
	the second of which is a weaker form of~\eqref{eq:hf-fourier-bound-lem}. Combining~\eqref{eq:hf-bound} with the monotonicity in $s$ of $|\xi|^{2s-2}$ on the regions $|\xi|<1$ and $|\xi|\ge1$, respectively, yields the majorants
	\begin{equation*}
		4(\log|\xi|)^{2}|\xi|^{2a-2}
		\qquad\text{and}\qquad
		64(\log|\xi|)^{2}|\xi|^{2b-2-4a}.
	\end{equation*}
	The first is integrable on $\{|\xi|<1\}$ because $2a-1>0$, and the second is integrable on $\{|\xi|\ge1\}$ provided that $4a+1-2b>0$. This condition is automatic because $a>1/2$ and $b<1$. The same integrability conclusions hold with the factor $(\log|\xi|)^2$ replaced by $|\log|\xi||$ or $1$. We therefore have integrable majorants for the Fourier integrand and its first two derivatives with respect to $s$, uniformly for $s,\sigma\in[a,b]$. Applying the dominated-convergence criterion for differentiation under the integral sign twice gives $ \partial_{s}^{2}E_{s}[\Phi_{\sigma}]=\pi^{-1}\!\int_{\R}(\log|\xi|)^{2}|\xi|^{2s-2}|\widehat{\Phi_{\sigma}'}(\xi)|^{2}\,d\xi, $ which is~\eqref{eq:phisigma-2}. A further application of dominated convergence, with the same majorant, gives continuity of the second derivative in $s$. The corresponding first differentiation yields $\partial_{s}E_{s}[\Phi_{\sigma}]=J(s,\Phi_{\sigma})$, and the nonnegativity in~\eqref{eq:phisigma-2} is immediate. This proves~(i).

	\emph{Step 2: The uniform bound.}
	We estimate the two frequency regions using the majorants from Step~1.  The elementary identities
	\begin{align*}
		 & \int_{0}^{1}(\log\xi)^{2}\,\xi^{c}\,d\xi \;=\; \frac{2}{(c+1)^{3}}\quad(c>-1),         \\
		 & \int_{1}^{+\infty}(\log\xi)^{2}\,\xi^{c}\,d\xi \;=\; \frac{2}{(-(c+1))^{3}}\quad(c<-1)
	\end{align*}
	follow from the substitutions $\xi=e^{-u}$ and $\xi=e^{u}$, respectively. Applying them with $c=2a-2$ and $c=2b-2-4a$ gives
	\begin{align*}
		\frac{1}{\pi}\int_{|\xi|<1}\!(\log|\xi|)^{2}|\xi|^{2a-2}\cdot 4\,d\xi
		 & \;=\; \frac{8}{\pi}\cdot\frac{2}{(2a-1)^{3}}\;=\;\frac{16}{\pi(2a-1)^{3}}, \\
		\frac{1}{\pi}\int_{|\xi|\ge 1}\!(\log|\xi|)^{2}|\xi|^{2b-2-4a}\cdot 64\,d\xi
		 & \;=\;\frac{256}{\pi(4a+1-2b)^{3}}.
	\end{align*}
	Adding these estimates shows that the right-hand side of~\eqref{eq:phisigma-2} is bounded above by~\eqref{eq:Mstar-def}, uniformly for $s,\sigma\in[a,b]$. This proves~(ii).

	\emph{Step 3: Semiconcavity.}
	Fix $\sigma\in[a,b]$ and set $\phi_{\sigma}(s):=E_{s}[\Phi_{\sigma}]$.  By~(i) and~(ii), $\phi_{\sigma}\in C^{2}([a,b])$ with $\phi_{\sigma}'(s)=J(s,\Phi_{\sigma})$ and $\phi_{\sigma}''\le M^{\ast}_{\mathrm{sc}}([a,b])$, so Taylor's theorem gives
	\begin{equation*}
		\phi_{\sigma}(s)\;\le\;\phi_{\sigma}(\sigma)+J(\sigma,\Phi_{\sigma})(s-\sigma)+\tfrac{1}{2}M^{\ast}_{\mathrm{sc}}([a,b])(s-\sigma)^{2}
		\qquad(s\in[a,b]).
	\end{equation*}
	By Theorem~\ref{thm:PSV}, applied at the parameter $s$ with competitor $\Phi_{\sigma}$, one has $\Es(s)\le E_{s}[\Phi_{\sigma}]=\phi_{\sigma}(s)$, while $\phi_{\sigma}(\sigma)=E_{\sigma}[\Phi_{\sigma}]=\Es(\sigma)$. This proves~(iii).
\end{proof}

For the inner interval~\eqref{eq:interior-interval} we apply Lemma~\ref{lem:semiconcave} separately on each cell, over $[a_{i},b_{i}] \supset [s_{i},s_{i+1}]$. This produces, for each $i$, the curvature bound $M^{\ast}_{\mathrm{sc}}([a_{i},b_{i}])$. No global constant is used because $M^{\ast}_{\mathrm{sc}}([a,b])$ grows without bound as $a\downarrow 1/2$.

\subsection{The fixed-point argument around the approximate solution}\label{ss:inverse-ap}

We next record the inverse estimate needed for the contraction argument. At a grid value $s=s_{i}$, the relevant operator is not the exact linearization $\fl+3\Phi_{s}^{2}-1$ at the true layer, but the linearization at the explicit approximate profile
\begin{equation*}
	G_{s}'(0)
	\;=\;
	\fl+V_{s}^{\mathrm{ap}},
	\qquad
	V_{s}^{\mathrm{ap}}(x):=3\bigl(\Phi_{s}^{\mathrm{ap}}(x)\bigr)^{2}-1,
\end{equation*}
acting from $H^{2s}_{\mathrm{odd}}(\R)$ to $L^{2}_{\mathrm{odd}}(\R)$. Since $V_{s}^{\mathrm{ap}}$ is an even real bounded function, $G_{s}'(0)$ is self-adjoint on the odd subspace with domain $H^{2s}_{\mathrm{odd}}(\R)$.

The following elementary proposition turns an $L^{2}$ coercivity bound into the $L^{2}\to H^{2s}$ inverse bound required in Proposition~\ref{prop:layer-radius}.

\begin{proposition}\label{prop:K0-from-coercivity}
	Let $s\in(1/2,\overline{s}]$, and assume that two numbers $\gamma>0$ and $M>0$ satisfy
	\begin{align}
		\langle G_{s}'(0)u,u\rangle_{L^{2}}
		 & \ge \gamma\,\|u\|_{L^{2}}^{2}
		\qquad\text{for every }u\in H^{2s}_{\mathrm{odd}}(\R), \label{eq:Gprime-coercive} \\
		\|V_{s}^{\mathrm{ap}}\|_{L^{\infty}(\R)}
		 & \le M. \label{eq:Vap-M}
	\end{align}
	Then $G_{s}'(0)\colon H^{2s}_{\mathrm{odd}}(\R)\to L^{2}_{\mathrm{odd}}(\R)$ is invertible and
	\begin{equation}\label{eq:K0-elliptic-bound}
		\bigl\|G_{s}'(0)^{-1}\bigr\|_{L^{2}\to H^{2s}}
		\;\le\;
		2^{s-\frac12}
		\left(\gamma^{-2}+\left(1+\frac{M}{\gamma}\right)^{2}\right)^{1/2}.
	\end{equation}
	In particular, if $\gamma\ge 9/10$ and $M\le4$, then the right-hand side of~\eqref{eq:K0-elliptic-bound} is $<9$ for every $s\in(1/2,\overline{s}]$.
\end{proposition}

\begin{proof}
	The coercivity bound~\eqref{eq:Gprime-coercive} and the spectral theorem give invertibility on $L^{2}_{\mathrm{odd}}$, with
	\begin{equation}\label{eq:L2-inverse-from-gamma}
		\|u\|_{L^{2}}\le\gamma^{-1}\|f\|_{L^{2}},
		\qquad u:=G_{s}'(0)^{-1}f .
	\end{equation}
	The equation $G_{s}'(0)u=f$ gives
	\begin{equation*}
		\fl u=f-V_{s}^{\mathrm{ap}}u,
	\end{equation*}
	and therefore, by~\eqref{eq:Vap-M} and~\eqref{eq:L2-inverse-from-gamma},
	\begin{equation}\label{eq:fl-u-bound}
		\|\fl u\|_{L^{2}}
		\le \|f\|_{L^{2}}+M\|u\|_{L^{2}}
		\le \left(1+\frac{M}{\gamma}\right)\|f\|_{L^{2}} .
	\end{equation}
	Finally, for $s\in[1/2, 1)$ and $t\ge0$,
	\begin{equation*}
		(1+t)^{2s}\le 2^{2s-1}(1+t^{2s}).
	\end{equation*}
	With $t=|\xi|^{2}$, Plancherel gives
	\begin{equation*}
		\|u\|_{H^{2s}}^{2}
		\le
		2^{2s-1}\bigl(\|u\|_{L^{2}}^{2}+\|\fl u\|_{L^{2}}^{2}\bigr).
	\end{equation*}
	Substituting~\eqref{eq:L2-inverse-from-gamma} and~\eqref{eq:fl-u-bound} proves~\eqref{eq:K0-elliptic-bound}. For $s\le\overline{s}=0.980$, $\gamma\ge9/10$, and $M\le4$, the bound is at most
	\begin{equation*}
		2^{0.480}
		\left(\left(\frac{10}{9}\right)^{2}+\left(1+\frac{40}{9}\right)^{2}\right)^{1/2}
		<9 .
	\end{equation*}
\end{proof}

By Proposition~\ref{prop:K0-from-coercivity}, the inverse bound for $G_{s}'(0)$ reduces to a coercivity estimate on the odd subspace. We obtain this coercivity from a Schur inequality which we now set up.

Write $A:=\fl+\tfrac{11}{10}$, a positive Fourier multiplier with symbol $|\xi|^{2s}+\tfrac{11}{10}$. Its full-line convolution kernel is
\begin{equation}
	R_{s}(z):=
	\frac{1}{2\pi} \int_{-\infty}^{+\infty} \frac{e^{iz\xi}}{|\xi|^{2s} + \frac{11}{10}} \, d\xi
	= \frac{1}{\pi}\int_{0}^{+\infty}\frac{\cos(z\xi)}{\xi^{2s}+\frac{11}{10}}\,d\xi,
\end{equation}
and the associated odd half-line kernel is
\begin{equation}
	K_{s}(x,y):=R_{s}(x-y)-R_{s}(x+y)\qquad(x,y>0),
\end{equation}
so that, for $f\in L^{2}(0,+\infty)$ extended oddly to $\R$, the restriction to $x>0$ of $A^{-1}f$ equals $\int_{0}^{+\infty}K_{s}(x,y)f(y)\,dy$.

Use the principal branch of $\zeta^{2s}$. For $z\ge0$, integrate $e^{iz\zeta}/(\zeta^{2s}+\frac{11}{10})$ over an indented quarter-circle contour in the first quadrant, with radii $\varepsilon$ and $R$. Since $s\in(1/2,\overline{s}]\subset(1/2,1)$, the denominator has no zero in this quadrant. Letting $\varepsilon\to0$ and $R\to+\infty$, the two circular integrals vanish (the outer one is $O(R^{1-2s})$ when $z=0$, and Jordan's lemma applies when $z>0$). Thus the integral along the positive real axis equals the integral along the positive imaginary axis. Substituting $\zeta = ir$ ($r > 0$, $d\zeta = i \, dr$), we obtain:
\begin{equation*}
	\int_{0}^{+\infty} \frac{e^{iz\xi}}{\xi^{2s} + \frac{11}{10}} \, d\xi
	= i \int_{0}^{+\infty} \frac{e^{-zr}}{(ir)^{2s} + \frac{11}{10}} \, dr.
\end{equation*}
Taking real parts and using
\begin{equation*}
	\frac{i}{(ir)^{2s} + \frac{11}{10}}
	= \frac{r^{2s}\sin(\pi s) + i\left(r^{2s}\cos(\pi s) + \frac{11}{10}\right)}
	{r^{4s} + \frac{11}{5} r^{2s}\cos(\pi s) + \frac{121}{100}}
\end{equation*}
yields
\begin{equation*}
	R_s(z) = \frac{\sin(\pi s)}{\pi} \int_{0}^{+\infty}
	\frac{r^{2s}e^{-zr}}{r^{4s} + \frac{11}{5} r^{2s}\cos(\pi s) + \frac{121}{100}} \, dr.
\end{equation*}
Thus $R_s$ is positive, even, and strictly decreasing on $[0,+\infty)$. Since $|x-y|<x+y$ for $x,y>0$, we conclude that $K_s(x,y)>0$ for all $x,y>0$.

Given a bounded profile $\Phi$ and a positive weight $w(x)$, set $q(x):=\bigl(3(1-\Phi(x)^{2})_{+}\bigr)^{1/2}$. The only hypothesis below is the pointwise Schur inequality
\begin{equation}\label{eq:Gprime-cert-BS}
	q(x)\int_{0}^{+\infty}K_{s}(x,y)\,q(y)\,w(y)\,dy
	\;\le\;\theta\,w(x)
	\qquad(x>0),
\end{equation}
for some number $\theta>0$. For the explicit approximate profiles $\Phi^{\mathrm{ap}}_{s_{i}}$ at the grid points, \eqref{eq:Gprime-cert-BS} is verified with an explicit weight $w$ and a value $\theta<1$ by interval arithmetic. That verification is entirely finite and is carried out separately. Here we only use that, once~\eqref{eq:Gprime-cert-BS} holds with $\theta<1$, the operator $G_{s}'(0)$ is coercive and invertible with a controlled inverse.

The strategy now is to avoid inverting the variable-coefficient operator $G_{s}'(0)$ directly. Instead, we subtract the desired spectral gap and write
\begin{equation*}
	G_{s}'(0)-\frac{9}{10}=A-\bigl(2-V_{s}^{\mathrm{ap}}\bigr).
\end{equation*}
The constant coefficient operator $A$ has the explicit inverse kernel $K_s$, while the second term is treated as an error. Its part that can decrease the quadratic form is bounded by $q^{2}=3(1-(\Phi_{s}^{\mathrm{ap}})^{2})_{+}$. The Schur inequality says precisely that the operator $qA^{-1}q$ has norm at most $\theta<1$. Equivalently, the error $A^{-1/2}q^{2}A^{-1/2}$ has the same bound. The error can therefore be absorbed, leaving the spectral gap $9/10$ and hence the required inverse bound for $G_{s}'(0)$.

\begin{proposition}\label{prop:Gprime-coercivity-cert}
	Let $s\in(1/2,\overline{s}]$, let $\Phi^{\mathrm{ap}}_{s}$ be a bounded odd profile with
	\begin{equation*}
		\bigl\|V_{s}^{\mathrm{ap}}\bigr\|_{L^{\infty}(\R)}
		=\bigl\|3(\Phi^{\mathrm{ap}}_{s})^{2}-1\bigr\|_{L^{\infty}(\R)}\le4,
	\end{equation*}
	and suppose the Schur inequality~\eqref{eq:Gprime-cert-BS}, with $\Phi=\Phi^{\mathrm{ap}}_{s}$, holds for some positive weight $w$ and some $\theta\in(0,1)$. Then
	\begin{equation}\label{eq:Gprime-cert-targets}
		\langle G_{s}'(0)u,u\rangle_{L^{2}}
		\ge \frac{9}{10}\,\|u\|_{L^{2}}^{2}
		\qquad\text{for all }u\in H^{2s}_{\mathrm{odd}}(\R),
	\end{equation}
	and thus $G_{s}'(0)\colon H^{2s}_{\mathrm{odd}}(\R)\to L^{2}_{\mathrm{odd}}(\R)$ is invertible with $ K_{0}(s) =\bigl\|(G_{s}'(0))^{-1}\bigr\|_{L^{2}\to H^{2s}} \le 9 . $
\end{proposition}

\begin{proof}
	Write $\Phi=\Phi^{\mathrm{ap}}_{s}$ and $V=3\Phi^{2}-1$. By hypothesis the pointwise inequality~\eqref{eq:Gprime-cert-BS} holds. We first turn it into an operator bound.

	Set
	\begin{equation*}
		H(x,y):=
		\left(3\bigl(1-\Phi(x)^{2}\bigr)_{+}\right)^{1/2}
		K_{s}(x,y)
		\left(3\bigl(1-\Phi(y)^{2}\bigr)_{+}\right)^{1/2}.
	\end{equation*}
	The kernel $H$ is symmetric and nonnegative. The pointwise inequality~\eqref{eq:Gprime-cert-BS} is exactly
	\begin{equation*}
		\int_{0}^{+\infty} H(x,y)w(y)\,dy\le\theta w(x)
		\qquad(x>0).
	\end{equation*}
	Schur's test gives $\|H\|_{L^{2}(0,+\infty)\to L^{2}(0,+\infty)}\le\theta$. Indeed for $f\in L^{2}(0,+\infty)$, Cauchy--Schwarz gives
	\begin{equation*}
		\left|\int_{0}^{+\infty}H(x,y)f(y)\,dy\right|^{2}
		\le
		\left(\int_{0}^{+\infty}H(x,y)w(y)\,dy\right)
		\left(\int_{0}^{+\infty}H(x,y)\frac{|f(y)|^{2}}{w(y)}\,dy\right).
	\end{equation*}
	After integration in $x$, the first factor is bounded by $\theta w(x)$. Since $H(x,y)=H(y,x)$, the same bound with $x$ and $y$ interchanged yields
	\begin{equation*}
		\|Hf\|_{L^{2}(0,+\infty)}^{2}
		\le
		\theta\int_{0}^{+\infty}\frac{|f(y)|^{2}}{w(y)}
		\left(\int_{0}^{+\infty}H(x,y)w(x)\,dx\right)dy
		\le
		\theta^{2}\|f\|_{L^{2}(0,+\infty)}^{2}.
	\end{equation*}
	Identifying an odd function on $\R$ with its restriction to $x>0$ multiplies both the domain and range norms by the same factor. Since $K_{s}$ is the odd half-line kernel of $A^{-1}$, this yields
	\begin{equation}\label{eq:Gprime-schur-operator-bound}
		\left\|
		\left(3(1-\Phi^{2})_{+}\right)^{1/2}
		A^{-1}
		\left(3(1-\Phi^{2})_{+}\right)^{1/2}
		\right\|_{L^{2}_{\mathrm{odd}}\to L^{2}_{\mathrm{odd}}}
		\le \theta<1 .
	\end{equation}

	It remains to derive the coercivity estimate. For $u\in H^{2s}_{\mathrm{odd}}(\R)$,
	\begin{equation*}
		\left\langle Au,u\right\rangle
		=
		\|(-\Delta)^{s/2}u\|_{L^{2}}^{2}
		+\frac{11}{10}\|u\|_{L^{2}}^{2}\ge0 .
	\end{equation*}
	Since $2-V=3(1-\Phi^{2})\le 3(1-\Phi^{2})_{+}$, we have
	\begin{align*}
		\left\langle\left(G_{s}'(0)-\frac{9}{10}\right)u,u\right\rangle
		 & =
		\left\langle Au,u\right\rangle
		-\int_{\R}\bigl(2-V(x)\bigr)|u(x)|^{2}\,dx \\
		 & \ge
		\left\langle Au,u\right\rangle
		-\int_{\R}3\bigl(1-\Phi(x)^{2}\bigr)_{+}|u(x)|^{2}\,dx .
	\end{align*}
	The last integral is controlled by~\eqref{eq:Gprime-schur-operator-bound}. Indeed, the pointwise definition gives $0\le q\le\sqrt{3}$; thus multiplication by $q$ is bounded, and the operators
	\begin{equation*}
		A^{-1/2}q^{2}A^{-1/2}
		\quad\text{and}\quad
		qA^{-1}q
	\end{equation*}
	have the same norm: if $T=qA^{-1/2}$, then they are $T^{*}T$ and $TT^{*}$, and both norms equal $\|T\|^{2}$. Hence~\eqref{eq:Gprime-schur-operator-bound} implies
	\begin{equation*}
		\int_{\R}3\bigl(1-\Phi(x)^{2}\bigr)_{+}|u(x)|^{2}\,dx
		=
		\langle A^{-1/2}q^{2}A^{-1/2}A^{1/2}u,A^{1/2}u\rangle
		\le
		\theta\,\langle Au,u\rangle .
	\end{equation*}
	Therefore
	\begin{equation*}
		\left\langle\left(G_{s}'(0)-\frac{9}{10}\right)u,u\right\rangle
		\ge
		(1-\theta)\langle Au,u\rangle
		\ge0,
	\end{equation*}
	and so
	\begin{equation*}
		\langle G_{s}'(0)u,u\rangle
		\ge
		\frac{9}{10}\|u\|_{L^{2}}^{2},
	\end{equation*}
	which is~\eqref{eq:Gprime-cert-targets}. Together with the bound $\|V_{s}^{\mathrm{ap}}\|_{L^{\infty}}\le4$, Proposition~\ref{prop:K0-from-coercivity} applied with $\gamma=9/10$ and $M=4$ gives $K_0(s) \leq 9$.
\end{proof}

The remaining task is to extract a usable sign bound for $J(s_{i},\Phi_{s_{i}})$ at each grid point. The true Cabr\'e--Sire layer $\Phi_{s_{i}}$ is not available in closed form, but the interval-arithmetic verification produces an approximate layer $\Phi^{\mathrm{ap}}_{s_{i}}$ together with a rigorous bound on its distance to the true layer. We first establish an explicit radius bound for the layer error. The value of $J$ is then transferred from the approximate profile to the true layer by the estimate of Lemma~\ref{lem:LJ-polarization}.

Because of the boundary conditions $\Phi_{s}(\pm\infty)=\pm1$, the layer itself does not belong to a Sobolev space. We therefore formulate the fixed-point problem for the deviation $\Phi_{s}-\Phi^{\mathrm{ap}}_{s}$ from a fixed odd reference profile $\Phi^{\mathrm{ap}}_{s}$ carrying the correct limits. The natural space for this deviation is $H^{2s}(\R)$, and restriction to the odd subspace removes the even translation mode of the linearized operator at the true layer. By Lemma~\ref{lem:sobolev-constants}, the sharp constant in the embedding $H^{2s}\hookrightarrow L^{\infty}$ is $C^{\ast}(2s)$. On $(1/2,\overline{s}]$ the function $s\mapsto C^{\ast}(2s)$ is bounded, decreasing, and explicitly evaluable in interval arithmetic.

\begin{proposition}\label{prop:layer-radius}
	Let $s\in(1/2,\overline{s}]$, and let $\Phi^{\mathrm{ap}}_{s}$ be an odd, smooth, bounded approximate profile with limits $\Phi^{\mathrm{ap}}_{s}(\pm\infty)=\pm1$ whose residual $F_{s}(\Phi^{\mathrm{ap}}_{s}):=\fl\Phi^{\mathrm{ap}}_{s}-\Phi^{\mathrm{ap}}_{s}+(\Phi^{\mathrm{ap}}_{s})^{3}$ lies in $L^{2}_{\mathrm{odd}}(\R)$. Define
	\begin{equation*}
		G_{s}\colon H^{2s}_{\mathrm{odd}}(\R)\longrightarrow L^{2}_{\mathrm{odd}}(\R),\qquad
		G_{s}(v)\;:=\;F_{s}(\Phi^{\mathrm{ap}}_{s}+v),
	\end{equation*}
	together with the constants
	\begin{equation}\label{eq:delta-K-def}
		\delta(s) \;:=\; \|G_{s}(0)\|_{L^{2}(\R)},\qquad
		K_{0}(s) \;:=\; \bigl\|G_{s}'(0)^{-1}\bigr\|_{L^{2}\to H^{2s}},
	\end{equation}
	the latter under the assumption $K_{0}(s)<+\infty$. Define
	\begin{equation}\label{eq:omega-def}
		\omega(s)\;:=\;6\,C^{\ast}(2s)\bigl(\|\Phi^{\mathrm{ap}}_{s}\|_{L^{\infty}}+C^{\ast}(2s)\bigr).
	\end{equation}
	Assume the smallness condition
	\begin{equation}\label{eq:contraction-condition}
		\alpha(s) \;:=\; 2\,K_{0}(s)^{2}\,\omega(s)\,\delta(s)\;<\;1.
	\end{equation}
	Set
	\begin{equation*}
		\tau(s)\;:=\;\frac{1-\sqrt{1-\alpha(s)}}{K_{0}(s)\,\omega(s)}.
	\end{equation*}
	Suppose also that $\tau(s)\le1$. Then $G_{s}$ has a unique zero $v_{s}$ in the closed ball $\{\|v\|_{H^{2s}}\le\tau(s)\}$, the function $\Phi_{s}:=\Phi^{\mathrm{ap}}_{s}+v_{s}$ is the Cabr\'e--Sire layer, and
	\begin{equation}\label{eq:layer-radius-conclusion}
		\|\Phi_{s}-\Phi^{\mathrm{ap}}_{s}\|_{H^{2s}(\R)}\;=\;\|v_{s}\|_{H^{2s}(\R)}\;\le\;\tau(s),
		\qquad
		\|\Phi_{s}-\Phi^{\mathrm{ap}}_{s}\|_{L^{\infty}(\R)}\;\le\;C^{\ast}(2s)\,\tau(s).
	\end{equation}
\end{proposition}

\begin{proof}
	Since $s>1/2$, the embedding $H^{2s}(\R)\hookrightarrow L^{\infty}(\R)$ makes $G_{s}$ a smooth map $H^{2s}_{\mathrm{odd}}(\R)\to L^{2}_{\mathrm{odd}}(\R)$; the preservation of oddness follows from the oddness of $\Phi^{\mathrm{ap}}_{s}$ and the reflection invariance of $\fl$. Moreover
	\begin{equation*}
		G_{s}'(0)=\fl-1+3(\Phi^{\mathrm{ap}}_{s})^{2},
	\end{equation*}
	and the inverse bound in~\eqref{eq:delta-K-def} is one of the hypotheses.

	We work on the closed unit ball $\{\|v\|_{H^{2s}}\le1\}$. For $v_{1},v_{2}$ in that ball, the difference $G_{s}'(v_{1})-G_{s}'(v_{2})$ acts by pointwise multiplication by $3(v_{1}-v_{2})(2\Phi^{\mathrm{ap}}_{s}+v_{1}+v_{2})$; bounding its operator norm $H^{2s}\to L^{2}$ by the $L^{\infty}$ norm of the multiplier and using Lemma~\ref{lem:sobolev-constants} gives
	\begin{align*}
		\|G_{s}'(v_{1})-G_{s}'(v_{2})\|_{H^{2s}\to L^{2}}
		 & \le 6\,C^{\ast}(2s)\bigl(\|\Phi^{\mathrm{ap}}_{s}\|_{L^{\infty}}+C^{\ast}(2s)\bigr)\,
		\|v_{1}-v_{2}\|_{H^{2s}} .
	\end{align*}
	Thus, by~\eqref{eq:omega-def}, $G_{s}'$ is $\omega(s)$-Lipschitz on the unit ball. Since $\tau(s)\le1$, the same estimate holds on the closed ball of radius $\tau(s)$.

	Now we write the contraction argument. Set
	\begin{equation*}
		X_{s}:=H^{2s}_{\mathrm{odd}}(\R),\qquad
		L_{s}:=G_{s}'(0),
	\end{equation*}
	and introduce the nonlinear remainder
	\begin{equation*}
		\mathcal R_{s}(v):=G_{s}(v)-G_{s}(0)-L_{s}v.
	\end{equation*}
	For $v$ in the closed unit ball of $X_s$, the fundamental theorem of calculus and the Lipschitz estimate above give
	\begin{equation}\label{eq:nonlinear-remainder-quadratic}
		\|\mathcal R_{s}(v)\|_{L^{2}}
		\le
		\int_{0}^{1}\!\bigl\|G_{s}'(tv)-G_{s}'(0)\bigr\|_{X_s\to L^{2}}\,
		\|v\|_{X_s}\,dt
		\le \frac{\omega(s)}{2}\,\|v\|_{X_s}^{2}.
	\end{equation}
	Define
	\begin{equation*}
		T_{s}(v):=v-L_{s}^{-1}G_{s}(v)
		=-L_{s}^{-1}\bigl(G_{s}(0)+\mathcal R_{s}(v)\bigr).
	\end{equation*}
	A fixed point of $T_s$ is equivalent to a zero of $G_s$. The definitions of $\alpha(s)$ and $\tau(s)$ yield the identities
	\begin{equation}\label{eq:tau-majorant-identities}
		K_{0}(s)\delta(s)
		+\frac{K_{0}(s)\omega(s)}{2}\,\tau(s)^{2}
		=\tau(s),
		\qquad
		K_{0}(s)\omega(s)\tau(s)
		=1-\sqrt{1-\alpha(s)}<1.
	\end{equation}
	It follows from~\eqref{eq:nonlinear-remainder-quadratic} and the first identity in~\eqref{eq:tau-majorant-identities} that $T_s$ maps the closed ball
	\begin{equation*}
		B_{\tau(s)}:=\{v\in X_s:\|v\|_{X_s}\le\tau(s)\}
	\end{equation*}
	into itself.

	To prove the contraction, let $v_1,v_2\in B_{\tau(s)}$. The line segment joining $v_1$ and $v_2$ lies in $B_{\tau(s)}$, and hence
	\begin{equation*}
		\mathcal R_s(v_1)-\mathcal R_s(v_2)
		=
		\int_0^1
		\bigl(G_s'(v_2+t(v_1-v_2))-L_s\bigr)(v_1-v_2)\,dt.
	\end{equation*}
	Therefore,
	\begin{equation*}
		\|T_s(v_1)-T_s(v_2)\|_{X_s}
		\le K_0(s)\omega(s)\tau(s)\,\|v_1-v_2\|_{X_s}.
	\end{equation*}
	By the second identity in~\eqref{eq:tau-majorant-identities}, the coefficient on the right is strictly smaller than $1$. Since $B_{\tau(s)}$ is complete, Banach's fixed-point theorem furnishes a unique $v_s\in B_{\tau(s)}$ satisfying $T_s(v_s)=v_s$, or equivalently $G_s(v_s)=0$. In particular, $\|v_s\|_{H^{2s}}\le\tau(s)$, which proves the first equality and inequality in~\eqref{eq:layer-radius-conclusion}.

	Finally, we identify the obtained solution with the Cabr\'e--Sire layer. Put $\Phi_{s}:=\Phi^{\mathrm{ap}}_{s}+v_{s}$. Then $F_{s}(\Phi_{s})=G_{s}(v_{s})=0$, i.e.\ $\fl\Phi_{s}=\Phi_{s}-\Phi_{s}^{3}$ on $\R$. Since $s>1/2$, the embedding $H^{2s}(\R)\hookrightarrow C_{0}(\R)$ gives $v_{s}(x)\to0$ as $|x|\to\infty$, so $\Phi_{s}$ inherits the limits $\Phi_{s}(\pm\infty)=\Phi^{\mathrm{ap}}_{s}(\pm\infty)=\pm1$. Moreover $v_{s}\in H^{2s}_{\mathrm{odd}}$ makes $\Phi_{s}$ odd, hence $\Phi_{s}(0)=0$. Thus $\Phi_{s}$ is an odd, bounded solution of the layer equation with the correct limits and normalization. By the uniqueness statement of Theorem~\ref{thm:CS}, it is the Cabr\'e--Sire layer. Finally the $H^{2s}\hookrightarrow L^{\infty}$ embedding with constant $C^{\ast}(2s)$ from Lemma~\ref{lem:sobolev-constants} turns $\|v_{s}\|_{H^{2s}}\le\tau(s)$ into the second inequality of~\eqref{eq:layer-radius-conclusion}.
\end{proof}

For any approximate layer satisfying the hypotheses of Proposition~\ref{prop:layer-radius}, set
\begin{equation}\label{eq:LJ-const-def}
	L_{J}(s):=2\,C_{\star}(s)\,\|\Phi^{\mathrm{ap}}_{s}\|_{\star}(s)\,+\,C_{\star}(s)^{2}\,\tau(s),
\end{equation}
where $\|\cdot\|_{\star}(s)$ and $C_{\star}(s)$ are defined in~\eqref{eq:starnorm} and~\eqref{eq:Cstar-LJ-def}, respectively.

\subsection{Strict decrease on the interior interval}\label{ss:realization}

We now state and prove the per-cell criterion on $I_{0}=[0.530,0.980]$. The interval-arithmetic verification of the conditions~\textup{(C1)--(C4)} below is described in Appendix~\ref{app:implementation}. These four conditions collect all the statements that we need to verify with computer assistance.

\begin{proposition}\label{prop:cells-I0}
	Let $I_{0}=[0.530,0.980]$, and let
	\begin{equation}\label{eq:grid-spec}
		s_{0}<\underline{s}<s_{1}<\cdots<s_{N-1}<\overline{s}<s_{N},
		\qquad N=55,
	\end{equation}
	be the fixed grid of $56$ points specified in the accompanying code. For every $i\in\{0,\dots,N-1\}$, let $[a_{i},b_{i}]\subset(1/2,1)$ be a closed subinterval containing $[s_{i},s_{i+1}]$ in its interior. Then for each $i\in\{0,\dots,N-1\}$, the interval-arithmetic certificate establishes:
	\begin{enumerate}
		\item[(C1)] The Schur inequality~\eqref{eq:Gprime-cert-BS} holds at $s=s_{i}$ with $\Phi=\Phi^{\mathrm{ap}}_{s_{i}}$, an explicit positive weight $w_{i}$, a constant $\theta_{i}<1$, and
		      $\|V^{\mathrm{ap}}_{s_{i}}\|_{L^{\infty}(\R)}\le4$.
		\item[(C2)] One has
		      \begin{equation*}
			      \alpha(s_i)<1,\qquad \tau(s_i)\le1.
		      \end{equation*}
		\item[(C3)] Defining
		      $
			      \eta_{i}\;:=\;|J(s_{i},\Phi^{\mathrm{ap}}_{s_{i}})|\,-\,L_{J}(s_{i})\,\tau(s_{i}),
		      $
		      one has
		      \begin{equation*}
			      J(s_{i},\Phi^{\mathrm{ap}}_{s_{i}})<0,\qquad \eta_{i}>0.
		      \end{equation*}
		\item[(C4)] One has
		      \begin{equation*}
			      \frac{2M^{\ast}_{\mathrm{sc}}([a_{i},b_{i}])\,(s_{i+1}-s_{i})}{\eta_i}<1.
		      \end{equation*}
	\end{enumerate}
	Consequently,
	\begin{equation}\label{eq:I0-slope}
		\Es(s')-\Es(s)\;\le\;-\,\tfrac{1}{2}\,\eta_{\min}\,(s'-s),\qquad \eta_{\min}:=\min_{0\le i<N}\eta_{i},
	\end{equation}
	for every $s,s'\in I_{0}$ with $s<s'$.
\end{proposition}

\begin{proof}
	Conditions~\textup{(C1)--(C4)} are verified by interval arithmetic; we show that they imply~\eqref{eq:I0-slope}. Fix $i\in\{0,\dots,N-1\}$ and set
	\begin{equation*}
		M_{i}:=M^{\ast}_{\mathrm{sc}}([a_{i},b_{i}]).
	\end{equation*}

	By~\textup{(C1)} and Proposition~\ref{prop:Gprime-coercivity-cert}, $K_0(s_i)\le9$. Thus~\textup{(C2)} and Proposition~\ref{prop:layer-radius} yield
	\begin{equation*}
		\|\Phi_{s_i}-\Phi^{\mathrm{ap}}_{s_i}\|_{H^{2s_i}}\le\tau(s_i).
	\end{equation*}
	Combining~\eqref{eq:star-sobolev} with Lemma~\ref{lem:LJ-polarization}(ii) and~\eqref{eq:LJ-const-def} gives
	\begin{equation*}
		\bigl|J(s_i,\Phi_{s_i})-J(s_i,\Phi^{\mathrm{ap}}_{s_i})\bigr|
		\le L_J(s_i)\tau(s_i),
	\end{equation*}
	so that, by~\textup{(C3)},
	\begin{equation}\label{eq:grid-sign}
		J(s_{i},\Phi_{s_{i}})\;\le\;-\eta_{i}.
	\end{equation}

	Applying the semiconcavity inequality~\eqref{eq:semiconc-pt} of Lemma~\ref{lem:semiconcave}(iii) on $[a_{i},b_{i}]$, first to the pair $(\sigma,s)$ and then to $(s,\sigma)$, gives, for $\sigma<s$ in $[a_{i},b_{i}]$,
	\begin{equation*}
		J(s,\Phi_s)(s-\sigma)-\tfrac12M_i(s-\sigma)^2
		\;\le\;\Es(s)-\Es(\sigma)
		\;\le\;J(\sigma,\Phi_\sigma)(s-\sigma)+\tfrac12M_i(s-\sigma)^2.
	\end{equation*}
	Comparing the outer terms with $\sigma=s_{i}$ and using~\eqref{eq:grid-sign} together with $M_{i}(s_{i+1}-s_{i})\le\eta_{i}/2$, which is~\textup{(C4)}, we obtain
	\begin{equation*}
		J(s,\Phi_s)
		\;\le\;J(s_i,\Phi_{s_i})+M_i(s-s_i)
		\;\le\;-\eta_i+M_i(s_{i+1}-s_i)
		\;\le\;-\tfrac12\eta_i
		\qquad(s\in[s_{i},s_{i+1}]).
	\end{equation*}

	The cells $[s_{i},s_{i+1}]$, $0\le i<N$, cover $I_{0}$ by~\eqref{eq:grid-spec}, and Lemma~\ref{lem:envelope-paper} gives $\Es'(s)=J(s,\Phi_{s})$ on $I_{0}$. Hence
	\begin{equation*}
		\Es'(s)\;\le\;-\tfrac12\eta_{\min}\qquad(s\in I_{0}),
	\end{equation*}
	and~\eqref{eq:I0-slope} follows from the fundamental theorem of calculus. In particular, $\Es$ is strictly decreasing on $I_{0}$.
\end{proof}

The choice~\eqref{eq:grid-spec} of grid is dictated by the constraints~(C2) and~(C4). The contraction radius~$\tau(s)$ and the curvature modulus~$M^{\ast}_{\mathrm{sc}}([a_{i},b_{i}])$ both grow as $s_{i}$ approaches the upper end of $I_{0}$, driving the cells to narrow there.

\begin{proof}[Proof of Theorem~\ref{thm:mono}]
	Theorem~\ref{thm:CS}, together with the explicit profile at $s=1$, gives existence and uniqueness of $\Phi_s$. Lemma~\ref{lem:layer-energy-variation} shows that $\Es(s)$ is finite for every $s\in(1/2,1]$, while Theorem~\ref{thm:cont} gives continuity on this interval. Propositions~\ref{prop:mono-lower} and~\ref{prop:mono-upper} give strict decrease on $(1/2,0.530]$ and $[0.980,1]$, respectively, and Proposition~\ref{prop:cells-I0} gives strict decrease on $[0.530,0.980]$. Thus, for any $s<s'$, chaining across the junction points $0.530$ and $0.980$ that lie between them yields $\Es(s)>\Es(s')$.

	Finally, Theorem~\ref{thm:asymp} gives
	\begin{equation*}
		\Es(1)=\frac{2\sqrt{2}}{3},
		\qquad
		\lim_{s\to(1/2)^+}\Es(s)=+\infty.
	\end{equation*}
	Consequently, the range of $\Es$ on $(1/2,1]$ is $[2\sqrt{2}/3,+\infty)$. Setting $\Es(1/2):=+\infty$ gives the continuous strictly decreasing bijection in~\eqref{eq:homeo}, which is a homeomorphism because a continuous strictly monotone bijection between intervals has continuous inverse.
\end{proof}

\appendix
\providecommand{\Fres}{F_{s}}
\providecommand{\Phiap}{\Phi^{\mathrm{ap}}_{s}}

\section{The computer-assisted validation}\label{app:implementation}

The interior monotonicity argument of Section~\ref{sec:mono-interior} reduces, through Proposition~\ref{prop:cells-I0}, to four conditions \textup{(C1)--(C4)} verified at an explicit grid of $56$ values of~$s$ covering $I_{0}=[0.530,0.980]$. Conditions \textup{(C2)-(C4)} are quantitative and direct once the quantities
\begin{equation*}
	\delta(s),\quad K_{0}(s),\quad \omega(s),\quad J(s,\Phiap),\quad
	\|\Phiap\|_{\star}(s),\quad L_{J}(s),\quad \tau(s)
\end{equation*}
are available as rigorous enclosures. The first, the coercivity estimate \textup{(C1)}, asks for the Schur inequality~\eqref{eq:Gprime-cert-BS}, whose finite verification is explained here. Our purpose here is to record the mathematics that turns each of these into a finite computation. We also explain here all the details of the Python implementation of these rigorous computations. We attach the code as supplementary material.

Computations are carried out using rigorous interval arithmetic in the \texttt{Arb} library~\cite{ArbJohansson} at $200$-bit working precision (about $60$ decimal digits). The appendix is organized as follows. Section~\ref{app:layer} constructs the explicit approximate layer $\Phiap$ and its residual $\delta(s)$. Section~\ref{app:schur} encloses the Green kernel $R_{s}$ entering \textup{(C1)} and assembles the finite certificate for~\eqref{eq:Gprime-cert-BS}. Section~\ref{app:Jeval} evaluates the derivative functional $J$ and the weighted seminorm $\|\cdot\|_{\star}$. Section~\ref{app:code} describes the organization of the code and reports the running times.

\subsection{The approximate layer and its residual}\label{app:layer}

The fixed-point argument of Proposition~\ref{prop:layer-radius} is carried out around an explicit odd profile $\Phiap$ carrying the correct limits $\Phiap(\pm\infty)=\pm1$. We build it from an arctan profile and a basis of functions on which $\fl$ acts in closed form. The idea of expanding in rational or algebraic families on the whole line whose fractional Laplacian is available analytically goes back to the spectral methods of~\cite{CayamaCuestaHoz2020}. The specific expression we use, $(1-ix)^{-\alpha}$, is the same one exploited in the analysis of nonlocal evolution equations in~\cite{LushnikovSilantyevSiegel2021}.

\subsubsection*{The $\sigma$-basis and the index-shift law}

For a parameter $\alpha>0$ set
\begin{equation}\label{eq:app-sigma-def}
	\sigma_{\alpha}(x)\;:=\;(1+x^{2})^{-\alpha/2}\,\sin\!\bigl(\alpha\arctan
	x\bigr)\;=\;\operatorname{Im}\bigl[(1-ix)^{-\alpha}\bigr],\qquad x\in\R.
\end{equation}
Each $\sigma_{\alpha}$ is odd, smooth, and decays like $\sin(\alpha\pi/2)\,|x|^{-\alpha}$ at infinity. For every integer $k\geq0$, its $k$-th derivative is also explicit. Indeed,
\begin{equation*}
	\frac{d}{dx}(1-ix)^{-\alpha}
	=i\alpha(1-ix)^{-\alpha-1},
\end{equation*}
and induction gives
\begin{equation*}
	\frac{d^{k}}{dx^{k}}(1-ix)^{-\alpha}
	=i^{k}\frac{\Gamma(\alpha+k)}{\Gamma(\alpha)}
	(1-ix)^{-\alpha-k}.
\end{equation*}
Since
\begin{equation*}
	1-ix=(1+x^{2})^{1/2}e^{-i\arctan x}
	\quad\text{and}\quad i^{k}=e^{ik\pi/2},
\end{equation*}
taking imaginary parts yields
\begin{equation*}
	\sigma_{\alpha}^{(k)}(x)
	=\frac{\Gamma(\alpha+k)}{\Gamma(\alpha)}
	(1+x^{2})^{-(\alpha+k)/2}
	\sin\!\left(\frac{k\pi}{2}+(\alpha+k)\arctan x\right).
\end{equation*}

The whole construction relies on the following Fourier-side expressions.

\begin{lemma}\label{lem:app-sigma-ft}
	For every $\alpha>0$,
	\begin{equation}\label{eq:app-sigma-ft}
		\widehat{\sigma_{\alpha}}(\xi)
		\;=\;-\,\frac{i\pi\,\sgn(\xi)\,|\xi|^{\alpha-1}e^{-|\xi|}}{\Gamma(\alpha)},
		\qquad
		\widehat{\sigma_{\alpha}'}(\xi)
		\;=\;\frac{\pi\,|\xi|^{\alpha}e^{-|\xi|}}{\Gamma(\alpha)} .
	\end{equation}
	Consequently $\fl$ acts as a pure index shift,
	\begin{equation}\label{eq:app-indexshift}
		\fl\sigma_{\alpha}\;=\;\frac{\Gamma(2s+\alpha)}{\Gamma(\alpha)}\,
		\sigma_{2s+\alpha},
	\end{equation}
	and, writing $g(x):=(2/\pi)\arctan x$, one has $\widehat{g'}(\xi)=2e^{-|\xi|}$ and $\fl g=(2/\pi)\,\Gamma(2s)\,\sigma_{2s}$.
\end{lemma}

\begin{proof}
	From the Gamma integral, $(1-ix)^{-\alpha}=\Gamma(\alpha)^{-1}\int_{0}^{\infty} t^{\alpha-1}e^{-t}e^{itx}\,dt$. Transforming in $x$ and using $\int_{\R}e^{itx}e^{-i\xi x}\,dx=2\pi\delta(\xi-t)$ gives
	\begin{equation*}
		\widehat{(1-ix)^{-\alpha}}(\xi)
		=(2\pi/\Gamma(\alpha))\,\xi^{\alpha-1}e^{-\xi} \mathbf 1_{\{\xi>0\}},\end{equation*}
	and likewise
	\begin{equation*}
		\widehat{(1+ix)^{-\alpha}}(\xi)
		=(2\pi/\Gamma(\alpha))\,|\xi|^{\alpha-1}e^{-|\xi|}\mathbf 1_{\{\xi<0\}}.
	\end{equation*}
	Subtracting and dividing by $2i$, as in~\eqref{eq:app-sigma-def}, yields the first identity in~\eqref{eq:app-sigma-ft}. The second follows from $\widehat{\sigma_{\alpha}'}(\xi)=i\xi\,\widehat{\sigma_{\alpha}}(\xi)$. Multiplying $\widehat{\sigma_{\alpha}}$ by the multiplier $|\xi|^{2s}$ and comparing with the same formula at index $2s+\alpha$ gives~\eqref{eq:app-indexshift}. For $g$, the Fourier transform $\widehat{(1+x^{2})^{-1}}(\xi)=\pi e^{-|\xi|}$ gives $\widehat{g'}(\xi)=2e^{-|\xi|}$, and the same multiplier computation yields $\fl g = (2/\pi)\Gamma(2s)\sigma_{2s}$.
\end{proof}

\subsubsection*{The approximate layer.}

We fix a finite index set $A=\{1,2,\dots,24\}\cup\{2s+k:0\le k\le 11\}$ and define the \emph{approximate layer}
\begin{equation}\label{eq:app-phiap}
	\Phiap\;=\;g+\sum_{\alpha\in A}a_{\alpha}\,\sigma_{\alpha},
	\qquad g(x)=\tfrac{2}{\pi}\arctan x,
\end{equation}
with real coefficients $a_{\alpha}$. By construction $\Phiap$ is odd, smooth, and satisfies $\Phiap(0)=0$, $\Phiap(\pm\infty)=\pm1$. The base profile $g$ carries the boundary limits $\pm1$ since every $\sigma_{\alpha}$ decays to $0$.

The index set $A$ splits into two blocks of different role. The bulk modes $\{\sigma_{\alpha}:1\le\alpha\le24\}$, of integer index, have a fast decay that makes them efficient at fitting the profile close to zero. However, they are unable to reproduce the true slow algebraic tail. That tail is instead supplied by the tail modes $\sigma_{2s+k}$ ($0\le k\le11$), whose fractional index $2s$ matches the algebraic Cabr\'e--Sire decay $\Phi_{s}-\sgn=O(|x|^{-2s})$ of Theorem~\ref{thm:CS}.

The coefficients $\{a_{\alpha}\}$ are produced by a Gauss-Newton method used to solve the fractional Allen-Cahn equation $\Fres(\Phiap)=\fl\Phiap-\Phiap+(\Phiap)^{3}$ so that the error is minimized in $L^{2}$. The residual is measured at the $800$ positive points
\begin{equation*}
	\theta_j=\pi+\frac{\pi(j+1/2)}{800},
	\qquad x_j=-\cot(\theta_j/2),
	\qquad 0\le j<800,
\end{equation*}
and the Gauss-Newton method is used to solve the least squares problem
\begin{equation*}
	\min_{\left\{ a_\alpha \right\}} \sum_{0 \leq j < 800} \left( F_s(\Phiap) w_j \right)^{2},
\end{equation*}
where $w_j=(\tfrac12\csc^2(\theta_j/2))^{1/2}$ corresponds to the square root of the Jacobian of the change of variables $dx/d\theta$ that turns the collocation sum into a weighted $L^{2}$ residual in~$x$. Starting from $a_{\alpha}=0$, the iteration takes at most $40$ steps and stops when the weighted residual no longer decreases. Figure~\ref{fig:layer-profiles} displays the resulting profiles.

\begin{figure}[t]
	\centering
	\includegraphics[width=0.78\textwidth]{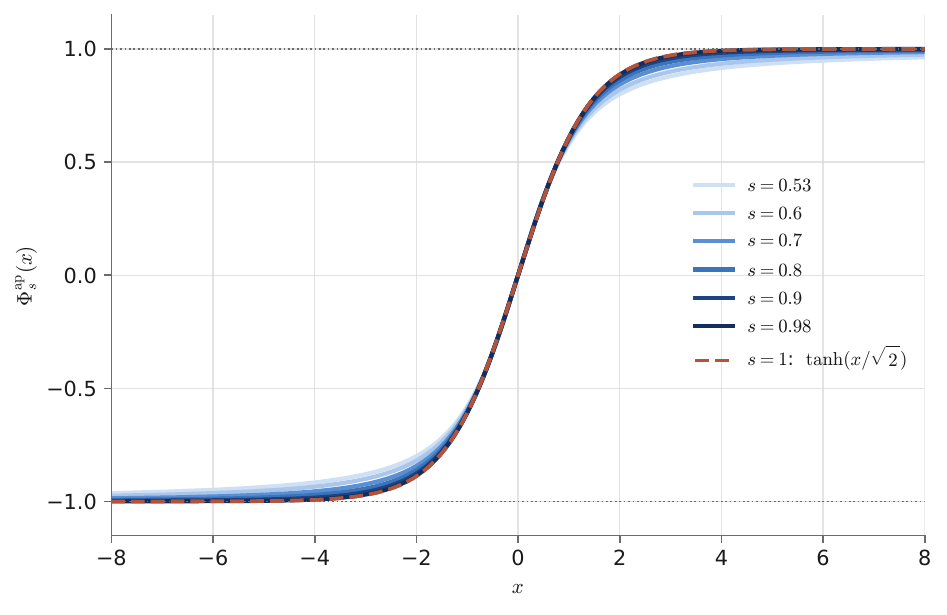}
	\caption{The approximate layer $\Phiap$ of~\eqref{eq:app-phiap},
		for several values of
		$s\in[0.53,0.98]$, against the exact profile
		$\Phi_{1}=\tanh(x/\sqrt2)$. The dependence on $s$ is carried mostly by the tails.}
	\label{fig:layer-profiles}
\end{figure}

By Lemma~\ref{lem:app-sigma-ft} the residual of the layer equation,
\begin{equation*}
	\Fres(\Phiap)\;=\;\frac{2}{\pi}\Gamma(2s)\,\sigma_{2s}
	+\sum_{\alpha\in A}a_{\alpha}\frac{\Gamma(2s+\alpha)}{\Gamma(\alpha)}\,
	\sigma_{2s+\alpha}
	-\Phiap+(\Phiap)^{3},
\end{equation*}
is an explicit function of $x$. The kinetic part is a finite sum and the algebraic nonlinear part is a polynomial in $g$ and the $\sigma_{\alpha}$. The residual is its $L^{2}$-norm, $\delta(s)=\|\Fres(\Phiap)\|_{L^{2}(\R)}$.

\subsubsection*{Validated computation of the residual}

We compute a rigorous enclosure of $\delta(s)$ by setting $G:=\Fres(\Phiap)$ and splitting its squared $L^{2}$ norm into an interior contribution and an analytic tail bound. Since $G$ is odd,
\begin{equation*}
	\|G\|_{L^{2}}^{2}=2\int_{0}^{x_{\max}}G(x)^{2}\,dx
	+2\int_{x_{\max}}^{\infty}G(x)^{2}\,dx.
\end{equation*}

For the interior contribution, divide $[0,x_{\max}]$ into cells $I=[c-r,c+r]$ of radius $r\le h/2$, with the final cell possibly going past $x_{\max}$, and write $a=c-r$ for the left endpoint of a given cell. At the midpoint $c$, we compute the normalized Taylor jets of $\Phiap$ and $\fl\Phiap$. The required coefficients are explicit: for $k\geq1$,
\begin{equation}\label{eq:app-g-jet}
	\frac{g^{(k)}(c)}{k!}
	=\frac{2}{\pi}\frac{(-1)^{k+1}}{k}
	(1+c^{2})^{-k/2}
	\sin\!\left(\frac{k\pi}{2}-k\arctan c\right),
\end{equation}
whereas, for $k\geq0$ and any index $\beta>0$,
\begin{equation}\label{eq:app-sigma-jet}
	\frac{\sigma_{\beta}^{(k)}(c)}{k!}
	=\frac{\Gamma(\beta+k)}{\Gamma(\beta)\Gamma(k+1)}
	(1+c^{2})^{-(\beta+k)/2}
	\sin\!\left(\frac{k\pi}{2}+(\beta+k)\arctan c\right).
\end{equation}
Products of normalized jets are computed by Cauchy convolution. We record the resulting recipe in detail, since it is what turns the interior contribution into a finite computation. For a function $u$ that is smooth near $c$, write
\begin{equation*}
	\mathcal T_{k}[u]:=\frac{u^{(k)}(c)}{k!},\qquad 0\le k\le N,
\end{equation*}
for its normalized Taylor jet at $c$, the dependence on the cell being suppressed from the notation. The Leibniz rule reads
\begin{equation}\label{eq:app-cauchy}
	\mathcal T_{k}[uv]=\sum_{j=0}^{k}\mathcal T_{j}[u]\,\mathcal T_{k-j}[v],
	\qquad 0\le k\le N.
\end{equation}

The two input jets are explicit. Indeed $\Phiap=g+\sum_{\alpha\in A}a_{\alpha}\sigma_{\alpha}$, while the index-shift law~\eqref{eq:app-indexshift} gives
\begin{equation*}
	\fl\Phiap
	=\frac{2}{\pi}\Gamma(2s)\,\sigma_{2s}
	+\sum_{\alpha\in A}a_{\alpha}\frac{\Gamma(2s+\alpha)}{\Gamma(\alpha)}\,
	\sigma_{2s+\alpha},
\end{equation*}
so that both functions are finite linear combinations of $g$ and of $\sigma$-modes. Their jets are therefore
\begin{equation}\label{eq:app-input-jets}
	\begin{aligned}
		\mathcal T_{k}[\Phiap]
		 & =\mathcal T_{k}[g]+\sum_{\alpha\in A}a_{\alpha}\,
		\mathcal T_{k}[\sigma_{\alpha}],                         \\
		\mathcal T_{k}[\fl\Phiap]
		 & =\frac{2}{\pi}\Gamma(2s)\,\mathcal T_{k}[\sigma_{2s}]
		+\sum_{\alpha\in A}a_{\alpha}\frac{\Gamma(2s+\alpha)}{\Gamma(\alpha)}\,
		\mathcal T_{k}[\sigma_{2s+\alpha}],
	\end{aligned}
\end{equation}
where $\mathcal T_{0}[g]=(2/\pi)\arctan c$, where $\mathcal T_{k}[g]$ is given by~\eqref{eq:app-g-jet} for $k\ge1$, and where each $\mathcal T_{k}[\sigma_{\beta}]$ is given by~\eqref{eq:app-sigma-jet}.

Three applications of~\eqref{eq:app-cauchy} now produce the jet of $G^{2}$. First
\begin{equation}\label{eq:app-cube-jet}
	\mathcal T_{k}[(\Phiap)^{2}]
	=\sum_{j=0}^{k}\mathcal T_{j}[\Phiap]\,\mathcal T_{k-j}[\Phiap],
	\qquad
	\mathcal T_{k}[(\Phiap)^{3}]
	=\sum_{j=0}^{k}\mathcal T_{j}[(\Phiap)^{2}]\,\mathcal T_{k-j}[\Phiap],
\end{equation}
then the residual jet, which is a linear combination of already computed quantities,
\begin{equation}\label{eq:app-G-jet}
	\mathcal T_{k}[G]
	=\mathcal T_{k}[\fl\Phiap]-\mathcal T_{k}[\Phiap]
	+\mathcal T_{k}[(\Phiap)^{3}],
	\qquad G=\fl\Phiap-\Phiap+(\Phiap)^{3},
\end{equation}
and finally one more convolution,
\begin{equation}\label{eq:app-G2-jet}
	\mathcal T_{k}[G^{2}]=\sum_{j=0}^{k}\mathcal T_{j}[G]\,\mathcal T_{k-j}[G],
	\qquad 0\le k\le N.
\end{equation}
Every step is carried out in interval arithmetic, so that~\eqref{eq:app-G2-jet} returns a vector of enclosures of the Taylor coefficients of $G^{2}$ at the midpoint $c$. In the certificate we take $N=40$, $h=1/50$, and $x_{\max}=200$, so that the interior contribution is assembled from $10^{4}$ cells.

Integrating the Taylor polynomial of $G^{2}$ over the symmetric cell $I$, where the odd-order terms drop out, Taylor's theorem gives
\begin{equation*}
	\left|\int_I G(x)^{2}\,dx
	-\sum_{0\leq 2m\leq N}\frac{2\,\mathcal T_{2m}[G^{2}]\,r^{2m+1}}{2m+1}\right|
	\leq \frac{2M_{N+1}r^{N+2}}{N+2},
	\qquad
	M_{N+1}:=\sup_{x\in I}\frac{|(G^{2})^{(N+1)}(x)|}{(N+1)!}.
\end{equation*}
The quantity $M_{N+1}$ is bounded without evaluating the basis functions on the entire interval. Indeed, for every $x\in I\subset[a,\infty)$ and every index $\beta>0$,
\begin{equation*}
	\frac{|\sigma_{\beta}^{(k)}(x)|}{k!}
	\leq\frac{\Gamma(\beta+k)}{\Gamma(\beta)\Gamma(k+1)}
	(1+a^{2})^{-(\beta+k)/2},
	\qquad
	\frac{|g^{(k)}(x)|}{k!}
	\leq\frac{2}{\pi k}(1+a^{2})^{-k/2}\quad(k\geq1),
\end{equation*}
together with $|g|\leq1$. These bounds may be combined exactly as in~\eqref{eq:app-input-jets}--\eqref{eq:app-G2-jet}, with every coefficient replaced by its absolute value and the difference in~\eqref{eq:app-G-jet} by a sum. Convolving them in this way, first for $G$ and then for $G^{2}$, yields nonnegative jets that dominate the normalized derivatives of $G$ and $G^{2}$. Their coefficient of order $N+1$ is a rigorous upper bound for $M_{N+1}$. Summing the resulting enclosures in interval arithmetic yields the desired enclosure of the interior contribution while avoiding the overestimation of a direct interval evaluation of $G^{2}$.

For the tail contribution, put
\begin{equation*}
	\eta:=\Phiap-1=(g-1)+\sum_{\alpha\in A}a_{\alpha}\sigma_{\alpha}.
\end{equation*}
The constant terms in $-\Phiap+(\Phiap)^{3}$ cancel, and therefore
\begin{equation}\label{eq:app-G-eta}
	G=\fl\Phiap+2\eta+3\eta^{2}+\eta^{3}.
\end{equation}
We shall separate only the $x^{-1}$ term of $G$. All remaining terms will be bounded together. This avoids computing the other coefficients in the asymptotic expansion of $G$.

We begin with a uniform estimate for a single mode. If $x\ge x_{\max}\ge1$ and $\beta\ge2s$, then
\begin{equation*}
	\arctan x=\frac{\pi}{2}-\arctan(x^{-1}),\qquad
	(1+x^{2})^{-\beta/2}\le x^{-\beta}.
\end{equation*}
Since the sine function is $1$-Lipschitz and $\arctan(x^{-1})\le x^{-1}$, it follows that
\begin{equation*}
	|\sigma_{\beta}(x)|
	\le x^{-\beta}
	\left(\left|\sin\frac{\pi\beta}{2}\right|+\frac{\beta}{x}\right)
	\le \mathcal B_{\beta}(x_{\max})x^{-2s},
\end{equation*}
where
\begin{equation}\label{eq:app-tail-mode-bound}
	\mathcal B_{\beta}(x_{\max})
	:=x_{\max}^{2s-\beta}
	\left(\left|\sin\frac{\pi\beta}{2}\right|
	+\frac{\beta}{x_{\max}}\right).
\end{equation}
The phase factor in this estimate is useful. For indices close to an even integer it records the additional decay caused by the small leading sine.

The index $1$ is the only index in $A$ below $2s$, so it must be treated separately. Set
\begin{equation*}
	A_{1}:=2\left(a_{1}-\frac{2}{\pi}\right).
\end{equation*}
We have $g(x)-1=-(2/\pi)\arctan(x^{-1})$ and $\sigma_{1}(x)=x/(1+x^{2})$. The elementary inequalities
\begin{equation*}
	0\le x^{-1}-\arctan(x^{-1})
	=\int_{0}^{x^{-1}}\frac{t^{2}}{1+t^{2}}\,dt\le\frac{x^{-3}}{3},
	\qquad
	\left|\frac{x}{1+x^{2}}-x^{-1}\right|
	=\frac{1}{x(1+x^{2})}\le x^{-3},
\end{equation*}
give
\begin{equation}\label{eq:app-tail-base}
	\left|(g(x)-1)+a_{1}\sigma_{1}(x)
	-\left(a_{1}-\frac{2}{\pi}\right)x^{-1}\right|
	\le \left(\frac{2}{3\pi}+|a_{1}|\right)x^{-3}.
\end{equation}

We next bound $\eta$. Every $\alpha\in A\setminus\{1\}$ satisfies $\alpha\ge2s$, so~\eqref{eq:app-tail-mode-bound} and~\eqref{eq:app-tail-base} imply
\begin{equation*}
	\left|\eta(x)-\left(a_{1}-\frac{2}{\pi}\right)x^{-1}\right|
	\le C_{\eta}x^{-2s},
	\qquad
	C_{\eta}:=\left(\frac{2}{3\pi}+|a_{1}|\right)x_{\max}^{2s-3}
	+\sum_{\alpha\in A\setminus\{1\}}
	|a_{\alpha}|\,\mathcal B_{\alpha}(x_{\max}).
\end{equation*}
In particular,
\begin{equation}\label{eq:app-tail-Meta}
	|\eta(x)|\le M_{\eta}x^{-1},
	\qquad
	M_{\eta}:=\left|a_{1}-\frac{2}{\pi}\right|
	+C_{\eta}x_{\max}^{1-2s}.
\end{equation}

It remains to bound the linear part of~\eqref{eq:app-G-eta}. The important point is to combine modes with the same index before taking absolute values. Let
\begin{equation*}
	D:=\{2s\}\cup\bigl(A\setminus\{1\}\bigr)\cup(2s+A),
	\qquad 2s+A:=\{2s+\alpha:\alpha\in A\},
\end{equation*}
and, for $\beta\in D$, define
\begin{equation}\label{eq:app-tail-dbeta}
	d_{\beta}:=
	\mathbf 1_{\{\beta=2s\}}\frac{2}{\pi}\Gamma(2s)
	+2\!\!\sum_{\substack{\alpha\in A\setminus\{1\}\\ \alpha=\beta}}\!\!a_{\alpha}
	+\!\!\sum_{\substack{\alpha\in A\\2s+\alpha=\beta}}\!\!
	a_{\alpha}\frac{\Gamma(2s+\alpha)}{\Gamma(\alpha)}.
\end{equation}
Empty sums are understood to be zero. By the index-shift law,
\begin{equation}\label{eq:app-tail-linear-identity}
	\fl\Phiap+2\eta-A_{1}x^{-1}
	=
	2\left((g-1)+a_{1}\sigma_{1}
	-\left(a_{1}-\frac{2}{\pi}\right)x^{-1}\right)
	+\sum_{\beta\in D}d_{\beta}\sigma_{\beta}.
\end{equation}
All indices in $D$ are at least $2s$. Hence~\eqref{eq:app-tail-mode-bound} and~\eqref{eq:app-tail-base} yield
\begin{equation}\label{eq:app-tail-Clin}
	\left|\fl\Phiap+2\eta-A_{1}x^{-1}\right|
	\le C_{\mathrm{lin}}x^{-2s},
	\qquad
	C_{\mathrm{lin}}
	:=2\left(\frac{2}{3\pi}+|a_{1}|\right)x_{\max}^{2s-3}
	+\sum_{\beta\in D}|d_{\beta}|\,\mathcal B_{\beta}(x_{\max}).
\end{equation}
The sums in~\eqref{eq:app-tail-dbeta} are formed before the absolute value in~\eqref{eq:app-tail-Clin}. Thus the estimate uses the cancellations between the shifted bulk modes and the tail modes.

The nonlinear terms require no expansion. From~\eqref{eq:app-tail-Meta}, and using $2s\le2$, we obtain
\begin{equation*}
	|3\eta^{2}+\eta^{3}|
	\le 3M_{\eta}^{2}x^{-2}+M_{\eta}^{3}x^{-3}
	\le C_{\mathrm{nl}}x^{-2s},
\end{equation*}
where
\begin{equation*}
	C_{\mathrm{nl}}
	:=3M_{\eta}^{2}x_{\max}^{2s-2}
	+M_{\eta}^{3}x_{\max}^{2s-3}.
\end{equation*}
Combining this with~\eqref{eq:app-tail-linear-identity}, we arrive at
\begin{equation*}
	G(x)=A_{1}x^{-1}+R(x),\qquad
	|R(x)|\le C_{R}x^{-2s}\quad(x\ge x_{\max}),
	\qquad
	C_{R}:=C_{\mathrm{lin}}+C_{\mathrm{nl}}.
\end{equation*}
Thus $A_{1}$ is the only asymptotic coefficient computed explicitly.

Finally,
\begin{equation*}
	G(x)^{2}
	\le A_{1}^{2}x^{-2}
	+2|A_{1}|C_{R}x^{-1-2s}
	+C_{R}^{2}x^{-4s}.
\end{equation*}
Since $s>1/2$, all three powers are integrable at infinity, and termwise integration gives
\begin{equation}\label{eq:app-tail-integral}
	\int_{x_{\max}}^{\infty}G(x)^{2}\,dx
	\le
	\frac{A_{1}^{2}}{x_{\max}}
	+\frac{|A_{1}|C_{R}}{s}\,x_{\max}^{-2s}
	+\frac{C_{R}^{2}}{4s-1}\,x_{\max}^{1-4s}.
\end{equation}
Every quantity in~\eqref{eq:app-tail-integral} is evaluated with interval arithmetic. In particular, each $d_{\beta}$ is first formed as an enclosure and only then replaced by an enclosure of its absolute value. Therefore the right-hand side of~\eqref{eq:app-tail-integral} is a rigorous upper bound. Adding it to the interior integral, restoring the factor $2$ for the two half-lines, and taking the square root yields the enclosure of $\delta(s)$.

The defect $\delta(s)$ depends only on the approximate layer $\Phiap$. The certificate evaluates the interior and tail bounds above at every grid point and then checks that the contraction parameter $\alpha(s)=2K_{0}(s)^{2}\omega(s)\delta(s)$ of~\eqref{eq:contraction-condition} is strictly smaller than~$1$. The resulting enclosures are recorded with the accompanying certificate data.

\subsection{The finite Schur certificate}\label{app:schur}\label{app:green-kernel}

We describe the finite computation used to certify the Schur bound in \textup{(C1)}. It combines rigorous enclosures of the Green kernel with an explicit positive weight, treating the compact core and the tail separately.

\subsubsection*{The Green kernel}
We now enclose the kernel underlying the coercivity certificate
\textup{(C1)}.  Recall from Section~\ref{ss:inverse-ap} the positive multiplier
$A=\fl+\tfrac{11}{10}$, its even convolution kernel
\begin{equation*}
	R_{s}(z)=\frac{1}{\pi}\int_{0}^{\infty}\frac{\cos(z\xi)}{\xi^{2s}+b}\,d\xi,
	\qquad b:=\tfrac{11}{10},
\end{equation*}
and the odd half-line kernel $K_{s}(x,y)=R_{s}(x-y)-R_{s}(x+y)$. The kernel $R_{s}$ has total mass $\int_{\R}R_{s}=\widehat{R_{s}}(0)=1/b=10/11$.

For $z>0$, choose $0<\delta<\Lambda$ and split the defining integral as
\begin{equation*}
	R_{s}(z)=\frac{1}{\pi}\left[
		\int_0^\delta\frac{\cos(z\xi)}{\xi^{2s}+b}\,d\xi
		+\int_\delta^\Lambda\frac{\cos(z\xi)}{\xi^{2s}+b}\,d\xi
		+\int_\Lambda^\infty\frac{\cos(z\xi)}{\xi^{2s}+b}\,d\xi
		\right].
\end{equation*}
The integrand is analytic on $[\delta,\Lambda]$, so the middle integral is enclosed using \texttt{Arb}'s rigorous integrator (\texttt{acb.integral} in \texttt{python-flint}). At the lower endpoint, $|\cos(z\xi)|\le 1$ and $\xi^{2s}+b\ge b$, and therefore
\begin{equation*}
	\Bigl|\int_{0}^{\delta}\frac{\cos(z\xi)}{\xi^{2s}+b}\,d\xi\Bigr|
	\le \int_{0}^{\delta}\frac{d\xi}{\xi^{2s}+b}
	\le \int_{0}^{\delta}\frac{d\xi}{b}
	=\frac{\delta}{b}.
\end{equation*}
At the upper region, integration by parts gives
\begin{equation*}
	\int_{\Lambda}^{\infty}\frac{\cos(z\xi)}{\xi^{2s}+b}\,d\xi
	=-\frac{\sin(z\Lambda)}{z(\Lambda^{2s}+b)}
	-\frac{1}{z}\int_{\Lambda}^{\infty}\sin(z\xi)
	\frac{d}{d\xi}\!\left(\frac{1}{\xi^{2s}+b}\right)d\xi.
\end{equation*}
Since $(\xi^{2s}+b)^{-1}$ is positive and decreasing to zero, the absolute integral of its derivative over $[\Lambda,\infty)$ equals $(\Lambda^{2s}+b)^{-1}$. Hence
\begin{equation*}
	\Bigl|\int_{\Lambda}^{\infty}\frac{\cos(z\xi)}{\xi^{2s}+b}\,d\xi\Bigr|
	\le\frac{1}{z(\Lambda^{2s}+b)}
	+\frac{1}{z(\Lambda^{2s}+b)}
	=\frac{2}{z(\Lambda^{2s}+b)}.
\end{equation*}
Together with the enclosure of the middle integral, these estimates give an enclosure of $R_{s}(z)$.

In the middle region, we need to compute integrals over intervals of $R_{s}$. Since $2s>1$, the function $(\xi^{2s}+b)^{-1}$ is integrable on $(0,\infty)$. Thus, for $0\le a_{1}\le a_{2}$, Fubini's theorem gives
\begin{equation}\label{eq:app-Rs-mass}
	\int_{a_{1}}^{a_{2}}R_{s}(u)\,du
	=\frac{1}{\pi}\int_{0}^{\infty}
	\left(\int_{a_{1}}^{a_{2}}\cos(u\xi)\,du\right)
	\frac{d\xi}{\xi^{2s}+b}
	=\frac{1}{\pi}\int_{0}^{\infty}
	\frac{\sin(a_{2}\xi)-\sin(a_{1}\xi)}{\xi(\xi^{2s}+b)}\,d\xi.
\end{equation}
Using the same cutoffs $0<\delta<\Lambda$ as above, we split the integral explicitly as
\begin{align*}
	\int_{a_{1}}^{a_{2}}R_{s}(u)\,du
	 & =\frac{1}{\pi}\int_{0}^{\delta}
	\frac{\sin(a_{2}\xi)-\sin(a_{1}\xi)}{\xi(\xi^{2s}+b)}\,d\xi
	\\
	 & \quad+\frac{1}{\pi}\int_{\delta}^{\Lambda}
	\frac{\sin(a_{2}\xi)-\sin(a_{1}\xi)}{\xi(\xi^{2s}+b)}\,d\xi
	\\
	 & \quad+\frac{1}{\pi}\int_{\Lambda}^{\infty}
	\frac{\sin(a_{2}\xi)-\sin(a_{1}\xi)}{\xi(\xi^{2s}+b)}\,d\xi.
\end{align*}
The integrand is analytic on $[\delta,\Lambda]$, so the middle integral is enclosed using \texttt{acb.integral}. At the lower region, the mean value theorem and the inequality $\xi^{2s}+b\ge b>1$ give
\begin{equation*}
	\left|\int_{0}^{\delta}
	\frac{\sin(a_{2}\xi)-\sin(a_{1}\xi)}{\xi(\xi^{2s}+b)}\,d\xi\right|
	\le (a_{2}-a_{1})\int_{0}^{\delta}\frac{d\xi}{\xi^{2s}+b}
	\le \delta(a_{2}-a_{1}).
\end{equation*}

It remains to bound the upper region. Suppose first that $0<a_{1}\le a_{2}$. Splitting the numerator into its two contributions gives
\begin{equation*}
	\int_{\Lambda}^{\infty}
	\frac{\sin(a_{2}\xi)-\sin(a_{1}\xi)}{\xi(\xi^{2s}+b)}\,d\xi
	=\int_{\Lambda}^{\infty}\frac{\sin(a_{2}\xi)}{\xi(\xi^{2s}+b)}\,d\xi
	-\int_{\Lambda}^{\infty}\frac{\sin(a_{1}\xi)}{\xi(\xi^{2s}+b)}\,d\xi.
\end{equation*}
Integration by parts gives, for $i=1,2$,
\begin{equation*}
	\int_{\Lambda}^{\infty}\frac{\sin(a_{i}\xi)}{\xi(\xi^{2s}+b)}\,d\xi
	=\frac{\cos(a_{i}\Lambda)}{a_{i}\Lambda(\Lambda^{2s}+b)}
	+\frac{1}{a_{i}}\int_{\Lambda}^{\infty}\cos(a_{i}\xi)
	\frac{d}{d\xi}\!\left(\frac{1}{\xi(\xi^{2s}+b)}\right)d\xi.
\end{equation*}
The function $[\xi(\xi^{2s}+b)]^{-1}$ is positive and decreasing to zero, so the absolute integral of its derivative over $[\Lambda,\infty)$ equals $[\Lambda(\Lambda^{2s}+b)]^{-1}$, which is at most $\Lambda^{-1-2s}$. Consequently,
\begin{equation*}
	\left|\int_{\Lambda}^{\infty}
	\frac{\sin(a_{i}\xi)}{\xi(\xi^{2s}+b)}\,d\xi\right|
	\le \frac{2}{a_{i}\Lambda^{1+2s}},
	\qquad i=1,2.
\end{equation*}
After including the factor $1/\pi$, the upper region contribution is thus bounded by
\begin{equation*}
	\frac{2}{\pi\Lambda^{1+2s}}
	\left(\frac{1}{a_{1}}+\frac{1}{a_{2}}\right).
\end{equation*}
If $a_{1}=0$, the implementation instead uses the uniform bound $|\sin(a_{2}\xi)-\sin(a_{1}\xi)|\le2$. Therefore,
\begin{equation*}
	\frac{1}{\pi}\left|\int_{\Lambda}^{\infty}
	\frac{\sin(a_{2}\xi)-\sin(a_{1}\xi)}{\xi(\xi^{2s}+b)}\,d\xi\right|
	\le\frac{2}{\pi}\int_{\Lambda}^{\infty}\xi^{-1-2s}\,d\xi
	=\frac{1}{\pi s\Lambda^{2s}}.
\end{equation*}
These bounds allow us to enclose $\int_{a_{1}}^{a_{2}}R_{s}(u)\,du$.

We can now verify the pointwise Schur inequality~\eqref{eq:Gprime-cert-BS} by a finite computation. Fix a grid value $s=s_{i}$, abbreviate $\Phi:=\Phiap$, $q(x):=(3(1-\Phi(x)^{2})_{+})^{1/2}$, and let $w$ be the explicit positive weight constructed below. The claim to be certified is
\begin{equation}\label{eq:app-schur-claim}
	\Theta_{s}(x)\;:=\;\frac{q(x)}{w(x)}\int_{0}^{\infty}
	K_{s}(x,y)\,q(y)\,w(y)\,dy\;\le\;\theta\qquad(x>0),
\end{equation}
for some $\theta<1$.

\subsubsection*{Design of the weight}

The only freedom in~\eqref{eq:app-schur-claim} is the weight~$w$, and the sharpness of the resulting bound depends on this choice. It is useful to introduce the integral operator
\begin{equation*}
	(Tf)(x)\;:=\;q(x)\int_{0}^{\infty}K_{s}(x,y)\,q(y)\,f(y)\,dy,
	\qquad T=q\,A^{-1}q,
\end{equation*}
in terms of which~\eqref{eq:app-schur-claim} reads $Tw\le\theta w$. We use the resulting ratio $Tw/w$ only to tune a one-parameter family of explicit weights.

For $a>0$, set
\begin{equation*}
	w_{a}(x):=\varepsilon+\kappa\,
	\frac{x/a}{\bigl(1+(x/a)^{2}\bigr)^{3}},
	\qquad \varepsilon:=\frac{1}{10},\qquad
	\kappa:=\frac{216}{25\sqrt{5}}.
\end{equation*}
The coefficient $\kappa$ normalizes the weight to have a maximum of $1+\varepsilon$. Indeed, if $b(z):=\kappa z/(1+z^{2})^{3}$, then
\begin{equation*}
	b'(z)=\kappa\frac{1-5z^{2}}{(1+z^{2})^{4}},
	\qquad b(1/\sqrt{5})=1.
\end{equation*}
Consequently $\varepsilon\le w_{a}\le1+\varepsilon$ on $[0,\infty)$, and $w_{a}$ is decreasing for $x\ge a/\sqrt{5}$. The constant $\varepsilon$ gives a uniform lower bound that will be needed in the tail estimate.

The scale $a$ is selected separately for every grid value $s=s_{i}$ by the following non-rigorous floating-point calculation. On the centered uniform grid of $N=2^{16}$ points in $[-L,L)$, with $L=1500$, the script evaluates $\Phiap$ from its $\sigma$-expansion and then forms the samples of~$q$. The action of $A^{-1}$ is computed by the fast Fourier transform since its Fourier multiplier is $ 1/(|\xi|^{2s}+11/10). $ To find the weight $w$ the code minimizes in $a$ the discrete Schur ratio
\begin{equation*}
	\rho(a):=\max_{x_j>0}
	\frac{(T w_a)(x_j)}{w_a(x_j)}.
\end{equation*}
This optimization is used only to find a suitable candidate. Figure~\ref{fig:schur-weight} shows the weights so obtained.

\begin{figure}[t]
	\centering
	\includegraphics[width=0.78\textwidth]{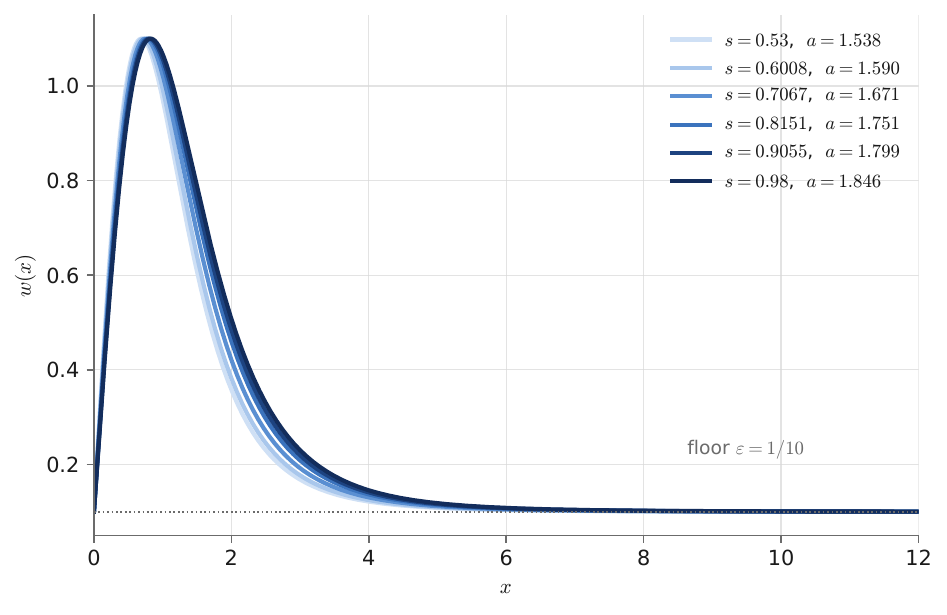}
	\caption{The weight $w_{a}$ on the positive half-line, at the values
		$a=a(s)$ selected by the minimization of $\rho(a)$, for six
		values of $s$.}
	\label{fig:schur-weight}
\end{figure}

We bound $\sup_{x>0}\Theta_{s}(x)$ by a partition of the $x$-axis into a compact part $(0,X]$ and a tail $[X,\infty)$.

				\subsubsection*{The compact part}

				Partition $(0,X]$ into cells $I_{j}=[jh,(j+1)h]$ and fix an $I_{j}$. We first need bounds on $q$ without making any monotonicity assumption. To do this, on each finite cell, an interval enclosure $[\Phi]_{k}$ is computed from the Taylor polynomial of degree $8$ at its midpoint and the explicit derivative bounds of Section~\ref{app:layer}. Thus, with
\begin{equation*}
	m_k:=\inf\{|z|:z\in[\Phi]_{k}\},
	\qquad
	\overline q_k:=\bigl(3(1-m_k^2)_{+}\bigr)^{1/2},
\end{equation*}
we have $\sup_{I_k}q\le\overline q_k$.

Since $q(x)\le\sup_{I_{j}}q$ and $w(x)\ge\inf_{I_{j}}w$ for $x\in I_{j}$, and the kernel integral is nonnegative, we may enclose the numerator and denominator of $\Theta_{s}$ separately over $I_{j}$:
\begin{equation*}
	\sup_{x\in I_{j}}\Theta_{s}(x)\;\le\;
	\frac{\sup_{I_{j}}q}{\inf_{I_{j}}w}\,N_{j},
	\qquad
	N_{j}\;:=\;\sup_{x\in I_{j}}\int_{0}^{\infty}K_{s}(x,y)\,q(y)\,w(y)\,dy .
\end{equation*}
It remains to bound the kernel integral $N_{j}$. We split the $y$-range into a near band of the $2K_{0}+1$ cells around $I_{j}$ and the complementary far region $F_{j}:=(0,\infty)\setminus\bigcup_{|k-j|\le K_{0}}I_{k}$:
\begin{equation*}
	\int_{0}^{\infty}K_{s}(x,y)\,q(y)\,w(y)\,dy
	\;=\;\sum_{|k-j|\le K_{0}}\int_{I_{k}}K_{s}(x,y)\,q(y)\,w(y)\,dy
	\;+\;\int_{F_{j}}K_{s}(x,y)\,q(y)\,w(y)\,dy ,
\end{equation*}
and estimate the two parts separately. Note that for $x\in I_{j}$ every $y\in F_{j}$ satisfies $|y-x|>K_{0}h$.

The estimates in the $y$-variable are continued beyond $X$ to a point $Y$, chosen so that the near band of every source row is covered. The remaining interval $[Y,\infty)$, which lies in the right-hand component of every $F_j$, is represented by a single terminal bound. To compute this bound, we make the change of variables $x=\tan\vartheta$, under which
\begin{equation*}
	\Phi(\tan\vartheta)
	=\frac{2}{\pi}\vartheta
	+\sum_{\alpha\in A}a_\alpha
	(\cos\vartheta)^\alpha\sin(\alpha\vartheta).
\end{equation*}
The compact interval $[\arctan Y,\pi/2]$ is divided into $10^4$ cells, and direct interval evaluation of this formula gives a bound $\overline q_\infty$ for $q$ on $[Y,\infty)$.

We start with the estimate for the first term. We first use that $K_{s}\ge0$. Indeed, as shown in Section~\ref{ss:inverse-ap}, $R_{s}$ is positive and strictly decreasing on $[0,\infty)$ and even, while $|x-y|<x+y$ for $x,y>0$. Thus
\begin{equation*}
	\int_{I_{k}}K_{s}(x,y)\,q(y)\,w(y)\,dy
	\;\le\;\Bigl(\sup_{I_{k}}q w\Bigr)\int_{I_{k}}K_{s}(x,y)\,dy ,
\end{equation*}
and enclose the remaining cell mass, uniformly for $x\in I_{j}$, by the mass bound
\begin{equation*}
	\int_{I_{k}}K_{s}(x,y)\,dy
	\;\le\;\int_{I_{k}}R_{s}(x-y)\,dy-\inf_{x\in I_{j}}\int_{I_{k}}R_{s}(x+y)\,dy.
\end{equation*}
Instead of computing all the quantities
\begin{equation*}
	\inf_{x\in I_{j}}\int_{I_{k}}R_{s}(x+y)\,dy \ge 0
\end{equation*}
for every $j, k$ in the partition, in the code implementation we fix a cutoff $L_{\mathrm{ref}} = 15$ and we set this quantity to zero when $j + k + 1\geq L_{\mathrm{ref}}/h$.

For the second term we use $K_{s}(x,y)\le R_{s}(x-y)$, so pulling the weight out at its far supremum leaves a single kernel mass beyond the band,
\begin{equation*}
	\int_{F_{j}}K_{s}(x,y)\,q(y)\,w(y)\,dy
	\;\le\;\Bigl(\sup_{F_{j}}q w\Bigr)\int_{F_{j}}R_{s}(x-y)\,dy
	\;\le\;\Bigl(\sup_{F_{j}}q w\Bigr)\int_{|u|>K_{0}h}R_{s}(u)\,du ,
\end{equation*}
where the last step used $F_{j}\subseteq\{|y-x|>K_{0}h\}$. The tail mass is controlled once, by
\begin{equation*}
	\int_{|u|>K_{0}h}R_{s}(u)\,du\;\le\;\frac{10}{11}-2\int_{0}^{K_{0}h}R_{s}(u)\,du ,
\end{equation*}
using $\int_{\R}R_{s}=10/11$ and the evenness of $R_{s}$. To evaluate the far weight, set $b_k:=\overline q_k\sup_{I_k}w$ and
\begin{equation*}
	P_k:=\max_{0\leq\ell\leq k}b_\ell,\qquad
	S_k:=\max_{\ell\geq k}b_\ell.
\end{equation*}
The finite list of bounds $b_k$ is completed by the bound $\overline q_\infty w(Y)$. The point $Y$ lies in the decreasing region of~$w$, so this bound is valid. The arrays $P_k$ and $S_k$ therefore bound $q w$ on $[0,(k+1)h]$ and $[kh,\infty)$, respectively. Since the two components of $F_j$ end at cell $j-K_0-1$ and begin at cell $j+K_0+1$, respectively, we obtain
\begin{equation*}
	\sup_{F_j}q w\leq
	\max\{P_{j-K_0-1},S_{j+K_0+1}\}=:B_j.
\end{equation*}
Here and below a term corresponding to an empty component is omitted.

In total, putting these estimates together, we have an upper bound $\theta_j$ for $\sup_{I_j} \Theta_s$ given by
\begin{align*}
	\theta_{j}:={} & \frac{\overline q_j}{\inf_{I_{j}}w}
	\Bigg\{
	\sum_{|k-j|\le K_{0}}b_k
	\left[
	\sup_{x\in I_{j}}\int_{I_{k}}R_{s}(x-y)\,dy
	-\mathbf 1_{j+k+1<\tfrac{15}{h}}
	\inf_{x\in I_{j}}\int_{I_{k}}R_{s}(x+y)\,dy
	\right] \notag                                       \\
	               & \hspace{4em}
	+B_j
	\left(\frac{10}{11}-2\int_{0}^{K_{0}h}R_{s}(u)\,du\right)
	\Bigg\}.
\end{align*}
Finally, the constant we use is denoted by
\begin{equation}\label{eq:app-core}
	\theta_{\mathrm{core}}:=\max_{j}\theta_{j},
	\qquad
	\sup_{0<x\le X}\Theta_{s}(x)\le\theta_{\mathrm{core}}.
\end{equation}

\subsubsection*{The tail}

For $r\in\{X/2,X\}$, let $Q_r$ be the maximum of the finite-cell bounds $\overline q_k$ over the cells meeting $[r,Y]$, together with the tail bound $\overline q_\infty$. Then
\begin{equation*}
	\sup_{x\ge r}q(x)\le Q_r.
\end{equation*}
For $x\ge X$ we bound the two factors of $\Theta_{s}(x)$ separately. Since $w\ge\varepsilon$, the first factor satisfies
\begin{equation}\label{eq:app-tail-prefactor}
	\frac{q(x)}{w(x)}\;\le\;\frac{Q_X}{\varepsilon}
	\qquad(x\ge X).
\end{equation}
Now we focus on the integral term. We split the $y$-integral at $y=X/2$,
\begin{equation}\label{eq:app-tail-split}
	\int_{0}^{\infty}K_{s}(x,y)\,q(y)\,w(y)\,dy
	\;=\;\int_{0}^{X/2}K_{s}(x,y)\,q(y)\,w(y)\,dy
	\;+\;\int_{X/2}^{\infty}K_{s}(x,y)\,q(y)\,w(y)\,dy,
\end{equation}
and bound each piece on its own scale.

On the far piece $y\ge X/2$ we use the full kernel mass. From $K_{s}(x,y)\le R_{s}(x-y)$ and $\int_{\R}R_{s}=10/11$,
\begin{equation*}
	\int_{X/2}^{\infty}K_{s}(x,y)\,dy\;\le\;\int_{\R}R_{s}\;=\;\frac{10}{11},
\end{equation*}
while $q(y)\le Q_{X/2}$ and $w\le\sup w$ give $q(y)w(y)\le Q_{X/2}\sup w$ there. Hence
\begin{equation}\label{eq:app-tail-far}
	\int_{X/2}^{\infty}K_{s}(x,y)\,q(y)\,w(y)\,dy
	\;\le\;Q_{X/2}\sup w\int_{X/2}^{\infty}K_{s}(x,y)\,dy
	\;\le\;\frac{10}{11}\,Q_{X/2}\sup w.
\end{equation}

On the near piece $y<X/2$, the constraint $x\ge X$ forces $x-y>X-X/2=X/2$, so the monotonicity of $R_{s}$ on $(0,\infty)$ yields
\begin{equation*}
	K_{s}(x,y)\;\le\;R_{s}(x-y)\;\le\;R_{s}(X/2).
\end{equation*}
Pulling this pointwise bound out of the integral gives
\begin{equation}\label{eq:app-tail-near}
	\int_{0}^{X/2}K_{s}(x,y)\,q(y)\,w(y)\,dy
	\;\le\;R_{s}(X/2)\int_{0}^{X/2}q(y)\,w(y)\,dy.
\end{equation}

Substituting \eqref{eq:app-tail-far} and~\eqref{eq:app-tail-near} into~\eqref{eq:app-tail-split}, and multiplying by the prefactor bound~\eqref{eq:app-tail-prefactor}, we obtain
\begin{equation}\label{eq:app-tail}
	\begin{aligned}
		\sup_{x\ge X}\Theta_{s}(x)
		 & \;\le\;\frac{Q_X}{\varepsilon}\Biggl(
		\frac{10}{11}\,Q_{X/2}\sup w
		+R_{s}(X/2)\int_{0}^{X/2}q(y)\,w(y)\,dy\Biggr)
		=:\theta_{\mathrm{tail}} .
	\end{aligned}
\end{equation}

At each grid point~$s_i$, the finite Schur certificate reduces to the two interval inequalities
\begin{equation*}
	\theta_{\mathrm{core}}<1,
	\qquad
	\theta_{\mathrm{tail}}<1.
\end{equation*}
Hence, with $\theta:=\max(\theta_{\mathrm{core}},\theta_{\mathrm{tail}})$, those two bounds give $\sup_{x>0}\Theta_s(x)\le\theta<1$, which proves the Schur inequality~\eqref{eq:Gprime-cert-BS}.

\subsection{Exact evaluation of the derivative functional}\label{app:Jeval}

Criteria \textup{(C3)} and \textup{(C4)} require the derivative functional $J(s,\Phiap)$ of~\eqref{eq:J-functional}, the weighted seminorm $\|\Phiap\|_{\star}(s)$ of~\eqref{eq:starnorm}, the Sobolev constant $C_{\star}(s)$ of~\eqref{eq:Cstar-LJ-def}, and the transfer constant $L_{J}(s)$ of~\eqref{eq:LJ-const-def}. For the approximate layer these are computed as exact modal sums, with no quadrature.

Differentiating~\eqref{eq:app-phiap} and using the linearity of the Fourier transform gives
\begin{equation*}
	\widehat{(\Phiap)'}(\xi)
	=\widehat{g'}(\xi)+\sum_{\alpha\in A}a_{\alpha}
	\widehat{\sigma_{\alpha}'}(\xi).
\end{equation*}
Recall that Lemma~\ref{lem:app-sigma-ft} gives $\widehat{g'}(\xi)=2e^{-|\xi|}$ and $\widehat{\sigma_{\alpha}'}(\xi)=(\pi/\Gamma(\alpha))|\xi|^{\alpha}e^{-|\xi|}$; therefore,
\begin{equation}\label{eq:app-phiap-ft}
	\widehat{(\Phiap)'}(\xi)
	=e^{-|\xi|}\left(2+\pi\sum_{\alpha\in A}
	\frac{a_{\alpha}}{\Gamma(\alpha)}|\xi|^{\alpha}\right)
	=e^{-|\xi|}P(|\xi|),
	\qquad
	P(t):=2+\pi\sum_{\alpha\in A}\frac{a_{\alpha}}{\Gamma(\alpha)}t^{\alpha}.
\end{equation}
Since $P$ is a finite sum, its square can be collected exactly as $P(t)^{2}=\sum_{\mu}c_{\mu}t^{\mu}. $ Substituting this expansion and \eqref{eq:app-phiap-ft} into the definition of $J$, and then using evenness, gives
\begin{equation*}
	J(s,\Phiap)
	=\frac{1}{2\pi}\int_{\R}\log|\xi|\,|\xi|^{2s-2}e^{-2|\xi|}
	\sum_{\mu}c_{\mu}|\xi|^{\mu}\,d\xi
	=\frac{1}{\pi}\sum_{\mu}c_{\mu}
	\int_{0}^{\infty}\log\xi\;\xi^{2s-2+\mu}e^{-2\xi}\,d\xi.
\end{equation*}
Set $\nu_{\mu}:=2s-1+\mu>0$. The integral in the last line can be evaluated by differentiating the Gamma integral with respect to $\nu > 0$,
\begin{equation*}
	L(\nu)
	:=\int_{0}^{\infty}\log\xi\;\xi^{\nu-1}e^{-2\xi}\,d\xi
	=\frac{d}{d\nu}\int_{0}^{\infty}\xi^{\nu-1}e^{-2\xi}\,d\xi
	=\frac{d}{d\nu}\bigl(2^{-\nu}\Gamma(\nu)\bigr)
	=2^{-\nu}\Gamma(\nu)\bigl(\psi(\nu)-\log2\bigr),
\end{equation*}
where $\psi(\nu) = \Gamma'(\nu)/\Gamma(\nu)$ denotes the digamma function. It follows that
\begin{equation}\label{eq:app-J-modal}
	J(s,\Phiap)=\frac{1}{\pi}\sum_{\mu}c_{\mu}L(\nu_{\mu}).
\end{equation}

The same substitution in the definition of the weighted seminorm yields
\begin{equation*}
	\|\Phiap\|_{\star}^{2}(s)
	=\frac{1}{\pi}\sum_{\mu}c_{\mu}
	\int_{0}^{\infty}|\log\xi|\;\xi^{\nu_{\mu}-1}e^{-2\xi}\,d\xi.
\end{equation*}
To evaluate this integral, use $|\log\xi|=\log\xi-2\log\xi\,\mathbf 1_{\{\xi<1\}}$ and expand the exponential on $[0,1]$. This gives
\begin{align*}
	I(\nu)
	 & :=\int_{0}^{\infty}|\log\xi|\;\xi^{\nu-1}e^{-2\xi}\,d\xi
	=L(\nu)-2\int_{0}^{1}\log\xi\;\xi^{\nu-1}e^{-2\xi}\,d\xi \notag \\
	 & =L(\nu)-2\sum_{k\ge0}\frac{(-2)^{k}}{k!}
	\int_{0}^{1}\xi^{\nu+k-1}\log\xi\,d\xi \notag                   \\
	 & =L(\nu)+2\sum_{k\ge0}\frac{(-2)^{k}}{k!\,(\nu+k)^{2}},
	\qquad \nu>0,
\end{align*}
where in the last step we have used integration by parts. Consequently,
\begin{equation*}
	\|\Phiap\|_{\star}^{2}(s)
	=\frac{1}{\pi}\sum_{\mu}c_{\mu}I(\nu_{\mu}).
\end{equation*}
Both $\psi$ (digamma) and $\Gamma$ are evaluated as enclosures, and the finite sum~\eqref{eq:app-J-modal} certifies that $J(s,\Phiap)<0$.

The constant $C_{\star}(s)$ of~\eqref{eq:Cstar-LJ-def} satisfies the bound
\begin{equation*}
	C_{\star}(s)\le \frac{1}{\sqrt{2se}}.
\end{equation*}
Indeed, the expression under the supremum in~\eqref{eq:Cstar-LJ-def} is invariant under $t\mapsto t^{-1}$, so it is enough to consider $t\ge1$. On that half-line, $(1+t)^{2s}\ge t^{2s}$, and hence
\begin{equation*}
	\frac{|\log t|\,t^{s}}{2(1+t)^{2s}}
	\le \frac{\log t}{2t^{s}}
	\le \frac{1}{2se},
\end{equation*}
because $t\mapsto(\log t)t^{-s}$ on $[1,\infty)$ has maximum $1/(se)$ at $t=e^{1/s}$. The certificate substitutes this explicit upper bound for every occurrence of $C_{\star}(s)$.

\subsection{Organization of the code}\label{app:code}

The accompanying code consists of ten Python modules. Four of them, \texttt{common.py}, \texttt{layer\_basis.py}, \texttt{layer\_approximation.py} and \texttt{green\_kernel.py}, are libraries used by the others and are not executed directly.

The remaining six modules are executable and are run in the order listed below. Each writes its output to a file that its successors read.

\begin{enumerate}
	\item \texttt{schur\_design.py} selects, in floating-point arithmetic, the
	      scale parameter of the positive weight $w$ of~\textup{(C1)}.
	\item \texttt{schur\_certificate.py} verifies~\textup{(C1)}, evaluating the two
	      Schur bounds of Section~\ref{app:schur} and the bound
	      $\|V^{\mathrm{ap}}_{s_{i}}\|_{L^{\infty}(\R)}\le4$ at each grid point.
	\item \texttt{contraction\_certificate.py} verifies~\textup{(C2)}. It builds
	      $\Phiap$, encloses the residual $\delta(s_{i})$, and forms the quantities
	      $\alpha(s_{i})$ and $\tau(s_{i})$ of the fixed-point argument of
	      Proposition~\ref{prop:layer-radius}.
	\item \texttt{j\_eval.py} verifies~\textup{(C3)}, evaluating
	      $J(s_{i},\Phiap)$ by the formula~\eqref{eq:app-J-modal}, together with
	      the seminorm $\|\cdot\|_{\star}$ and the resulting margin $\eta_{i}$.
	\item \texttt{semiconc\_margin.py} verifies~\textup{(C4)}, evaluating the
	      curvature constant $M^{\ast}_{\mathrm{sc}}([a_{i},b_{i}])$ and the
	      margin for each of the $55$ cells.
	\item \texttt{verify\_all.py} collects the certified constants of the four
	      preceding conditions into a single table.
\end{enumerate}

A complete run of the six steps took $2$ h $8$ min on an Intel Core i7-12700H with $14$ cores and $16$ GB of memory.

\bibliographystyle{plain}
\bibliography{refs.bib}

\end{document}